\documentclass[11pt,reqno]{amsart}

\usepackage[latin1]{inputenc}
\usepackage{amsmath,amsthm,amssymb,amscd,mathrsfs,mathtools,color,soul}
\usepackage[shortlabels]{enumitem}
\definecolor{darkgreen}{rgb}{0.0, 0.6, 0.13}
\usepackage[pdftex, colorlinks=true, pdfstartview=FitV, linkcolor=blue, citecolor=red, urlcolor=blue]{hyperref}

\usepackage[margin=1.07in]{geometry}
\newtheorem{thm}{Theorem}[section]
 
 \newtheorem{lem}[thm]{Lemma}
 \newtheorem{prop}[thm]{Proposition}

 \theoremstyle{definition}
 \newtheorem{df}[thm]{Definition}
 \theoremstyle{remark}
 \newtheorem{rem}[thm]{Remark}
 
 \numberwithin{equation}{section}

\newcommand{\eps}{\epsilon}
\newcommand{\Rb}{\mathbb R}
\newcommand{\Tb}{\mathbb T}
\newcommand{\Eb}{\mathbb E}
\newcommand{\Hb}{\mathbb H}
\newcommand{\Gc}{\mathcal G}

\newcommand{\Mb}{\mathbb M}

\newcommand{\Zb}{\mathbb Z}
\newcommand{\Jc}{\mathcal J}

\newcommand{\Af}{\mathfrak A}
\newcommand{\Bc}{\mathcal B}
\newcommand{\Nc}{\mathcal N}
\newcommand{\Hc}{\mathcal H}
\newcommand{\Wc}{\mathcal W}
\renewcommand{\Mc}{\mathcal M}
\newcommand{\Rc}{\mathcal R}
\newcommand{\Uc}{\mathcal U}
\newcommand{\Pc}{\mathcal P}
\newcommand{\Ac}{\mathcal A}
\newcommand{\Lc}{\mathcal L}
\newcommand{\Dc}{\mathcal D}
\newcommand{\Ec}{\mathcal E}
\newcommand{\Qc}{\mathcal Q}
\newcommand{\Sc}{\mathcal S}
\newcommand{\Vc}{\mathcal V}
\newcommand{\Ic}{\mathcal I}
\newcommand{\Kc}{\mathcal K}
\newcommand{\Tc}{\mathcal T}

\newcommand{\Xc}{\mathcal X}
\newcommand{\hf}{\mathfrak h}
\newcommand{\nf}{\mathfrak n}
\newcommand{\mf}{\mathfrak m}
\newcommand{\pf}{\mathfrak p}

\newcommand{\ef}{\mathfrak e}
\newcommand{\ff}{\mathfrak f}
\newcommand{\lf}{\mathfrak l}
\newcommand{\qf}{\mathfrak q}
\newcommand{\Cf}{\mathfrak C}
\newcommand{\Yc}{\mathcal Y}
\newcommand{\rf}{\mathfrak r}

\newcommand{\Vs}{\mathscr V}

\def\T{\mathbb{T}}
\def\Z{\mathbb{Z}}

\definecolor{bluegreen}{rgb}{0.0, 0.3, 0.9}

\begin{document}

\author{Yu Deng}
\address{University of Chicago}
\email{yudeng@uchicago.edu}

\author{Alexandru D. Ionescu}
\address{Princeton University}
\email{aionescu@math.princeton.edu}

\author{Fabio Pusateri}
\address{University of Toronto}
\email{fabiop@math.toronto.edu}

\thanks{
\noindent
Y. D. was supported in part by a Simons Collaboration Grant on Wave
Turbulence, NSF grant DMS-2246908, and a Sloan Fellowship. A. D. I.
was supported in part by NSF-FRG grant DMS-2245228, by NSF-UEFISCDI
grant DMS-2407694, and a Simons Collaboration Grant on Wave Turbulence.
F. P. was supported in part by a start-up grant from the University of
Toronto,
and NSERC grants RGPIN-2018-06487 and RGPIN-2025-06419.}

\title[Wave Turbulence for 2D gravity waves: propagation of randomness]{On
the wave turbulence theory of 2D gravity waves, II: \\ propagation of randomness}

\begin{abstract}
\normalsize
This is a sequel to our work \cite{DIP}, initiating the rigorous study of wave turbulence 
in the context of the water waves equations. 
We combine energy estimates, normal forms, and probabilistic and combinatorial arguments 
to complete the construction of long-time
solutions with random initial data for the $2$d ($1$d interface) gravity water waves system 
on large tori. 

More precisely, we consider solutions to the Cauchy problem with spatial period $R\gg 1$, 
which have randomly distributed
phases at the initial time, with an $L^\infty$ norm of size $\approx\epsilon$ and large energy 
of size $\approx\epsilon^2 R$.
Under the assumption $\epsilon\in[R^{-8/3+},R^{0-}]$ we show that randomness propagates on a time interval $[0,T_1]$ with $T_1 = O(\epsilon^{-8/3+})$;
in particular, solutions satisfy similar $L^2$ and $L^\infty$ type bounds 
as those satisfied at the initial time on this long time interval.

This is the first long-time regularity result for solutions of water waves systems with large energy (but small local energy), which is the correct setup for applications to wave turbulence. Such a result is only possible in the presence of randomness. The proof requires several novel ideas, including a robust framework involving Feynman trees, compatible with normal forms, a new time and frequency dependent renormalization procedure, and new combinatorial arguments that are specific to one dimensional problems.
\end{abstract}

\maketitle

\setcounter{tocdepth}{1}

\begin{quote}
\tableofcontents
\end{quote}

\medskip

\section{Introduction} 

Our work in this paper is motivated by the outstanding problem of understanding 
rigorously the wave turbulence theory of water waves.
This is a classical problem in Mathematical Physics, going back to work of Hasselmann \cite{Has1, Has2}, and a key part of Hasselman's {\it{theory of climate variability}}. We refer the reader to the recent book \cite{Naz} 
for a more complete description of the problem and many examples.

The wave turbulence problem has received intense attention in recent years, mainly in the context of simpler 
dispersive models such as semilinear Schr\"{o}dinger equations;
see for example \cite{BGHS1, BGHS2, ColGer1, ColGer2, JinMa, DeHa1, DeHa2, DeHa3, DeHa4, StafTran}. 
The focus in these papers has been the rigorous derivation of the {\it{wave kinetic equation}} (WKE).
At the analytical level, this amounts 
to understanding precisely the dynamics of solutions corresponding to randomized initial data, 
in the large box and weakly nonlinear limit, and on long time intervals.

Our goal is to initiate a similar investigation for irrotational water waves models. 
There is a fundamental difference between water waves, which are quasilinear evolutions, 
and semilinear models: the solution cannot be constructed using the Duhamel formula,
due to unavoidable loss of derivatives. In particular, the solution cannot be described using 
iterations of the Duhamel formula (Feynman trees), which is the central idea of the analysis of semilinear models. 

To address this fundamental difficulty we propose a new mechanism, based on a combination of two main ingredients:

(1) {\it{Energy estimates:}} In our first paper \cite{DIP} we prove energy estimates. These are {\it{deterministic}} estimates, in the sense that they hold for all solutions. The key point is to show that the energy increment is controlled by an $L^\infty$-based 
norm of the solution, which can then be exploited effectively in the analysis of solutions with random initial data.

(2) {\it{Propagation of randomness and pointwise bounds:}} 
The second step consists in proving pro\-pa\-gation of randomness 
(in a suitable form) then deriving optimal pointwise bounds on randomized solutions in a large box. 
For this one can use the iterated Duhamel formula; the point is that losses 
of derivatives are allowed at this stage, thanks to the accompanying estimate on the high order energy.

In this paper we complete the second part of this program for the problem of pure
gravity water waves in two dimensions (1d interface).

\subsection{The water waves equations}
We consider the two-dimensional irrotational gravity water wave problem 
on a large torus of size $R$, written in the form
\begin{equation}\label{ww0}
\left\{
\begin{aligned}\partial_th&=G(h)\phi,
\\
\partial_t\phi&=-h-\frac{1}{2}(\partial_x\phi)^2+\frac{(G(h)\phi+\partial_xh\cdot\partial_x\phi)^2}{2(1+(\partial_xh)^2)},
\end{aligned}
\right.
\end{equation} 
where $G(h)=\sqrt{1+(\partial_xh)^2}\cdot\Nc(h)$ and $\Nc(h)$ is the Dirichlet 
to Neumann operator associated with the (periodic) domain below the graph of a function $h$,
$\Omega_t := \{(x,y) \,: \, y<h(x), \, x \in \mathbb{T}_R\}$ where $\mathbb{T}_R := \mathbb{R}/(2\pi R)$.
We denote the complex physical unknown 
\begin{equation}\label{complex}u:=h+i|\partial_x|^{1/2}\phi,
\end{equation}
where $\phi$ is the trace on the interface $\partial \Omega_t$ of the fluid's velocity potential in $\Omega_t$.

In this work, which is a sequel to our work \cite{DIP}, we prove the long-time existence of
random solutions for the full equations \eqref{ww0} for the Cauchy problem with random data.

\subsection{The random data problem and main result}  
We now set up the random initial data problem for the equation \eqref{ww0}. 
Let $\{g_k(\omega)\}$ be i.i.d. normalized complex Gaussian random variables, 
and $\psi$ be a fixed, rapidly decaying Lipschitz function on $\Rb$. 
We set the random initial data $u(0)=u_{\mathrm{in}}$, for the unknown $u$ defined in \eqref{complex}, as
\begin{equation}\label{data_u}
u_{in}(x) = \frac{1}{2\pi R}\sum_{k\in\Zb_R\setminus \{0\}}\epsilon R^{1/2} \sqrt{\psi(k)}g_k(\omega)e^{ikx}
\end{equation} 
where $\Zb_R=R^{-1}\Zb$. Without loss of generality we may assume that $\int_{\Tb_R} h = 0$, so that our solution has zero mean, consistent with the assumption that the sum in \eqref{data_u} is taken over $k\neq 0$;
this condition is preserved by the solution flow of \eqref{ww0}. 
We now state our main result as follows.

\begin{thm}\label{main} 
Assume that $0<\theta_0\leq 1/10$, $N_0$ is an integer $\geq (100/\theta_0)^{3}$, and $\eta_0:=(\theta_0/100)^3$. Assume that $R\gg 1$ is sufficiently large (depending on $\theta_0,N_0$), $\epsilon\in[R^{-8/3+\theta_0},R^{-\theta_0}]$, and set $T_1:=\epsilon^{-8/3+\theta_0}$. Finally, assume that $\psi\in C^1(\Rb)$ satisfies the bounds
\begin{equation}\label{databoun}
|\psi(x)|+|\psi'(x)|\leq (1+|x|)^{-N_0-10},\qquad\text{ for any }x\in\Rb.
\end{equation}
Then, with probability $\geq 1-e^{-R^{\eta_0}}$, the equation (\ref{ww0}) with initial data (\ref{data_u}) has a smooth solution $u=h+i|\partial_x|^{1/2}\phi\in C([0,T_1]:H^{N_0})$
on the time interval $[0,T_1]$, which satisfies the estimates
\begin{align}\label{mainconc}
\|u(t)\|_{H^{N_0}} \leq \epsilon^{1-\theta_0}R^{1/2}, \qquad \sum_{\alpha\leq 6}\|D^\alpha_xu(t)\|_{L^\infty} \leq\epsilon^{1-\theta_0} , \qquad \text{ for any }t\in[0,T_1].
\end{align}
\end{thm}

\subsection{Remarks}\label{ssecdisc1} We make some comments on the statement of Theorem \ref{main}.

\smallskip

(1) {\it The smallness assumption.} We emphasize that our solutions are assumed to be small in $L^\infty$-type norms (or, equivalently, in local energy norms per unit area), but have very large total energy. This is the proper framework for understanding wave turbulence and for physically relevant applications connected to the Hasselman theory, since this is the natural appearance of oceans: locally small perturbations of the flat solution, but which have large total energy due to the large size of the oceans.

This is in sharp contrast with all the previous (deterministic) extended lifespan results for solutions of the water waves system, such as those in \cite{ADa, ADb, BertiDelort, BeFeFr, BeFePu, DIPP, GMS2, HIT, IT, IT2, IoPu2, IoPu4, IoPu5, Wa1, WuAG, Wu3DWW, WuNew, ZhengWW, Zheng22} and others, in which the total energy of the solution at the initial time (and, in fact, high order energies involving high Sobolev norms) was assumed to be small.
\smallskip

(2) {\it Propagation of randomness and pointwise control}. Given that our solutions have large total energy, long time existence is only possible due of randomness. The main point of the theorem is therefore to prove that the solutions retain
 some form of randomness on long time intervals, which, in particular, leads to uniform pointwise control as stated in \eqref{mainconc}. 

This requires an elaborate scheme, starting from the Duhamel formula and further expanding the solutions in terms of the initial data using Feynman trees. A key difficulty is that our water waves system is quasilinear, so the standard scheme based on the Duhamel formula loses derivatives. We address this by truncating the solution in frequency, and then controlling the original solution $u$ in high Sobolev norms (using a deterministic energy argument) and the truncated solution $u_{\mathrm{tr}}$ as well as the difference $u-u_{\mathrm{tr}}$ in $L^\infty$-type norms (using a probabilistic argument based on Feynman expansions).
\smallskip

(3) {\it Lifespan of the solution.} Wave turbulence can only be observed if we control the evolution for a long time, and only in the presence of randomness. At the analytical level, the basic question then becomes for how long we can control the solution given random initial data of $L^\infty$ norm $\|u_{in}\|_{L^\infty}\approx \epsilon\ll 1$ and energy $\mathcal{E}\approx \epsilon^2 R\gg 1$ (as in \eqref{data_u}). 

Ideally, in a cubic problem one would like to control the solution for a time $T_\epsilon\approx\eps^{-4}$, which is the {\it kinetic time} at which one could hope to observe wave turbulence and derive the wave kinetic equation. This is precisely the conclusion in the case of the cubic NLS in dimensions $\geq 2$, as proved in a sequence of
remarkable papers \cite{BGHS1, BGHS2, DeHa1, DeHa2, DeHa3}, but the problem remains wide open for more complex (especially quasilinear) evolutions such as the water waves systems.

In the case of water waves we are not able to reach the kinetic time due to multiple issues, as described 
in Subsection \ref{MainIdeas} below. A more realistic first step is to establish control of the solution for as long as possible, both at the energy level and in terms of propagation of randomness, with the hope that this is a key step towards understanding a full wave turbulence theory. Our main claim in this paper is that in the case of irrotational gravity water waves in 2D we can have such control for a time $T_1\approx\epsilon^{-8/3+}$. 
\smallskip

(4) {\it Parameters.} For the rest of this paper, we fix several parameters. As in Theorem \ref{main} assume that $\theta_0\in (0,1/10]$ is given, and set $\theta:=\theta_0/100$. We also fix $N_0\geq \theta^{-3}$ as the highest Sobolev regularity we will propagate on our solution, and $\eta_0:=\theta^3$ sufficiently small. Later on, in Section \ref{SecRan}, we will fix two more constants $A=\lfloor\theta^{-1}\rfloor$ and $N=\lfloor\theta^{-2}\rfloor$.

The implicit constants in $\lesssim$ may depend on the parameters $\theta_0,N_0$, while $R$ is assumed to be sufficiently large relative to these parameters, and $\epsilon\in[R^{-8/3+\theta_0},R^{-\theta_0}]$.

\subsection{Main ideas of the proof}\label{MainIdeas} We provide a complete outline of the proof in Section \ref{SecStra}. 
The basic idea is to adapt and extend the analysis developed recently by Deng-Hani \cite{DeHa1, DeHa2, DeHa3} in the case of the cubic NLS to our more complicated water waves system. There are several significant differences, which lead to substantial new difficulties\footnote{These issues should be thought of as generic, and are likely to appear in the study of the wave turbulence theory of many dispersive evolution equations, both semilinear and quasilinear.}, such as:
\smallskip

$\bullet\,\,\,$ Truncation of the solution in frequency. The water waves system is a quasilinear system, so the solution cannot be described using just the Duhamel formula due to loss of derivatives. However, in order to be able to use the main randomness assumption on the data, we need to express a substantial part of the solution as, essentially, an algebraic expression in terms of the initial data.

Our approach is to first truncate the solution at a certain frequency level (depending on $R$ and $\theta_0$). We can then control the high frequencies using the deterministic energy estimates proved in our earlier work \cite{DIP}, and the low frequencies using Feynman expansions (derived from the Duhamel formula) and the randomness of the initial data. 
\smallskip

$\bullet\,\,\,$ {\it Normal form variable.} The original water waves system \eqref{ww0} is quadratic in the variable $u$. However, in order to control the solution for a longer time, we need to perform a normal form and write the system in terms of a new variable $w$. This is possible, of course, since the system does not contain quadratic resonances, but the issue is that our randomness assumption is in terms of the original variable $u$ (as in \eqref{data_u}), not in terms of the modified normal form variable $w$. 

To address this issue we develop a new and more elaborate notion of trees, in the context of Feynman expansions. Our trees naturally contain two types of nodes: the interaction nodes, corresponding to the nonlinear interactions, and the normal form nodes which correspond to the transition to the normal form variable, and which appear right before the leaves. 

In addition to a first (quadratic) normal form, it turns out that we also need 
to perform a second (cubic) partial normal form which eliminates non-resonant contributions.
For the remaining, potentially resonant, interactions we prove a $C^{1/2}$ smoothness property for the symbol.
Such a regularity is then needed in the combinatorial part of the argument to prove a necessary
cancellation between``irregular chains" (first introduced in \cite{DeHa1} for the cubic NLS, where the symbol
is the constant $1$).


\smallskip  

$\bullet\,\,\,$ {\it Time and frequency dependent renormalization.} 
Renormalization is a key step in order to remove certain nonlinear interactions that do not satisfy good estimates and identify a ``profile" of the solution that can be controlled for a long time. In the case of the cubic NLS only a simple renormalization procedure is needed, namely multiplication by a linear oscillatory factor that depends only on the (conserved) mass of the solution.

In our case, such a simple renormalization is not possible. We employ instead a two-step renormalization process. The first step is explicit and involves weakening the contribution of the so-called ``trivial cubic resonances", which represent the largest nonlinear contribution in the problem. The second step is more subtle and involves an implicit renormalization factor that depends on both time and frequency, and which weakens higher order resonant contributions. 

It is important to notice that both of these renormalization steps are {\it{deterministic}}: 
they depend only on the profile $\psi$ in \eqref{data_u}, not on the full random initial data $u_{in}$.
\smallskip

$\bullet\,\,$ {\it Unfavorable counting in dimension $1$ evolutions.}
One dimensional problems appear naturally in fluids and other physical situations, but have not been treated
in the series of papers on the cubic NLS cited above. Only the recent work \cite{Va} dealt with a $1$d MMT model.
In this work we notice some tree structures\footnote{This could be named ``broken chains" for their
resemblance to the ``irregular chains" of Deng-Hani \cite{DeHa1}.}
whose size diverges for times sightly
larger than our existence time $T_1$. See Appendix \ref{SecDiv} for details. 
This divergence is due to some unfavorable counting in one dimension, where a small ball
of size $\lambda$ only contributes linearly, while it contributes $\lambda^d$ in dimension $d$; 
see \eqref{eq.lower_bd} and \eqref{eq.lower_bd_1}.
Moreover, unlike the case of ``irregular chains" \cite{DeHa1} there do not seem to be apparent
cancellations\footnote{It is conceivable that (approximate) integrability would play a role here, 
but we leave this problem open for further investigation.} within these structures.

\subsection{Further results}\label{further} The ideas developed here can be adapted to other models, to prove long-time regularity results for data of small local energy, in the presence of randomness. There are two main factors that ultimately determine the length of time on which one can establish control: (1) the combinatorics of the problem, in particular the linear dispersion relation, the structure of resonances, and the dimension of the underlying space, and (2) the range of validity of a high order energy estimate. 

For example, in the case of irrotational water waves in 3D (a system similar to \eqref{ww0} on a 2D large torus $\Tb_R^2$), with random initial data similar to \eqref{data_u}, our approach would give control the solution on the time interval $[0,T_1]$, where 
\begin{equation}\label{MMT0}
T_1\approx\epsilon^{-2}.
\end{equation}
The limitation here is mainly due to the energy estimate: since there are nontrivial cubic resonances we can only control high order energy norms of the solution on a time of length $\approx \epsilon^{-2}$, where $\epsilon$ measures $L^\infty$-type norms of the solution. Randomness and pointwise control can then be propagated on the time interval $[0,T_1]$ by adapting the construction in this paper.

Another widely used wave turbulence model is the MMT equation
\begin{equation}\label{MMT1}
\partial_t\psi+i|\partial_x|^\alpha\psi=\pm i |\partial_x|^{\beta/4}\big(|\partial_x|^{\beta/4}\psi\cdot \big||\partial_x|^{\beta/4}\psi\big|^2\big)
\end{equation}
in one or two dimensions $x\in\T_R$ or $x\in\Tb_R^2$, for various parameters $\alpha\in(0,2]\setminus \{1\}$ and $\beta\in\Rb$. The model has been introduced by Majda--McLaughlin--Tabak \cite{MMT}, and has been studied extensively both at the numerical and at the theoretical level. 

In the special (semilinear) case $\beta=0$ our analysis in this paper leads to a time of existence of the solution $T_1\approx\epsilon^{-8/3+}$, improving on the recent result of Vassilev \cite{Va}. Energy estimates are not needed in this case, since there is no derivative loss in the Duhamel formula, and one can proceed like in this paper with several simplifications (no truncation step is needed, one can use only standard ternary Feynman trees, without normal form branching nodes, and the renormalization involves only the $L^2$ norm of the solution). 

On the other hand, for larger values $\beta>0$ one would likely need to also prove energy estimates to keep control of some high order energy of the solution, which would lead to a shorter lifespan of the solution with $T_1\approx\epsilon^{-2}$ as in \eqref{MMT0}.

\subsection{Organization} 

The rest of the paper is organized as follows: 
in Section 2 we collect all the ingredients needed for the proof of Theorem \ref{main} and reduce matters to proving two propositions: Proposition \ref{normalprop} concerning normal forms and long expansions of the nonlinearity, and Proposition \ref{apriori} concerning propagation of randomness and pointwise control of the solution. 

In Section \ref{SecRan} we introduce the main ideas of the probabilistic argument, including truncation, the construction of approximate solutions using suitable Feynman trees expansions, 
and the time and frequency dependent renormalization procedure. 
The point is to reduce the proof of  Proposition \ref{apriori} to Propositions \ref{propjr} and \ref{propdiff}. 
The last part of Section \ref{SecRan} contains some results on counting bounds that will be used throughout the paper.

Section \ref{sec.irre} treats the irregular chains. Here,  to obtain the main cancellation as in \cite{DeHa1} we 
need to exploit the $C^{1/2}$ smoothness of the symbols of the transformed water waves equations 
proved in Proposition \ref{normalprop}.
Section \ref{sec.molecule} contains the heart of the combinatorial analysis and the estimates for molecules,
eventually leading to the main bound in Proposition \ref{propcong_2}.
Section \ref{sec.rem} is dedicated to the proof of Proposition \ref{propdiff} which
estimates the difference between the profile of the approximate solution constructed by Feynman trees
expansions, and the true (renormalized) profile.

Section \ref{sectionNormal} contains the proof of the normal form Proposition \ref{normalprop}.
Finally, Section \ref{SecDiv} shows an example Feynman diagrams whose size diverges 
in the limit $\epsilon\rightarrow 0$ for times larger than $\epsilon^{-8/3-}$, showing the almost optimality 
of our current arguments.

\medskip
\section{Preliminaries and strategy of the proof}\label{SecStra}
In this section we describe the general strategy for the proof of Theorem \ref{main},
which is composed of several steps. 
We state all the main propositions that lead to the proof of the main theorem, and prove some of the more immediate supporting results.

After introducing some notation,
we begin our proof with a normal form transformation that eliminates all quadratic terms (in $u$)
in the equations, as well as the non-resonant cubic terms.
This is the main content of Proposition \ref{normalprop}
together with an homogeneity expansion for the water waves equations up to arbitrary order $A$, 
with suitable control on the remainders. The proofs of these results are given in Section \ref{sectionNormal}.
Then we show how the initial data for the transformed variable essentially preserves 
the initial randomness \eqref{data_u} and recall the main energy increment Proposition from \cite{DIP}.
Subsection \ref{SsecPR} contains the main probabilistic result of the paper with the 
``propagation of ramdomness'' stated in Proposition \ref{apriori}. The strategy for this proposition
is explained in more details in Section \ref{SecRan}.
Using Proposition \ref{apriori} 
we perform our main bootstrap argument in Proposition \ref{mainboot}, thereby obtaining the main theorem.

\subsection{Some notation}\label{SecNot}
We introduce first some notation. For functions $f\in L^2(\mathbb{T}_R)$ we define the Fourier transform
\begin{equation}\label{FT1}
\widehat{f}(\xi)=\mathcal{F}(f)(\xi):=\int_{\mathbb{T}_R}f(x)e^{-i\xi x}\,dx,\qquad \xi\in\mathbb{Z}/R:=\{n/R:\,n\in\mathbb{Z}\}.
\end{equation}
The definition shows easily that
\begin{equation}\label{FT2}
\begin{split}
&f(x)=\frac{1}{2\pi R}\sum_{\xi\in\mathbb{Z}/R}\widehat{f}(\xi)e^{i\xi x},\\
&\int_{\mathbb{T}_R}f(x)\overline{g(x)}\,dx=\frac{1}{2\pi R}\sum_{\xi\in\mathbb{Z}/R}\widehat{f}(\xi)\overline{\widehat{g}(\xi)},\\
&\mathcal{F}(fg)(\xi)=\frac{1}{2\pi R}\sum_{\eta\in\mathbb{Z}/R}\widehat{f}(\xi-\eta)\widehat{g}(\eta).
\end{split}
\end{equation}

For any $k\in\mathbb{Z}$ let $k^+:=\max(k,0)$ and $k^-:=\min(k,0)$. We fix an even smooth function $\varphi: \Rb\to[0,1]$ supported in $[-8/5,8/5]$ and equal to $1$ in $[-5/4,5/4]$,
and define
\begin{equation*}
\varphi_k(x) := \varphi(x/2^k) - \varphi(x/2^{k-1}) , \qquad \varphi_{\leq k}(x):=\varphi(x/2^k),
  \qquad\varphi_{>k}(x) := 1-\varphi(x/2^{k}),
\end{equation*}
for any $k\in\mathbb{Z}$. Let $P_k$, $P_{\leq k}$, and $P_{>k}$ denote the operators on $\mathbb{T}_R$ defined by the Fourier multipliers $\varphi_k$,
$\varphi_{\leq k}$, and $\varphi_{\geq k}$ respectively. Notice that $P_k=0$ if $2^kR\leq 1/2$.
We also let $\overline{\Zb}:=\Zb\cup\{-\infty\}$ and define the function $\varphi_{-\infty}$ and the operator $P_{-\infty}$ by
\begin{equation*}
\varphi_{-\infty}(\xi):=\mathbf{1}_{\{0\}}(\xi),\qquad P_{-\infty}f(x):=\frac{1}{2\pi R}\int_{\Tb_R}f(y)\,dy.
\end{equation*}

We define the Sobolev spaces $\dot{H}^{N,b}(\Tb_R)$ and the $L^\infty$-based spaces $\dot{W}^{N,b}(\Tb_R)$ by the norms 
\begin{equation}\label{normsx1}
\begin{split}
\|f\|_{\dot{H}^{N,b}}&:=\big\{\|P_{-\infty}f\|_{L^2}^2R^{-2b}+\sum_{k\in\Zb}\|P_kf\|_{L^2}^2(2^{2Nk}+2^{2bk})\big\}^{1/2},\\
\|f\|_{\dot{W}^{N,b}}&:=\|P_{-\infty}f\|_{L^\infty}R^{-b}+\sum_{k\in\Zb}\|P_kf\|_{L^\infty}(2^{Nk}+2^{bk}),
\end{split}
\end{equation}
where $b\in[-1,1]$ and $N\geq \max(b,0)$, see (2.25)--(2.26) in \cite{DIP}. 

Many of our operators are defined by multipliers in the Fourier space. A convenient norm to estimate multipliers is the Wiener norm. For $d\geq 1$ we define the class of multipliers
\begin{equation}\label{Sinfty}
S^\infty=S^\infty((\Zb/R)^d):= \{m: (\Zb/R)^d \to \mathbb{C} :\, {\| m \|}_{S^\infty} := {\|\mathcal{F}^{-1}(m)\|}_{L^1(\Tb_R^d)} < \infty \}.
\end{equation}

Finally we introduce paradifferential operators on $\T_R$, 
which are defined for functions $f,g\in H^1(\T_R)$ by the formula
\begin{equation}\label{Taf}
\begin{split}
\mathcal{F} \big( T_fg\big)(\xi) &:= \frac{1}{2\pi R} 
  \sum_{\eta\in\Z/R} \chi (\xi-\eta,\eta)\widehat{f}(\xi-\eta) \widehat{g}(\eta),\\
\chi(x,y)&:=\sum_{k\in\mathbb{Z}}\varphi_{\leq k-20}(x)\varphi_k(y+x/2).
\end{split}
\end{equation}
We remark that paradifferential is an important technical tool in \cite{DIP};
however, in this paper the notation \eqref{Taf} and associated bounds
are only used in Propositions \ref{IncremBound} and \ref{mainboot}.

\subsection{Normal form reduction}\label{SsecNF}
We consider solutions $(h,\phi)\in C([T_1,T_2]:H^{N_2}\times \dot{H}^{N_2+1/2,1/2})$ of the system (\ref{ww0}) 
on large one-dimensional tori $x\in\Tb_R$, satisfying
\begin{equation}\label{aprio}
\|(h+i|\partial_x|^{1/2}\phi)(t)\|_{H^{N_2}}\leq B_2,\qquad \|(h+i|\partial_x|^{1/2}\phi)(t)\|_{\dot{W}^{N_\infty,0}}
  \leq \varepsilon_\infty,
\end{equation}
for some parameters $N_2\geq 20$, $N_\infty\in[4,N_2-4]$, $B_2\in [0,\infty)$, $\varepsilon_\infty\ll1$, and for any $t\in[T_1,T_2]$. 
We will prove the following:

\begin{prop}[Normal form reduction] \label{normalprop} 
Assume $A\in[3,N_2/2]$ is a fixed integer. Then we can construct a normal form transformation
\begin{equation}\label{normal}
w:=u+\mathcal{A}_2(u)+\mathcal{A}_3(u),
\end{equation} 
such that the new variable $w$ satisfies the equation
\begin{equation}\label{ww1}(\partial_t+i|\partial_x|^{1/2})w=\sum_{r=3}^A\Nc_r(w)+\Nc_{>A}.
\end{equation} 
Here $\mathcal{A}_r$ and $\Nc_r$ are homogeneous expressions of degree $r$, which can be represented in the Fourier space in the form
\begin{equation}\label{nonlin1}
\begin{split}
\widehat{\mathcal{A}_r(u)}(\xi)=\frac{1}{(2\pi R)^{r-1}}\sum_{\iota_j\in\{\pm\}}\sum_{\eta_1,\ldots,\eta_r\in\Zb/R,\,\eta_1+\ldots+\eta_r=\xi}a_{r,\iota_1\ldots\iota_r}(\eta_1,\ldots,\eta_r)\prod_{j=1}^r\widehat{u^{\iota_j}}(\eta_j),\\
\widehat{\Nc_r(w)}(\xi)=\frac{1}{(2\pi R)^{r-1}}\sum_{\iota_j\in\{\pm\}}\sum_{\eta_1,\ldots,\eta_r\in\Zb/R,\,\eta_1+\ldots+\eta_r=\xi}n_{r,\iota_1\ldots\iota_r}(\eta_1,\ldots,\eta_r)\prod_{j=1}^r\widehat{w^{\iota_j}}(\eta_j),
\end{split}
\end{equation} 
where the symbols $a_{r,\iota_1\ldots\iota_r}$ and $n_{r,\iota_1\ldots\iota_r}$ vanish unless there is $r'\in\{0,\ldots,r\}$ such that $\iota_1=\ldots=\iota_{r'}=+$ and $\iota_{r'+1}=\ldots=\iota_r=-$. Here and below we define $z^+=z$ and $z^-=\overline{z}$ for any $z\in\mathbb{C}$. 

The functions $w$, $\mathcal{N}_n(w)$, and $\Nc_{>A}$ satisfies the bounds
\begin{equation}\label{nonlin2}
\begin{split}
&\|w\|_{H^{N_2-2}}\lesssim B_2,\qquad \|w\|_{\dot{W}^{N_\infty-2,0}}\lesssim \varepsilon_\infty,\\
&\|\mathcal{N}_n(w)\|_{H^{N_2-n-3/2}}\lesssim \varepsilon_\infty^{n-1}B_2,\qquad \|\Nc_{>A}\|_{H^{N_2-A-5/2}}\lesssim \varepsilon_\infty^AB_2.
\end{split}
\end{equation}

In addition, the symbols  $a_{r,\iota_1\ldots\iota_r}$, $n_{r,\iota_1\ldots\iota_r}$ in (\ref{nonlin1}) satisfy the following properties:
\begin{enumerate}
\item For any $r\in\{3,\ldots,A\}$ the symbols $n_{r,\iota_1\ldots\iota_r}$ are homogeneous of degree $r-1/2$ and satisfy suitable $S^\infty$ estimates, i.e. for any $k,k_1,\ldots,k_r\in\overline{\Zb}$ and $\lambda\in[0,\infty)$
\begin{equation}\label{symbolbound}
\begin{split}
&n_{r,\iota_1\ldots\iota_r}(\lambda(\xi_1,\ldots,\xi_r))=\lambda^{r-1/2}n_{r,\iota_1\ldots\iota_r}(\xi_1,\ldots,\xi_r),\\
&\big \|n_{r,\iota_1\ldots\iota_r}(\xi_1,\ldots,\xi_r)\varphi_{k_1}(\xi_1)\cdot\ldots\cdot\varphi_{k_r}(\xi_r)\varphi_k(\xi_1+\ldots+\xi_r)\big\|_{S^\infty}\lesssim 2^{k/2}2^{(r-1)\max(k_1,\ldots,k_r)}.
\end{split}
\end{equation}

\item For $r\in\{2,3\}$ the symbols $a_{r,\iota_1\ldots\iota_r}$ are homogeneous of degree $r-1$ and satisfy suitable $S^\infty$ estimates, i.e. for any $k,k_1,\ldots,k_r\in\overline{\Zb}$ and $\lambda\in[0,\infty)$
\begin{equation}\label{symbolboundN}
\begin{split}
&a_{r,\iota_1\ldots\iota_r}(\lambda(\xi_1,\ldots,\xi_r))=\lambda^{r-1}a_{\iota_1\ldots\iota_r}(\xi_1,\ldots,\xi_r),\\
&\big \|a_{r,\iota_1\ldots\iota_r}(\xi_1,\ldots,\xi_r)\varphi_{k_1}(\xi_1)\cdot\ldots\cdot\varphi_{k_r}(\xi_r)\varphi_k(\xi_1+\ldots+\xi_r)\big\|_{S^\infty}\lesssim 2^{k/2}2^{(r-3/2)\max(k_1,\ldots,k_r)}.
\end{split}
\end{equation}

\item For $r=3$ the symbols $n_{3,\iota_1\iota_2\iota_3}$ vanish unless $(\iota_1\iota_2\iota_3)=(++-)$. The symbol $n_{3,++-}$ is purely imaginary, symmetric in the first two variables, supported in the set $\{(\eta_1,\eta_2,\eta_3)\in\Rb^3:\,\eta_1\eta_2\eta_3(\eta_1+\eta_2+\eta_3)\leq 0\}$, and satisfies the $C^{1/2}$ bounds
\begin{equation}\label{symbolcubic}
|n_{3,++-}(\underline{\xi})-n_{3,++-}(\underline{\eta})|\lesssim \big|\underline{\xi}-\underline{\eta}\big|^{1/2}(|\underline{\xi}|+|\underline{\eta}|)^2\qquad\text{ for any }\underline{\xi},\underline{\eta}\in\Rb^3.
\end{equation}
\end{enumerate}
\end{prop}

\subsection{Initial data and energy estimate} 
Recall the initial data $u_{\mathrm{in}}$ given by \eqref{data_u}. 
By \eqref{normal} 
\begin{equation}\label{data_w}w(0) = w_{\mathrm{in}} = u_{\mathrm{in}}+\Ac_2(u_{\mathrm{in}})
  +\Ac_3(u_{\mathrm{in}}).
\end{equation} 
We start with some general bounds on products of functions with random Fourier coefficients. 

\begin{prop}\label{initialprop} 

(i) [Multilinear Gaussian estimates] Assume that $g_k:\Omega\to \mathbb{C}$, $k\in\{1,\ldots,n\}$, are i.i.d. normalized complex Gaussian random variables on a probability space $\Omega$ and define
\begin{equation}\label{genprob1}
G(\omega):=\sum_{k_1,\ldots,k_r\in\{1,\ldots,n\}}a_{k_1,\ldots,k_r}\prod_{i=1}^rg_{k_i}^{\iota_i}(\omega),
\end{equation}
where $r\geq 1$, $\omega\in\Omega$, $\iota_j\in\{+,-\}$, and $a_{k_1,\ldots,k_r}\in\mathbb{C}$. Then, for any $A\geq 1$
\begin{equation}\label{genprob2}
\mathbb{P}\big\{\omega:\,|G(\omega)|\geq A\|G(\omega)\|_{L^2_\omega}\big\}\leq Ce^{-A^{2/r}/C},
\end{equation}
for some constant $C\geq 1$. 

(ii) Assume that $\Mc$ is an $r$-linear operator defined by
\begin{equation}\label{gen1}
\mathcal{F}\big[\Mc(f_1,\ldots,f_r)\big](\xi):=\frac{1}{(2\pi R)^{r-1}}\sum_{\eta_1,\ldots,\eta_r\in\mathbb{Z}/R,\,\eta_1+\ldots+\eta_r=\xi}m(\eta_1,\ldots,\eta_r)\widehat{f}_1(\eta_1)\cdot\ldots\cdot \widehat{f_r}(\eta_r),
\end{equation}
with an associated symbol $m$ satisfying 
\begin{equation}\label{gen2}
|m(\eta_1,\ldots,\eta_r)| \leq (1+|\eta_1|+\ldots+|\eta_r|)^A,
\end{equation}
for some $A\geq 0$. Let $f=f(\omega)$ be a random function with Fourier coefficients $\widehat{f}(k) = 
R^{1/2}\varphi(k)\cdot g_k(\omega)$ with i.i.d Gaussians $g_k$, where $\varphi:\Rb\to\Rb$ satisfies $|\varphi(x)|\leq (1+|x|)^{-A'}$, $A'\geq A+r+2$. 
Then, for $\eta\in(0,1/10)$ and $\iota_1,\ldots,\iota_r\in\{+,-\}$, with high probability $\geq 1 - e^{-R^{\eta/r}}$ one has
\begin{equation}\label{gen2.5}
\| \Mc(f^{\iota_1},\dots, f^{\iota_r}) \|_{L^\infty(\Tb_R)} 
  + R^{-1/2} \| \Mc(f^{\iota_1},\dots, f^{\iota_r}) \|_{L^2(\Tb_R)} \lesssim R^{\eta}.
\end{equation}

(iii) In particular, there is a set $\Omega'\subseteq\Omega$ with  $\mathbb{P}(\Omega\setminus\Omega')\leq e^{-R^{2\eta_0}}$ such that for all $\omega\in\Omega'$ the initial data $u_{in}$ defined in \eqref{data_u} satisfies the bounds
\begin{equation}\label{initialenergy}
\|u_{in}\|_{H^{N_0+4}}+\|w_{\mathrm{in}}\|_{H^{N_0+4}}\leq \epsilon^{1-\theta/2} R^{1/2},
  \quad \|u_{\mathrm{in}}\|_{\dot{W}^{8,0}}+\|w_{\mathrm{in}}\|_{\dot{W}^{8,0}}\leq \epsilon ^{1-\theta/2}.
\end{equation}
\end{prop}

\begin{proof} (i) The general bounds \eqref{genprob2} follow by the hypercontractivity estimate for Ornstein-
Uhlenbeck semigroups; see for example \cite{DeHa1} and references therein.

(ii) Let
\begin{equation}\label{gen3}
\begin{split}
G(x,\omega)&:=\langle\nabla\rangle_x\Mc(f^{\iota_1},\dots, f^{\iota_r})(x)\\
&=\frac{1}{(2\pi R)^{r}}\sum_{k_1,\ldots,k_r\in\mathbb{Z}/R}e^{ix(k_1+\ldots+k_r)}m(k_1,\ldots,k_r)(1+|k_1|^2+\ldots+|k_r|^2)^{1/2}\\
&\times\varphi(k_1)\cdot\ldots\cdot\varphi(k_r)\cdot R^{r/2}g_{\iota_1k_1}^{\iota_1}(\omega)\cdot\ldots\cdot g_{\iota_rk_r}^{\iota_r}(\omega).
\end{split}
\end{equation}
For any $x\in\Tb_R$ fixed we use the bounds \eqref{gen2} and the decay assumption on $\varphi$ to estimate
\begin{equation*}
\|G(x,\omega)\|_{L^2_\omega}^2\lesssim_r R^{-r}\sum_{k,k'\in(\Zb/R)^r}(1+|k|)^{-r-1}(1+|k'|)^{-r-1}\int_\Omega \prod_{a,b\in\{1,\ldots,r\}}g_{\iota_ak_a}^{\iota_a}(\omega)\overline{g_{\iota_bk_b}^{\iota_b}(\omega)}\,d\omega.
\end{equation*}
Notice that, due to independence, the variables $k_1,\ldots,k_r, k'_1,\ldots, k'_r$ must come in pairs in order for the $\omega$-integral to be nonzero. Thus the summation over $k,k'$ contributes a factor $O(R^r)$, and it follows that
\begin{equation}\label{gen4}
\|G(x,\omega)\|_{L^2_\omega}\lesssim 1\qquad\text{ for any }x\in\Tb_R.
\end{equation}

It follows from \eqref{genprob2} that for any $x\in\Tb_R$ and $A\geq 1$
\begin{equation*}
\mathbb{P}\big\{\omega:\,|G(x,\omega)|\geq A\|G(x,\omega)\|_{L^2_\omega}\big\}\leq Ce^{-A^{2/r}/C},
\end{equation*}
for some constant $C\geq 1$. Therefore, using also \eqref{gen4},
\begin{equation*}
\|G(x,\omega)\|_{L^p_\omega}\lesssim_r p^{r/2}\qquad\text{ for any }x\in\Tb_R,\,p\in[2,\infty).
\end{equation*}
In particular $\|G(x,\omega)\|_{L^p_\omega L^p_x}\lesssim_r p^{r/2}R^{1/p}$ for any $p\in[2,\infty)$, so
\begin{equation*}
R^{\eta}\mathbb{P}\big\{w:\,\|G(x,\omega)\|_{L^p_x}\geq R^{\eta}\big\}^{1/p}\lesssim_r p^{r/2}R^{1/p}.
\end{equation*}
We apply this with $p=\eta^{-2}R^{\eta/r}$, so $|\mathbb{P}\big\{w:\,\|G(x,\omega)\|_{L^p_x}\geq R^{\eta}\big\}|\leq e^{-R^{\eta/r}}$ for $R$ sufficiently large. The $L^\infty$ bound in \eqref{gen2.5} follows since $\|\Mc(f^{\iota_1},\dots, f^{\iota_r}) \|_{L^\infty(\Tb_R)}\lesssim \|G(x,\omega)\|_{L^p_x}$ for any $p\in[2,\infty]$ and $\omega\in\Omega$. The $L^2$ bound in \eqref{gen2.5} then follows by integration.
\end{proof}

Next, we recall the following energy estimate proved in \cite{DIP}:

\begin{thm}[Energy increment]
\label{IncremBound}
Assume that $N_2\geq 6$, $N_\infty\in[4,N_2-2]$, $t_1\leq t_2\in\mathbb{R}$,
and $(h,\phi)\in C([t_1,t_2]:H^{N_2}\times \dot{H}^{N_2,1/2})$
is a solution of the system \eqref{ww0}. Let
\begin{align}\label{omega}
B:=\frac{G(h)\phi+\partial_xh\partial_x\phi}{1+(\partial_xh)^2},\qquad\omega := \phi-T_Bh,
\end{align}
see \eqref{Taf} for the definition of the paraproduct $T_ab$. Assume that 
\begin{align}\label{BootstrapFull}
\begin{split}
&\|h(t)\|_{H^{N_2}}+\||\partial_x|^{1/2}\omega(t)\|_{H^{N_2}} +\||\partial_x|^{1/2}\phi(t)\|_{H^{N_2-1/2}}\leq A,
\\
&\|h(t)\|_{\dot{W}^{N_\infty,0}} + \||\partial_x|^{1/2}\omega(t)\|_{\dot{W}^{N_\infty,0}}
  + \||\partial_x|^{1/2}\phi(t)\|_{\dot{W}^{N_\infty-1/2,0}}\leq \varepsilon(t),
\end{split}
\end{align}
for any $t\in[t_1,t_2]$, where $A\in (0,\infty)$ and $\varepsilon:[t_1,t_2]\to [0,\infty)$ 
is a continuous function satisfying $\varepsilon(t)\ll 1$.
Then there is a constant $C_{N_2}\geq 1$ such that
\begin{equation}\label{MainGrowth}
\big\|(h+i|\partial_x|^{1/2}\omega)(t_2)\big\|_{H^{N_2}}^2
  \leq C_{N_2}\big\|(h+i|\partial_x|^{1/2}\omega)(t_1)\big\|_{H^{N_2}}^2 + C_{N_2} A^2\int_{t_1}^{t_2}[\varepsilon(s)]^3\,ds.
\end{equation}
\end{thm}

\subsection{Propagation of randomness and main bootstrap}\label{SsecPR}
The key step in the proof of Theorem \ref{main} is the following result,
which shows propagation of randomness in the form of decay of an $L^\infty$-type norm:

\begin{prop}[Propagation of randomness]\label{apriori} 
Recall $T_1=\epsilon^{-8/3+\theta_0}$ and $\theta=\theta_0/100$. With probability $1-3e^{-R^{2\eta_0}}$ the following statement is true: assume $t\in[0,T_1]$ and $u=h+i|\partial_x|^{1/2}\phi\in C([0,t]: H^{N_0})$ is a solution of the system \eqref{ww0}, $w$ is defined as in \eqref{normal}, and
\begin{equation}\label{bootstrap}
\sup_{s\in[0,t]}\big[\|u(s)\|_{H^{N_0}}+\|w(s)\|_{H^{N_0-2}}\big]
  \leq \epsilon^{1-\theta} R^{1/2},\quad \sup_{s\in[0,t]} \big[\|u(s)\|_{\dot{W}^{6,0}} +\|w(s)\|_{\dot{W}^{4,0}}\big] \leq \epsilon^{1-\theta}.
\end{equation} 
Then
\begin{equation}\label{bootstrap2}
\sup_{s\in[0,t]} \|w(s)\|_{\dot{W}^{4,0}} \lesssim \epsilon^{1-\theta/2}.
\end{equation}
\end{prop}

The proof of Proposition \ref{apriori} will occupy most of the the paper,
and will follow from the main steps described in Section \ref{SecRan}. We show now that Proposition \ref{apriori} can be used as a black box, together with Propositions \ref{normalprop}, \ref{initialprop} (ii) and Theorem \ref{IncremBound}, to obtain Theorem \ref{main}.

\begin{prop}[Main Bootstrap]\label{mainboot}
Assume $N_0, T_1, u=h+i|\partial_x|^{1/2} \phi$ are as in Theorem \ref{main} and  define $\omega$ as in \eqref{omega}. Let $N_1:=6$ and assume that initially $u(0)=u_{\mathrm{in}}$ is a random function satisfying \eqref{data_u}. Then, with probability at least $1-4e^{-R^{2\eta_0}}$, the following statement holds: assume $t\in[0,T_1]$ and $u=h+i|\partial_x|^{1/2}\phi\in C([0,t]:H^{N_0})$ is a solution of the system \eqref{ww0} satisfying
\begin{align}
\label{mainbootasE}
\sup_{t'\in[0,t]}
  \Big( \|(h+i|\partial_x|^{1/2}\omega)(t')\|_{H^{N_0+1/2}} +\|(h+i|\partial_x|^{1/2}\phi)(t')\|_{H^{N_0}} \Big)
  & \leq \epsilon^{1-\theta} R^{1/2},\\
\label{mainbootasD}
\sup_{t'\in[0,t]}
  \Big(\|(h+i|\partial_x|^{1/2}\omega)(t')\|_{\dot{W}^{N_1+1/2,0}}+ \|(h+i|\partial_x|^{1/2}\phi)(t')\|_{\dot{W}^{N_1,0}} \Big) & 
  \leq \epsilon^{1-\theta}.
\end{align}
Then
\begin{align}
\label{mainbootconcE}
\sup_{t'\in[0,t]}  \Big( \|(h+i|\partial_x|^{1/2}\omega)(t')\|_{H^{N_0+1/2}} +\|(h+i|\partial_x|^{1/2}\phi)(t')\|_{H^{N_0}} \Big)
  & \lesssim \epsilon^{1-3\theta/4} R^{1/2},\\
\label{mainbootconcD}
\sup_{t'\in[0,t]}
  \Big(\|(h+i|\partial_x|^{1/2}\omega)(t')\|_{\dot{W}^{N_1+1/2,0}}+ \|(h+i|\partial_x|^{1/2}\phi)(t')\|_{\dot{W}^{N_1,0}} \Big) & \lesssim\epsilon^{1-3\theta/4}.
\end{align}
In particular, Theorem \ref{main} holds.
\end{prop}

\begin{proof} For convenience let us introduce
\begin{equation}\label{variables0}
v := h+i|\partial_x|^{1/2}\omega= u - i|\partial_x|^{1/2}T_B h,
\end{equation}
and recall that from \cite[Lemma 3.1]{DIP} we have
\begin{align}\label{estBDIP}
\|B(t')\|_{H^{N_0-1}} \lesssim \epsilon^{1-\theta} R^{1/2} , \qquad \|B(t')\|_{L^\infty}+\sum_{k\geq 0}2^{(N_1-1)k}\|P_kB(t')\|_{L^\infty} \lesssim \epsilon^{1-\theta},
\end{align}
for all $t'\in[0,t]$. Recall the definition \eqref{normal} of the variable $w = u + \mathcal{A}_2(u) + \mathcal{A}_3(u)$
and that the multilinear operators $\mathcal{A}_r$ have symbols satisfying \eqref{symbolboundN}. In view of Lemma \ref{touse} we have 
\begin{align}\label{normalprod}
\begin{split}
& {\| \mathcal{A}_r(u) \|}_{H^{N_0-r+1}} \lesssim R^{1/2}\epsilon^{r(1-\theta)}, \qquad
 {\| \mathcal{A}_r(u) \|}_{\dot{W}^{N_1-r+1,0}} \lesssim \epsilon^{r(1-\theta)}.
\end{split}
\end{align}
for $r=2,3$. 

The bounds \eqref{initialenergy} on $u_{in}$, the bounds \eqref{estBDIP}, and the definition $\omega:=\phi-T_Bh$ show that
\begin{equation}\label{hola1}
\|(h+i|\partial_x|^{1/2}\omega)(0)\|_{H^{N_0+2}}\lesssim \eps^{1-\theta/2}R^{1/2}.
\end{equation}
We apply now Theorem \ref{IncremBound}, use the bootstrap assumption \eqref{mainbootasE}--\eqref{mainbootasD}, and recall that $T_1=\epsilon^{-8/3+\theta_0}$ to see that
\begin{equation*}
\sup_{t'\in[0,t]}\|(h+i|\partial_x|^{1/2}\omega)(t')\|_{H^{N_0+1/2}}\lesssim \eps^{1-\theta/2}R^{1/2}.
\end{equation*}
Since $\phi=\omega+T_Bh$, it follows using \eqref{estBDIP} that $\||\partial_x|^{1/2}\phi(t')\|_{H^{N_0}}\lesssim \eps^{1-\theta/2}R^{1/2}$. Thus
\begin{equation}\label{hola2}
\sup_{t'\in[0,t]}  \Big( \|(h+i|\partial_x|^{1/2}\omega)(t')\|_{H^{N_0+1/2}} +\|(h+i|\partial_x|^{1/2}\phi)(t')\|_{H^{N_0}} \Big)\lesssim \epsilon^{1-\theta/2} R^{1/2},
\end{equation}
which, in particular, implies the desired bounds \eqref{mainbootconcE}.

To prove the pointwise bounds \eqref{mainbootconcD} we notice that $\|w(t')\|_{\dot{W}^{N_1-2,0}}\lesssim \epsilon^{1-\theta}$ and $\|w(t')\|_{H^{N_0-2}}\lesssim \epsilon^{1-\theta}R^{1/2}$, due to the assumptions \eqref{mainbootasD} and the bounds \eqref{normalprod}. The hypothesis \eqref{bootstrap} of Proposition \ref{apriori} is therefore satisfied, so we can apply the proposition to conclude that
\begin{equation}\label{hola3}
\sup_{t'\in[0,t]}\|w(t')\|_{\dot{W}^{4,0}}\lesssim\epsilon^{1-\theta/2}.
\end{equation}
Since $u=w-\mathcal{A}_2(u)-\mathcal{A}_3(u)$, it follows using also \eqref{normalprod} that
\begin{equation}\label{hola4}
\sup_{t'\in[0,t]}\|u(t')\|_{\dot{W}^{4,0}}\lesssim\epsilon^{1-\theta/2}.
\end{equation}
Using the inequality $\|u\|_{H^{N_0}}(t')\lesssim \epsilon^{1-\theta/2} R^{1/2}$ proved in \eqref{hola2} we can further estimate
\begin{equation*}
\begin{split}
\|u(t')\|_{\dot{W}^{8,0}}&=\sum_{k\in\overline{\Zb}}\|P_ku\|_{L^\infty}(2^{8k}+1)\\
&\lesssim\sum_{k\in\overline{\Zb},\,k\leq K_0}2^{4K_0}\|P_ku\|_{L^\infty}(2^{4k}+1)+\sum_{k\geq K_0}\|P_ku\|_{L^2}2^{N_0k}2^{(9-N_0)k}\\
&\lesssim 2^{4K_0}\epsilon^{1-\theta/2}+2^{(9-N_0)K_0}\epsilon^{1-\theta/2}R^{1/2},
\end{split}
\end{equation*}
for any $K_0\geq 0$. We fix now $K_0$ such that $2^{N_0K_0}\approx R^{1/2}$. Since $R^{5/N_0}\lesssim \epsilon^{-\theta_0/4}$ (recall the assumptions $\epsilon\leq R^{-\theta_0}$ and $N_0\geq \theta^{-3})$ it follows that $\|u\|_{\dot{W}^{8,0}}\lesssim\epsilon^{1-3\theta/4}$, which suffices to prove the desired bounds \eqref{mainbootconcD}.
\end{proof}

\section{Propagation of randomness: Proof of Proposition \ref{apriori}}\label{SecRan}
In this section we break down the proof of Proposition \ref{apriori} into several steps. 
It is here that we need to introduce 
probabilistic and combinatorial tools, including Feynman trees, couples, and renormalization.

We begin with a first approximation of the normal form variable $w$, which we denote by $w_{\mathrm{tr}}$, which is obtained by truncating the equations in homogeneity and projecting away all high-frequencies; 
see Proposition \ref{properror1}.

Then we introduce the renormalized interaction coefficients $(a_k)_{\mathrm{tr}}(t)$, see the definition \eqref{profile},
obtained by renormalizing the Fourier coefficients of $w_{\mathrm{tr}}$ by the phase $\Gamma$
(and applying some convenient time and amplitude rescaling).
The introduction of this renormalization is one of the main new ingredients in our proof. 
Note how at this initial stage $\Gamma$ is composed of a term $\Gamma_0=\Gamma_0(k)$ which takes into account 
the standard trivial resonances of the form $(k,k,k_1,k_1)$, 
and a yet-to-be-determined time-frequency dependent contribution $\Gamma_1=\Gamma_1(t,k)$. We note also that the functions $\Gamma$ and $\Gamma_1$ are deterministic, i.e. they do not depend on the choice of the random coefficients $g_k(\omega)$. 

Then, from the Duhamel's formula \eqref{profileeqn} associated to $(a_k)_{\mathrm{tr}}$ 
we construct our main approximation, denoted $w_{\mathrm{app}}$, by long Feynman trees expansions (depending on the parameter $\theta_0$), leading to the formulas \eqref{ansatzprofile}-\eqref{ansatz3} and \eqref{defofjt}.
Note that our definitions of trees and other combinatorial objects (as well as the estimates in later sections)
need to take into account the normal form transformation which, 
in particular, implies that the initial data for $w$ (and $w_{\mathrm{tr}}$)
is not a standard Gaussian sum as in the case of the original physical variable $u$.

In Subsection \ref{sectionkq} we introduce the expressions for couples, 
which will be the main object to be estimated,
and define the renormalizing phase $\Gamma_1$ by a fixed point identity.

Subsection \ref{mainproof} contains the statements of all the main estimates:
Proposition \ref{propjr} gives the main bounds for trees/couples, and for $\Gamma_1$,
while Proposition \ref{propdiff} estimates the difference between the approximate solution's
profile ${(a_k)}_{\mathrm{app}}$ and the truncated solution's profile ${(a_k)}_{\mathrm{tr}}$.

Finally in Subsection \ref{sec.L1est} we give an estimate for oscillatory integrals 
involving multilinear phases obtained from the modified time oscillations $|k|^{1/2}t + \Gamma(t,k)$ to be used in the rest of the arguments,
and in Subsection \ref{seccount} we provide some basic counting estimate for modified quadratic and cubic phases.

\smallskip
\subsection{Frequency and homogeneity truncation}\label{approxtr} 
For the rest of this paper we fix one more parameter $A:=\lfloor\theta^{-1}\rfloor$, and then we fix 
$$K_{\mathrm{tr}}:=\lfloor \log_2(R^{1/A})\rfloor.$$  
With the notation in Proposition \ref{normalprop}, we define $w_{\mathrm{tr}}$ as the solution to the Cauchy problem
\begin{equation}\label{wwapp}
(\partial_t+i|\partial_x|^{1/2})w_{\mathrm{tr}}=\sum_{r=3}^AP_{\leq K_{\mathrm{tr}}}\big[\Nc_r(w_{\mathrm{tr}})\big],\qquad w_{\mathrm{tr}}(0) = w_{\mathrm{in}}^{\mathrm{tr}},
\end{equation} 
where $P_{\leq K}$ denote the standard Littlewood-Paley projections (as defined in section \ref{SecNot}), and
\begin{align}\label{wtrin}
w_{\mathrm{in}}^{\mathrm{tr}} := P_{\leq K_{\mathrm{tr}}} u_{\mathrm{in}} + P_{\leq K_{\mathrm{tr}}}\Ac_2(P_{\leq K_{\mathrm{tr}}}u_{\mathrm{in}})
  +P_{\leq K_{\mathrm{tr}}}\Ac_3(P_{\leq K_{\mathrm{tr}}}u_{\mathrm{in}}).
\end{align}
The equation (\ref{wwapp}) is obtained by truncating the higher order nonlinearities, 
as well as all projecting away high frequencies.
The following proposition shows how to control the difference between the real solution $w$
and the truncated solution $w_{\mathrm{tr}}$:

\begin{prop}\label{properror1} 
Assume that $t\leq T_1$ and $u\in C([0,t]:H^{N_0})$ is a solution of the system \eqref{ww0} satisfying the assumption \eqref{bootstrap}
\begin{equation}\label{error1}
\sup_{s\in[0,t]}\big[\|u(s)\|_{H^{N_0}}+\|w(s)\|_{H^{N_0-2}}\big]
  \leq \epsilon^{1-\theta} R^{1/2},\quad \sup_{s\in[0,t]} \big[\|u(s)\|_{\dot{W}^{6,0}} +\|w(s)\|_{\dot{W}^{4,0}}\big] \leq \epsilon^{1-\theta}.
  \end{equation} 
Recall the nonlinearities $\mathcal{N}_r(w)=\mathcal{N}_r(w,\ldots,w)$ defined in \eqref{nonlin1} and define
\begin{align}\label{properr1not}
& \Nc_r^j(F)[G] = \Nc_r(F,\dots, G, \dots, F), \quad 0\leq j \leq r,
\end{align}
where the function $G$ appears in the $j-$th position and the other $r-1$ arguments are $F$. 

In addition to \eqref{error1}, we assume that any solution $W$ to the linear equation with forcing 
\begin{equation}\label{error1.1'}
(\partial_t+i|\partial_x|^{1/2})W = \sum_{r=3}^A\sum_{1\leq j\leq r}P_{\leq K_{\mathrm{tr}}}
   \Nc_r^j(w_{\mathrm{tr}})[W] 
   + \Rc
\end{equation} 
satisfies the bounds
\begin{equation}\label{error1.2'}
\| W \|_{L_t^\infty\dot{W}^{4,0}_x}
  \lesssim R^8(\| W(0)\|_{H^5} + \|\Rc\|_{L_t^\infty H^5_x} ).
\end{equation}
Then, as a conclusion we have that
\begin{equation}\label{error2'}
\sup_{t'\in[0,t]}\|w(t')-w_{\mathrm{tr}}(t')\|_{\dot{W}^{4,0}} \leq \epsilon^{A/4}.
\end{equation}


\end{prop}

The above proposition has been formulated in a way to keep the logic of the proof simpler and separate the deterministic parts of the argument and the probabilistic ones.
However, the claim that the bounds \eqref{error1.2'} hold for all linear solutions of the equation \eqref{error1.1'} (which is an implicit claim about the function  $w_{tr}$) relies on 
probabilistic arguments, and therefore holds for a set of large probability.
The proof of \eqref{error1.2'} is given at the end of Section \ref{sec.rem}.

\begin{proof}[Proof of Proposition \ref{properror1}]
Recall that $w$ satisfies the equation \eqref{ww1} with $w(0)=w_{\mathrm{in}}$ satisfying the bounds in Proposition \ref{initialprop},
while $w_{\mathrm{tr}}$ solves the equation \eqref{wwapp}.
Then, their difference $W = w - w_{\mathrm{tr}}$ satisfies the equation \eqref{error1.1'} with
\begin{align}\label{errpr1'}
W(0) =  w_{\mathrm{in}} - w_{\mathrm{in}}^{\mathrm{tr}} 
\end{align}
and
\begin{align}\label{errpr3'}
\begin{split}
\mathcal{R} & := \mathcal{N}_{>A} (u) + \sum_{r=3}^A P_{> K_{\mathrm{tr}}} \mathcal{N}_r(w)
  + \sum_{r=3}^A P_{\leq K_{\mathrm{tr}}} \Big( \Nc_r(w) - \Nc_r(w_{\mathrm{tr}})
  - \sum_{1\leq j\leq r} \Nc_r^j \big(w_{\mathrm{tr}}\big) [W] \Big).
\end{split}
\end{align}

We bound first the contribution of the initial data, i.e. we show that
\begin{align}\label{errprconc1}
{\big\| w_{\mathrm{in}} - w_{\mathrm{in}}^{\mathrm{tr}}
  \big\|}_{H^5} \lesssim \epsilon^{2A/3}.
\end{align}
For this we estimate
\begin{align*}
{\big\| w_{\mathrm{in}} - w_{\mathrm{in}}^{\mathrm{tr}}
  \big\|}_{H^5}
  & \lesssim {\big\| P_{> K_{\mathrm{tr}}}u_{\mathrm{in}} \|}_{H^5} 
  + \sum_{r=2,3}{\big\|  \Ac_r(u_{\mathrm{in}}) - P_{\leq K_{\mathrm{tr}}} \Ac_r(P_{\leq K_{\mathrm{tr}}}u_{\mathrm{in}}) \big\|}_{H^5}.
\end{align*}
Since $\|u_{\mathrm{in}}\|_{H^{N_0}}\leq \epsilon^{1-\theta} R^{1/2}$
(see \eqref{error1}) we can estimate
\begin{align}\label{errprd1}
\begin{split}
{\big\| P_{> K_{\mathrm{tr}}}u_{\mathrm{in}} \|}_{H^5} 
  & \lesssim (R^{1/A})^{-(N_0-5)} {\| u_{\mathrm{in}} \|}_{H^{N_0}}
  \lesssim R^{-A} \epsilon^{1-\theta} R^{1/2}.
\end{split}
\end{align}
Moreover, by writing $\Ac_2(f) = \Ac_2(f,f)$ (with the natural definition of the bilinear operator $\Ac_2$)
and using Lemma \ref{level2T} (ii) and the bounds \eqref{symbolboundN}, we can estimate
\begin{align}\label{errprd2}
\begin{split}
& {\big\| \Ac_2(u_{\mathrm{in}}) - P_{\leq K_{\mathrm{tr}}} \Ac_2(P_{\leq K_{\mathrm{tr}}} u_{\mathrm{in}}) \big\|}_{H^5}
  \\
  & \lesssim {\big\| P_{> K_{\mathrm{tr}}} \Ac_2(u_{\mathrm{in}}, u_{\mathrm{in}})  \big\|}_{H^5}
  + {\big\| \Ac_2(P_{> K_{\mathrm{tr}}} u_{\mathrm{in}}, u_{\mathrm{in}})  \big\|}_{H^5}
  + {\big\| \Ac_2(P_{\leq K_{\mathrm{tr}}}u_{\mathrm{in}}, P_{> K_{\mathrm{tr}}}u_{\mathrm{in}}) \big\|}_{H^5}
  \\
  & \lesssim (R^{1/A})^{-(N_0-10)} 
  {\|u_{\mathrm{in}} \|}_{H^{N_0}} {\| u_{\mathrm{in}} \|}_{\dot{W}^{4,0}}
  \lesssim R^{-A} \epsilon^{2-2\theta} R^{1/2}.
\end{split}
\end{align}
A similar argument can be used to obtain similar bounds for the contribution of the cubic term $\Ac_3(u_{in})$. The desired bounds \eqref{errprconc1} follow.

Using now the assumption \eqref{error1.2'} we would like to prove the following bootstrap estimate:
if
\begin{equation}\label{errprboot}
\sup_{t'\in[0,t]}\|w(t')-w_{\mathrm{tr}}(t')\|_{\dot{W}^{4,0}} \leq 2\epsilon^{A/4}
\end{equation} 
then
\begin{equation}\label{errprbootconc}
\sup_{t'\in[0,t]}\|w(t')-w_{\mathrm{tr}}(t')\|_{\dot{W}^{4,0}} \leq \epsilon^{A/4}. 
\end{equation} 
We need to bound $\|\Rc\|_{L^\infty_tH^5_x}$. The first two terms in \eqref{errpr3'} can be directly estimated using the estimates in the last line of \eqref{nonlin2} and the apriori assumptions \eqref{error1}: for any $s\in[0,t]$
\begin{align}\label{errprconc2}
\big\| \mathcal{N}_{>A} (u(s)) \big\|_{H^5} 
\lesssim {\|w(s)\|}_{H^{N_0-2}} {\|w(s)\|}_{\dot{W}^{4,0}}^A 
\lesssim 
R^{1/2} \cdot (\epsilon^{1-\theta})^{A+1}  \lesssim \epsilon^{2A/3},
\end{align}
and
\begin{align}\label{errprconc2.5}
\begin{split}
\|P_{>K_{tr}}\mathcal{N}_r(w(s))\|_{H^5}\lesssim
 R^{-A/2}\|\mathcal{N}_r(w(s))\|_{H^{N_0/2}} \lesssim R^{-A/3}.
\end{split}
\end{align}

Finally, we estimate the last term in \eqref{errpr3'}. Observe that the expression
in the square parenthesis is a linear combination of terms
of the form $\Nc_r(w_1,w_2,\dots,w_r)$ such that there exists $1\leq i\neq j \leq r$ 
with $w_i = w_j = w - w_{\mathrm{tr}}$,
and, for $k\neq i,j$, $w_k \in \{w,w_{\mathrm{tr}},w-w_{\mathrm{tr}}\}$.
For each $r=3,\dots, A$, we can use \eqref{nonlin1}, the $S^\infty$ bounds \eqref{symbolbound}, 
and the product estimates in Lemma \ref{level2T} (ii)  to bound 
\begin{align}\label{errprd5}
\begin{split}
&\Big\| 
  P_{\leq K_{\mathrm{tr}}} \big[ \Nc_r(w(s)) - \Nc_r(w_{\mathrm{tr}}(s))
  - \sum_{1\leq j\leq r} \Nc_r^j \big(w_{\mathrm{tr}}(s),\dots, w(s)-w_{\mathrm{tr}}(s), 
  \dots, w_{\mathrm{tr}}(s) \big] \Big\|_{H^5}
  \\
  & \lesssim \big\| w(s) - w_{\mathrm{tr}}(s) \big\|_{\dot{W}^{4,0}} 
  \cdot \big( \| w(s) \|_{\dot{W}^{4,0}} + \| w_{\mathrm{tr}}(s) \|_{\dot{W}^{4,0}} \big)^{r-3}
  \\
  & \qquad \times \Big[ \big\| w(s) - w_{\mathrm{tr}}(s) \big\|_{\dot{W}^{4,0}}
  \big( \| w(s) \|_{H^{r+5}} + \| w_{\mathrm{tr}}(s) \|_{H^{r+5}} \big)
  \\
  & \qquad + 
  \big\| w(s) - w_{\mathrm{tr}}(s) \big\|_{H^{r+5}}
  \big( \| w(s) \|_{\dot{W}^{4,0}} + \| w_{\mathrm{tr}}(s) \|_{\dot{W}^{4,0}} \big) \Big]
  \\
  & \lesssim \epsilon^{A/4}\epsilon^{(1-\theta)(r-3)}\Big[ \epsilon^{A/4} \cdot (\epsilon^{1-\theta} R^{1/2} + \| w_{\mathrm{tr}}(s) \|_{H^{r+5}})
  + \epsilon^{1-\theta}\| w(s) - w_{\mathrm{tr}}(s) \|_{H^{r+5}} \Big].
\end{split}
\end{align}
Next, we use $w_{\mathrm{tr}} = P_{\leq K_{\mathrm{tr}}+2}w_{\mathrm{tr}}$ and the bounds \eqref{error1} and \eqref{errprboot}
to see that
\begin{align*}
\| w(s) - w_{\mathrm{tr}}(s) \|_{H^{r+5}} &\lesssim  \|P_{\leq K_{tr}+2}[w(s)-w_{\mathrm{tr}}(s)]\|_{H^{r+5}}+\|P_{>K_{tr}+2}w(s)\|_{H^{r+5}}\\
&\lesssim R^2\epsilon^{A/4}  
    + R^{-A} \epsilon^{1-\theta} R^{1/2}.
\end{align*}
In particular $\|w_{tr}(s)\|_{H^{r+5}}\lesssim  \epsilon^{1-\theta} R^{1/2}$ (since $\epsilon^{A/4}\leq R^{-25}$ due to the assumption $\epsilon\leq R^{-\theta_0}$), so the right-hand side of \eqref{errprd5} is bounded by $C\epsilon^{A/2}R^{2}$. 
Using also \eqref{errprconc2}--\eqref{errprconc2.5} it follows that $\|\Rc\|_{H^5}\lesssim \epsilon^{A/2}R^{2}$. 
Using the assumption \eqref{error1.2'} it follows that $\|w(s)-w_{\mathrm{tr}}(s)\|_{\dot{W}^{4,0}} \lesssim \epsilon^{A/2}R^{10}$ for any $s\in[0,t]$, which implies the desired bootstrap bounds \eqref{errprbootconc}.
\end{proof}

\subsection{Renormalization}\label{extrarenorm}We define the {\it renormalized profile}
\begin{equation}\label{profile}
(a_k)_{\mathrm{tr}}(s) = (R^{1/2} \epsilon)^{-1} \cdot
  e^{iT_1(|k|^{1/2}s+\Gamma(s,k))}\cdot\widehat{ w_{\mathrm{tr}}}(T_1s,k),
\end{equation} 
for $(s,k)\in[0,1]\times\Zb_R$, where $\Zb_R:=\Zb/R$ and $\Gamma\in C_s^0L_k^\infty([0,1]\times\Zb_R)$ is a (deterministic) renormalization factor to be defined below. The equation \eqref{wwapp} becomes
\begin{equation}\label{aex1}
\partial_s(a_k)_{\mathrm{tr}}(s)=iT_1\partial_s\Gamma(s,k)(a_k)_{\mathrm{tr}}(s)+(R^{1/2}\epsilon)^{-1}e^{iT_1(|k|^{1/2}s+\Gamma(s,k))}\cdot T_1\widehat{ \mathcal{N}_{\mathrm{tr}}}(T_1s,k),
\end{equation}
where $\mathcal{N}_{tr}$ is the nonlinearity in the right-hand side of \eqref{wwapp}. Using \eqref{nonlin1} we can write
\begin{equation*}
\begin{split}
\widehat{ \mathcal{N}_{\mathrm{tr}}}(T_1s,k)&=\sum_{r=3}^A\frac{\varphi_{\leq K_{tr}}(k)}{(2\pi R)^{r-1}}\sum_{\iota_j\in\{\pm\}}\sum_{k_1,\ldots,k_r\in\Zb_R,\,\iota_1k_1+\ldots+\iota_rk_r=k}\negmedspace n_{r,\iota_1\ldots\iota_r}(\iota_1k_1,\ldots,\iota_rk_r)\prod_{j=1}^r\widehat{w_{\mathrm{tr}}}(T_1s,k_j)^{\iota_j}\\
&=\sum_{r=3}^A\frac{\varphi_{\leq K_{tr}}(k)}{(2\pi R)^{r-1}}\sum_{\iota_1k_1+\ldots+\iota_rk_r=k}^\ast n_{r,\iota_1\ldots\iota_r}(\iota_1k_1,\ldots,\iota_rk_r)(R^{1/2}\epsilon)^r\\
&\qquad\qquad\prod_{j=1}^r\Big\{[(a_{k_j})_{\mathrm{tr}}(s)]^{\iota_j}e^{-iT_1\iota_j(|k_j|^{1/2}s+\Gamma(s,k_j))}\Big\},
\end{split}
\end{equation*}
where, for simplicity of notation, $\sum\limits_{\iota_1k_1+\ldots+\iota_rk_r=k}^\ast$ denotes the sum over all the signs $\iota_1,\ldots,\iota_r\in\{+,-\}$ and all $k_1,\ldots,k_r\in\Z_R$ satisfying $\iota_1k_1+\ldots+\iota_rk_r=k$. The formula \eqref{aex1} becomes
\begin{equation}\label{aex2}
\begin{split}
\partial_s(a_k)_{\mathrm{tr}}(s)&=iT_1\partial_s\Gamma(s,k)(a_k)_{\mathrm{tr}}(s)+\sum_{r=3}^A\varphi_{\leq K_{tr}}(k)(\epsilon R^{-1/2})^{r-1}T_1\\
&\times\sum_{\iota_1k_1+\ldots+\iota_rk_r=k}^\ast i\widetilde{n}_{\iota_1\ldots\iota_r}(k_1,\ldots,k_r)e^{iT_1\Omega_r(s,k;k_1,\ldots,k_r)}\prod_{j=1}^r[(a_{k_j})_{\mathrm{tr}}(s)]^{\iota_j},
\end{split}
\end{equation}
where
\begin{equation}\label{aex3}
\begin{split}
\widetilde{n}_{\iota_1\ldots\iota_r}(k_1,\ldots,k_r)&:=-in_{r,\iota_1\ldots\iota_r}(\iota_1k_1,\ldots,\iota_rk_r)(2\pi)^{-r+1},\\
\Omega_r(s,k;k_1,\ldots,k_r)&:=s\big[|k|^{1/2}-\sum_{j=1}^r\iota_j|k_j|^{1/2}\big]+\big[\Gamma(s,k)-\sum_{j=1}^r\iota_j\Gamma(s,k_j)\big].
\end{split}
\end{equation}

We are looking for renormalization factors $\Gamma$ of the form
\begin{align}\label{defGamma}
\Gamma({t},k)=\int_0^{t}(\epsilon^2\Gamma_0(k)+T_1^{-1}\cdot\Gamma_1({t}',k))\,\mathrm{d}{t}'=\epsilon^2{t}\cdot \Gamma_0(k)+T_1^{-1}\int_0^{t}\Gamma_1({t}',k)\,\mathrm{d}{t}'.
\end{align}
where, with $n_3:=\widetilde{n}_{++-}$, 
\begin{equation}
\label{defgauge}
\Gamma_0(k) := -R^{-1}\sum_{k_1\in\Zb_R}2n_3(k_1,k,k_1)\psi(k_1)\varphi_{\leq K_{\mathrm{tr}}}(k)\varphi_{\leq K_{\mathrm{tr}}}(k_1)^2,
\end{equation} 
and $\Gamma_1({t},k)\in  C_t^0L_k^\infty([0,1]\times\Zb_R)$ will be defined implicitly in Proposition \ref{fixedrenorm} below to achieve a higher order cancellation. Then 
\begin{equation}\label{profileeqn}
\begin{aligned}
&(a_k)_{\mathrm{tr}}({t})= R^{-1/2} \epsilon^{-1}\widehat{w_{\mathrm{in}}^{\mathrm{tr}}}(k)+i\sum_{r=3}^A(\epsilon R^{-1/2})^{r-1}T_1\\
&\quad\times\sum_{\iota_1k_1+\cdots+\iota_rk_r=k}^\ast\widetilde{n}_{\iota_1\ldots\iota_r}(k_1,\ldots,k_r)\varphi_{\leq K_{\mathrm{tr}}}(k)\int_0^{t} e^{iT_1\Omega_r(t',k;k_1,\ldots,k_r)}\prod_{j=1}^r(a_{k_j})_{\mathrm{tr}}({t}')^{\zeta_j}\,\mathrm{d}{t}'
\\
&\quad+iT_1\int_0^{t}(\epsilon^2\Gamma_0(k)+T_1^{-1}\Gamma_1({t}',k))(a_k)_{\mathrm{tr}}({t}')\mathrm{d}{t}',
\end{aligned}
\end{equation} 
where the resonance factors $\Omega_r$ in \eqref{profileeqn} are defined as
\begin{equation}\label{resfactor}
\begin{split}
\Omega_r({t},k;k_1,\cdots,k_r):=\Big[|k|^{1/2}-\sum_{j=1}^r\iota_j|k_j|^{1/2}&\Big]\cdot{t}+\epsilon^2\Big[\Gamma_0(k)-\sum_{j=1}^r\iota_j\Gamma_0(k_j)\Big]\cdot{t}\\
&+T_1^{-1}\int_0^{t}\Big[\Gamma_1({t}',k)-\sum_{j=1}^r\iota_j\Gamma_1({t}',k_j)\Big]\,\mathrm{d}{t}'.
\end{split}
\end{equation}

\medskip
\subsection{Duhamel expansion, trees and couples} 
We shall expand the solution $w_{\mathrm{tr}}$ to (\ref{wwapp}), or equivalently the solution $(a_k)_{\mathrm{tr}}(t)$ to (\ref{profileeqn}), using Duhamel iterations; the terms occurring in such expansions are naturally indexed by trees, and correlations between them naturally correspond to couples, which are two trees with their leaves paired. Therefore we start by introducing these combinatorial structures.
\begin{df}\label{deftree} Define a \emph{tree} $\Tc$ to be a rooted tree with two kinds of branching nodes: the interaction or I-branching nodes, and the normal form or N-branching nodes. We assume that each I-branching node has at least $3$ and at most $A$ children nodes; moreover each N-branching node has either 2 or 3 children nodes, and all of them are leaves. Define the order $n=n(\Tc)$ (resp. $n_I=n_I(\Tc)$ and $n_N=n_N(\Tc)$) to be the number of branching nodes (resp. number of I- and N-branching nodes) and the rank $r=r(\Tc)$ to be the number of leaves. We also refer to $n_I(\Tc)$ and $r(\Tc)$ as the \emph{order} and \emph{rank} of the tree $\Tc$.

For each  tree $\Tc$, we also choose a sign $\zeta_\nf\in\{\pm\}$ for each node $\nf\in\Tc$. We require that for any (I- or N-) branching node $\nf$, the children nodes of $\nf$ that have the same sign as $\nf$ must occur to the left of those that have opposite sign as $\nf$. The sign of the root $\rf$ of $\Tc$ is called the sign of $\Tc$.

We may also turn certain leaves of a given tree $\Tc$ into pairs, such that each pair contains two leaves of opposite sign; the resulting structure is called a \emph{paired tree}, and we denote the pairing structure by $\Pc$. Finally, given two trees $(\Tc^+,\Tc^-)$ such that $\Tc^{\pm}$ has sign $\pm$ and the total number of leaves is even, we may \emph{partition} the set of leaves of $\Tc^\pm$ into two-element opposite-sign pairs, to form a \emph{couple} $\Qc$. Note that a couple $\Qc$ naturally turns each tree $\Tc^{\pm}$ into a paired tree.

For any tree, paired tree or couple, we shall use the letter $\rf$, $\nf$, $\mf$ and $\lf$ to respectively denote the root(s), arbitrary nodes, arbitrary branching nodes and arbitrary leaves. For any branching node $\mf$, denote all its children nodes by $\mf_1,\cdots,\mf_r$ from left to right. The sets $\Bc$, $\Ic$, $\Nc$ and $\Lc$ denote the sets of branching nodes, I- and N-branching nodes, and leaves.
\end{df}

\begin{df}[Admissibility]\label{defadm} 
Given a paired tree $\Tc$, we say $\Tc$ is \emph{admissible} if for any I-branching node $\nf$ 
with exactly $3$ children nodes and any two children nodes $\nf_1,\nf_2$ of $\nf$ with opposite signs, 
the subtrees rooted as $\nf_1$ and $\nf_2$ {\it must not} have their leaves completely paired.
\end{df}

\begin{df}\label{defwick} Given complex Gaussian random variables $(X_1,\cdots,X_m)$ and $\zeta_j\in\{\pm\}$ for $1\leq j\leq m$, suppose that any two $X_j$ are either equal or independent, define (inductively) the renormalized product
\begin{equation}\label{renorm}\mathfrak{R}(X_1^{\zeta_1}\cdots X_m^{\zeta_m}):=X_1^{\zeta_1}\cdots X_m^{\zeta_m}-\sum_{\mathcal{P}\neq\varnothing}\prod_{\{i,j\}\in\mathcal{P}}\mathbb{E}(X_i^{\zeta_i}X_j^{\zeta_j})\cdot\mathfrak{R}\bigg(\prod_{j\in O}X_j^{\zeta_j}\bigg),
\end{equation} where $\mathcal{P}$ runs over all pairing structures of $\{1,\cdots,m\}$, i.e. collections of disjoint two-element subsets, such that $\zeta_i=-\zeta_j$ for each $\{i,j\}\in\mathcal{P}$, and $O$ is the set of elements in $\{1,\cdots,m\}$ not in any pairing in $\mathcal{P}$. For example, we define
\[
\begin{aligned}\mathfrak{R}(XY\overline{Z}\overline{W})&=XY\overline{Z}\overline{W}-\mathbb{E}(X\overline{Z})Y\overline{W}-\mathbb{E}(X\overline{W})Y\overline{Z}-\mathbb{E}(Y\overline{W})X\overline{Z}-\mathbb{E}(Y\overline{Z})X\overline{W}\\
&+2\mathbb{E}(X\overline{Z})\mathbb{E}(Y\overline{W})+2\mathbb{E}(X\overline{W})\mathbb{E}(Y\overline{Z}).
\end{aligned}\]
\end{df}
\begin{df}\label{defdec} For any paired tree $\Tc$ or couple $\Qc$, define a \emph{decoration} of it to be a vector $(k_\nf)$, where $\nf$ runs over all nodes of $\Tc$ or $\Qc$, such that $k_{\lf}=k_{\lf'}$ for any two paired leaves $\lf$ and $\lf'$, and
\begin{equation}\label{defdec0}\iota_\mf k_\mf=\sum_{j}\iota_{\mf_j}k_{\mf_j}\end{equation} for any branching node $\mf$ of $\Tc$ or $\Qc$, where $\mf_j$ are all children nodes of $\mf$.  We call it a $k$-decoration if $k_\rf=k$ where $\rf$ is the root. For any decoration $(k_\nf)$ and any I-branching node $\mf$ with children nodes $\mf_j\,(1\leq j\leq r)$, and any time variable ${t}_\mf$ associated to $\mf$, define
\begin{equation}\label{ddefomegan}\Omega_\mf=\Omega_\mf(t_\mf,k_{\mf}; k_{\mf_1},\ldots, k_{\mf_r}):=\zeta_\mf\cdot\Omega_r({t}_{\mf},k_\mf;k_{\mf_1},\ldots,k_{\mf_r})
\end{equation} where $\Omega_r$ is defined in (\ref{resfactor}) but with the corresponding signs $\zeta_j$ replaced by $\zeta_{\mf_j}\cdot\zeta_\mf$.
\end{df}
\begin{df} For any decorated tree $\Tc$ and any decorated couple $\mathcal{Q}$ we define the integral
\begin{equation}\label{moredef1}
\begin{split}
&\mathcal{I}_{\mathcal{G}}(t,(k_\nf)_\nf):=\int_{\mathcal{D}_\mathcal{G}}\prod_{\mf\in\Ic} e^{iT_1\Omega_\mf(t_\mf)}\,\mathrm{d}{t}_\mf,\\
&\Dc_\mathcal{G}:=\big\{(t_\mf)_{\mf\in\Ic}:0<{t}_{\mf'}<{t}_\mf<{t}\text{ whenever }\mf'\mathrm{\ is\ a\ child\ of\ }\mf\big\},
\end{split}
\end{equation}
where $\mathcal{G}\in\{\mathcal{T},\mathcal{Q}\}$. For each paired tree $\Tc$ with pairing structure $\mathcal{P}$, define the expression
\begin{equation}\label{defofjt}
\begin{aligned}(\Jc_\Tc)_k({t})&:=(\epsilon R^{-1/2})^{r(\Tc)-1}(T_1)^{n_I(\Tc)}\sum_{(k_\nf)}\prod_{\nf\in\Tc}\varphi_{\leq K_{\mathrm{tr}}}(k_\nf)\prod_{\lf\in\Pc}^{(+)}\psi(k_{\lf})\prod_{\hf\in\Lc\backslash\Pc}\psi(k_\hf)^{1/2}\Re\Big(\prod_{\hf\in\Lc\backslash\Pc}g_{k_{\hf}}^{\zeta_{\hf}}\Big)
\\
&\times\prod_{\mf\in\Ic}[i\widetilde{n}_{\zeta_{\mf_1}\ldots\zeta_{\mf_r}}(k_{\mf_1},\ldots,k_{\mf_r})]^{\zeta_\mf}\prod_{\mf\in\Nc} 
\widetilde{a}_{\zeta_{\mf_1}\ldots\zeta_{\mf_r}}(k_{\mf_1},\ldots,k_{\mf_r})^{\zeta_\mf}\mathcal{I}_{\Tc}(t,(k_\nf)_\nf),
\end{aligned}
\end{equation} 
where
\begin{equation}\label{moredef2}
\widetilde{a}_{\zeta_{\mf_1}\ldots\zeta_{\mf_r}}(k_{\mf_1},\ldots,k_{\mf_r}):=a_{r,\zeta_{\mf_1}\ldots\zeta_{\mf_r}}(\zeta_{\mf_1}k_{\mf_1},\ldots,\zeta_{\mf_r}k_{\mf_r}).
\end{equation}
Here $n_I(\Tc)$ and $r(\Tc)$ are the order and rank of $\Tc$ as in Definition \ref{deftree}, $(k_\nf)$ runs over all $k$-decorations of $\Tc$, $\nf$, $\lf$ and $\hf$ run over all nodes, all paired leaves of sign $+$, and all unpaired leaves respectively, $\Ic$ and $\Nc$ denote the sets of I-branching and N-branching nodes of $\Tc$, and the Wick product $\Re$ is as in Definition \ref{defwick}.
\end{df}

\medskip
\subsection{The approximate solution}\label{Ssecwapp}
We now define the approximate solution $w_{\mathrm{app}}$, 
which is an approximation to $w_{\mathrm{tr}}$ and approximately solves the equation \eqref{wwapp} as follows: let 
\begin{equation}\label{ansatz}
(a_k)_{\mathrm{app}}({t}) = \sum_\Tc(\Jc_\Tc)_k({t}),
\end{equation} 
where $\Tc$ runs over all admissible\footnote{We emphasize that it is important in the proof that the sum is taken only over the admissible trees, see Definition \ref{defadm}. The main role of the renormalization coefficient $\Gamma_1$ is to eliminate the higher order contributions of the paired trees that are not admissible.} paired trees of rank $r(\Tc)\leq N:=\lfloor\theta^{-2}\rfloor$ that have sign $+$, and $\Jc_\Tc$ is defined in \eqref{defofjt}. Let
\begin{align}\label{ansatzprofile}
\begin{split}
& w_{\mathrm{app}}(t,x):= \frac{1}{2\pi R}\sum_{k\in\Zb_R}e^{ik\cdot x}\widehat{w_{\mathrm{app}}}(t,k),
\\
& (a_k)_{\mathrm{app}}(t) = R^{-1/2} \epsilon^{-1} \cdot
  e^{iT_1(|k|^{1/2}t + \Gamma(t,k))}\cdot\widehat{ w_{\mathrm{app}}}(T_1t,k).
\end{split}
\end{align}
(which is the same relation between $(a_k)_{\mathrm{tr}}$ and $w_{\mathrm{tr}}$ in \eqref{profile}).
We also define
\begin{equation}\label{ansatz2}
(\Jc_r)_k({t})=\sum_{r(\Tc)=r}(\Jc_\Tc)_k({t}),
\end{equation} where the sum is taken over all admissible paired trees of rank $r$ that have sign $+$, so we have
\begin{equation}\label{ansatz3}
(a_k)_{\mathrm{app}}({t})=\sum_{r=1}^N(\Jc_r)_k({t}).
\end{equation}

We remark that all these definitions depend on the choice of the second renormalization $\Gamma_1=\Gamma_1(t,k)$ defined in Section \ref{extrarenorm}; in fact, this $\Gamma_1(t,k)$ will be constructed by a fixed point equation, see Proposition \ref{fixedrenorm} below.

\subsection{Expressions $\Kc_\Qc$ for couples $\Qc$}\label{sectionkq} 
A major part of the proof in this paper is to estimate the second moment 
$\Eb|(\Jc_r)_k(t)|^2$, which naturally leads to the study of the quantity $\Kc_\Qc$ for couples $\Qc$. Namely, for any couple $\Qc$ we define
\begin{equation}\label{defofkq}
\begin{aligned}\Kc_\Qc(t,k)&:=(\epsilon R^{-1/2})^{r(\Qc)-2}(T_1)^{n_I(\Qc)}\sum_{(k_\nf)}\prod_{\lf\in\Lc}^{(+)}\psi(k_{\lf})\prod_{\mf\in\Ic}
[i\widetilde{n}_{\zeta_{\mf_1}\ldots\zeta_{\mf_r}}(k_{\mf_1},\ldots,k_{\mf_r})]^{\zeta_\mf}
\\
& \times\prod_{\nf\in\Qc}\varphi_{\leq K_{\mathrm{tr}}}
(k_\nf)\prod_{\mf\in\Nc} \widetilde{a}_{\zeta_{\mf_1}\ldots\zeta_{\mf_r}}(k_{\mf_1},\ldots,k_{\mf_r})^{\zeta_\mf}
\mathcal{I}_\Qc(t,(k_\nf)_\nf).
\end{aligned}
\end{equation}
Here $n_I(\Qc)$ and $r(\Qc)$ are the order and rank of $\Qc$,  and $(k_\nf)$ runs over all $k$-decorations of $\Qc$. Moreover $\nf$ and $\lf$ run over all nodes and all leaves with sign $+$ respectively, and $\mathcal{I}_\Qc$ is defined as in \eqref{moredef1}. Note that the definition (\ref{defofkq}) depends on the choice of $\Gamma_1$ in Section \ref{extrarenorm}; the following proposition allow us to fix a unique choice of this $\Gamma_1$.

\begin{prop}\label{fixedrenorm}
There exists a unique $\Gamma_1\in C_t^0L_k^\infty([0,1]\times \Zb_R)$ with $\|\Gamma_1\|_{C_t^0L_k^\infty}\leq 1$, 
such that
\begin{equation}\label{fixedrenormeqn}
  \Gamma_1(t,k)=\epsilon^2T_1\cdot R^{-1}\sum_{k_1\in\Zb_R}(-2)n_3(k_1,k,k_1) 
  \varphi_{\leq K_{\mathrm{tr}}}(k)\varphi_{\leq K_{\mathrm{tr}}}(k_1)^2\sum_{\Qc}\Kc_\Qc(t,k_1),
\end{equation} 
where $\Kc_\Qc$ is defined as in (\ref{defofkq}) (which depends on $\Gamma_1$ itself), 
and the summation is taken over all nontrivial admissible couples with the rank of each tree being at most $N$.
\end{prop}

\subsection{Main estimates}\label{mainproof} 
To prove Proposition \ref{apriori} 
we will need estimates of the approximate solution $w_{\mathrm{app}}$, as well as its relation to the actual solution $w$.
As Proposition \ref{properror1} already provides the estimate of $w-w_{\mathrm{tr}}$, now we only need to control the error $w_{\mathrm{tr}}-w_{\mathrm{app}}$ (or equivalently $(a_k)_{\mathrm{tr}}-(a_k)_{\mathrm{app}}$). These estimates are summarized in the following propositions.

\begin{prop}\label{propjr} 
Suppose $\Gamma_1$ is arbitrary with $\|\Gamma_1\|_{C_t^0L_k^\infty}\leq 1$.
Then, for any $1\leq r\leq N$ and any ${t}\in[0,1]$, we have
\begin{equation}\label{jrest1}
R^{-1}\sum_{k\in\Zb_R}\langle k\rangle^{10} \Eb|(\Jc_r)_k({t})|^2
  \leq \epsilon^{-2}T_1^{-1}\cdot\epsilon^{\theta r/8}.
\end{equation} 
Moreover, for any choice of $\Gamma_1$, denote the right hand side of (\ref{fixedrenormeqn}) by $\Xc[\Gamma_1]$,
then we have that
\begin{equation}\label{fixedrenorm2}
\|\Xc[\Gamma_1]\|_{C_t^0L_k^\infty}\leq \epsilon^{\theta/8} \quad \mathrm{if}\quad\|\Gamma_1\|_{C_t^0L_k^\infty}\leq 1;
\end{equation} and that
\begin{equation}\label{fixedrenorm3}
\|\Xc[\Gamma_1]-\Xc[\Gamma_1']\|_{C_t^0L_k^\infty}\leq \epsilon^{\theta/8}
  \|\Gamma_1-\Gamma_1'\|_{C_t^0L_k^\infty}\quad 
  \mathrm{if}\quad \|\Gamma_1\|_{C_t^0L_k^\infty},\|\Gamma_1'\|_{C_t^0L_k^\infty}\leq 1.
\end{equation}
\end{prop}
\begin{proof} 

This is proved in Section \ref{sec.molecule}, using the preparations in Subsection \ref{sec.L1est}
and Section \ref{sec.irre}.
\end{proof}

\begin{prop}\label{propdiff} 
Assume that the functions $u$ and $w$ satisfy the bootstrap assumptions \eqref{bootstrap} on the time interval $[0,t_0]$, $t_0\leq T_1$, and define the coefficients $(a_k)_{\mathrm{tr}}$ and $(a_k)_{\mathrm{app}}$ as before. Then, with probability $\geq 1-2e^{-R^{2\eta_0}}$, for any $t'\in[0,t_0/T_1]$ we have
\begin{align*}
R^{-1}\sum_{k\in\Zb_R}\langle k\rangle^{10}
  |(a_k)_{\mathrm{tr}}({t}')-(a_k)_{\mathrm{app}}({t}')|\leq\epsilon^{N\theta/12}
\end{align*}
Moreover, the function $w_{\mathrm{tr}}$ has the property that the bounds \eqref{error1.2'} hold for solutions of the linear equation \eqref{error1.1'}.
\end{prop}

\begin{proof} 
This is proved in Section \ref{sec.rem}.
\end{proof}

With Propositions \ref{propjr} and \ref{propdiff}, we can now prove Proposition \ref{fixedrenorm}. 
Then combining with Proposition \ref{properror1}, we can prove Proposition \ref{apriori}.

\begin{proof}[Proof of Proposition \ref{apriori}] Notice first that Proposition \ref{fixedrenorm} follows from Proposition \ref{propjr} by a fixed point argument. Then we notice that we are making the bootstrap assumptions \eqref{bootstrap}, which match with \eqref{error1}, so that the conclusion \eqref{error2'} of Proposition \ref{properror1} gives us
\begin{equation}\label{prapriori3}
\sup_{t'\in[0,t]}\|w(t')-w_{\mathrm{tr}}(t')\|_{\dot{W}^{4,0}} \leq \epsilon^{A/4}.
\end{equation}
Notice that the assumption \eqref{error1.2'} on solutions of \eqref{error1.1'}, which is needed to obtain
\eqref{error2'} is proved in Subsection \ref{ssecprerror1.2'}. 
Moreover, by the definitions \eqref{ansatz3} and \eqref{ansatzprofile}, we have, for $t\in[0,1]$,
\begin{align*}
(1-\partial_x^2)^p w_{\mathrm{app}}(T_1 t,x) = \frac{\epsilon}{2\pi R^{1/2}} \sum_{k \in \Zb_R} 
  \langle k \rangle^{2p} (a_k)_{\mathrm{app}}(t) e^{iT_1(|k|^{1/2}t + \Gamma(t,k))}
\end{align*}
which is a multilinear Gaussian sum, 
and, by \eqref{ansatz3} and the main estimates \eqref{jrest1}, we know that
\begin{align*}
& \Eb \big| (1-\partial_x^2)^p w_{\mathrm{app}}(T_1 t,x) \big|^2
\lesssim \frac{\epsilon^2}{R} \Eb \Big| \sum_{k \in \Zb_R} 
  \langle k \rangle^{5} (a_k)_{\mathrm{app}}(t) e^{ik\cdot x} e^{iT_1(|k|^{1/2}t + \Gamma(t,k))} \Big|^2
\\
& \lesssim \frac{\epsilon^2}{R} \Eb \sum_{k \in \Zb_R} 
  \sum_{r=0}^N \langle k \rangle^{10} \big| (\Jc_r)_k(t) \big|^2 
  = \frac{\epsilon^2}{R} \sum_{k \in \Zb_R} \sum_{\Qc} 
  \langle k \rangle^{10} \Kc_\Qc(t,k) \lesssim T_1^{-1},
\end{align*}
where the sums are over all possible admissible trees and couples.
Using the hyper-contractivity estimates of Proposition \ref{initialprop} (i), we obtain,
with probability $\geq 1 - e^{-R^{\eta_0}}$, 
for all $t \leq T_1$,
\begin{align}\label{prapriori4}
\sup_{t' \leq t}\| w_{\mathrm{app}}(t') \|_{\dot{W}^{4,0}} \lesssim \epsilon^{1-\theta/2} + 
  \epsilon^{-\theta} T_1^{-1/2}.
\end{align}

Finally, we recall the definitions \eqref{ansatzprofile} and \eqref{profile},
and use Proposition \ref{propdiff} to obtain, for all $t \leq T_1$,
\begin{align*}
\sup_{ t'\in[0,t]} \big\| w_{\mathrm{app}}(t') - w_{\mathrm{tr}}(t') \big\|_{\dot{W}^{4,0}}
  & \lesssim 
  \frac{\epsilon}{R^{1/2}} 
  \sup_{ t'\in[0,t/T_1]} \sum_{k \in \Zb_R} \langle k\rangle^{5} |(a_k)_{\mathrm{app}}(t') - (a_k)_{\mathrm{tr}}(t')|
  \lesssim \epsilon^{N\theta/12}.
\end{align*}
Putting this together with \eqref{prapriori4} and \eqref{prapriori3}, and in view of our choice of parameters, 
we obtain \eqref{bootstrap2} thus completing the proof of Proposition \ref{apriori} and of Theorem \ref{main}.
\end{proof}

\smallskip
\subsection{An $L^1$ estimate on time integration}\label{sec.L1est} 
In this section, we estimate the time integrals $\mathcal{I}_\Gc$ defined in \eqref{moredef1}. Note that the integrals depends on the functions $\Omega_\mf(t_\mf)$ for all I-branching nodes $\mf$, which are defined as in (\ref{ddefomegan}) and (\ref{resfactor}).
In particular, they depend on the values
\begin{equation}\label{defalpha0}
\alpha_\mf^0:=T_1\Big[\zeta_\mf|k_\mf|^{1/2}-\sum_{j=1}^r\zeta_{\mf_j}|k_{\mf_j}|^{1/2}\Big]+\epsilon^2T_1\cdot\Big[\zeta_\mf\Gamma_0(k_\mf)-\sum_{j=1}^r\zeta_{\mf_j}\Gamma_0(k_{\mf_j})\Big]
\end{equation} 
as well as the function $\Gamma_1(k,t)$.

\begin{lem}\label{timeintest} For any function $\Gamma_1$ that satisfies $\|\Gamma_1\|_{C_t^0L_k^\infty}\leq 1$, regardless of the exact choice of $\Gamma_1$, we always have
\begin{equation}\label{timeintest1}
|\Ic_\Gc(t,(k_\nf)_\nf)|\leq \Ac\big((\alpha_\mf^0)_{\mf}\big),
\end{equation} 
where $\Ac$ is an explicit function depending on $\alpha_\mf^0$ only, and satisfying the bounds
\begin{equation}\label{timeintest2}
\sum_{q_\mf^0\in \Sc_\mf}\sup_{|\alpha_\mf^0-q_\mf^0|\leq 1}\big|\Ac\big((\alpha_\mf^0)_{\mf}\big)\big|\leq (\log R)^{r(\Qc)},
\end{equation} 
where each $\Sc_\mf$ is a subset of $\Zb$ of at most $R^{20}$ elements possibly depending on $\mf$.
\end{lem}

\begin{proof} 
Using (\ref{ddefomegan}) and (\ref{resfactor}), we can write the integral $\mathcal{I}_\Gc$ as
\begin{equation}\label{defI2}
\Ic_{\Gc}=\int_{\Dc_\Qc}\prod_{\mf} e^{i\alpha_\mf^0 t_\mf}\cdot e^{i\Psi_\mf(t_\mf)}\,dt_\mf,
\end{equation} 
where for each $\mf$ the function $\Psi_\mf$ is defined by
\begin{equation}\label{defPsi}
\Psi_\mf(t) = \int_0^t\Big[\zeta_\mf\Gamma_1(t',k_\mf)-\sum_{j=1}^r\zeta_{\mf_j}\Gamma_1(t',k_{\mf_j})\Big]\,dt'.
\end{equation} By our assumption on $\Gamma_1$, we may assume that $\Psi_\mf$ is supported on $[-2,2]$, and is uniformly bounded in $C_t^1$.

Now in (\ref{defI2}) we start by integrating in $t_\mf$ for some branching node $\mf$ whose children are all leaves. 
With all the other $t_{\mf'}$ fixed, we have an integral of form
\begin{equation}\label{IBP}
\int e^{i\alpha_\mf^0 t_\mf}\cdot e^{i\Psi_\mf(t_\mf)}\,\mathrm{d}t_\mf
=\frac{-1}{\alpha_\mf^0}\int e^{i\alpha_\mf^0 t_\mf}\cdot\Psi_\mf'(t_\mf)e^{i\Psi_\mf(t_\mf)}\,\mathrm{d}t_\mf
+\frac{1}{i\alpha_\mf^0}e^{i\alpha_\mf^0t_\mf}\cdot e^{i\Psi_\mf(t_\mf)}\bigg|_{\mathrm{boundary\ values}}.
\end{equation} 
Here the boundary values are either an absolute constant, or the value $t_{\qf}$ 
for the parent node $\qf$ of $\mf$. We may assume $|\alpha_\mf^0|\geq 1$ 
as the integral in $t_\mf$ is trivially bounded in $\alpha_\mf^0$ and the integrability in $\alpha_\mf^0$ 
on bounded sets follows trivially; then in the above expression, 
the boundary terms are either constant or has the form
\begin{equation}\label{IBPbdr}
\frac{1}{i\alpha_\mf^0}e^{i\alpha_\mf^0t_{\qf}}\cdot e^{i\Psi_\mf(t_{\qf})}
\end{equation}
which will be carried to the next integration involving $t_{\qf}$, but with a decay factor $(i\alpha_\mf^0)^{-1}$.

As for the bulk term in \eqref{IBP}, since we already have the $(\alpha_\mf^0)^{-1}$ factor,
we can just fix the value of $t_\mf^0$ and view the expression 
$e^{i\alpha_\mf^0 t_\mf}\cdot\Psi_\mf'(t_\mf)e^{i\Psi_\mf(t_\mf)}$ 
as a bounded constant, then integrate in the other variables and finally take average over $t_{\mf}$.

In either cases, we have reduced the integral $\Ic$ to $(i\alpha_\mf^0)^{-1}$ times another integral 
involving one fewer time variables, with possibly the additional factor 
$e^{i\alpha_\mf^0t_{\qf}} e^{i \Psi_\mf(t_{\qf})}$ inserted into the integrand. 
But this factor only shifts the phase 
$e^{i\alpha_{\qf}^0t_{\qf}}$ to $e^{i(\alpha_{\qf}^0+\alpha_{\mf}^0)t_{\qf}}$; 
moreover the additional factor 
$e^{i\Psi_\qf(t_\mf)}$ is changed to $e^{i(\Psi_\qf(t_\qf)+\Psi_\mf(t_\qf))}$, 
but the function $\Psi_\mf+\Psi_\qf$ is still supported on $[-2,2]$ and uniformly bounded in $C^1$.

As such, we can keep integrating in the order of time variables from bottom to top (i.e. each time we integrate in some variable $t_\mf$ where $\mf$ is one of the lowest nodes that still remains), so that for each integration in $t_\mf$ the boundary only contains $t_\qf$ for $\qf$ being the parent node of $\mf$. Moreover, each time we get a denominator $(\alpha_\mf^0)^{-1}$, where this $\alpha_\mf^0$ may also be shifted by some linear combination of previous $\alpha_\rf^0$ variables. Putting together (again we may ignore the cases where $|\alpha_\mf^0|\leq 1$ etc.), we get that
\[|\Ic_\Gc(t,(k_\nf)_\nf)|\lesssim \sum_{(\beta_\mf^0)}\prod_\mf\langle \beta_\mf^0\rangle^{-1},\] where $\beta_\mf^0$ is defined from $\alpha_\mf^0$ by the following rule: 
\begin{itemize}
\item If $\mf$ has no child that is a branching node, then $\beta_\mf^0=\alpha_\mf^0$;
\item Otherwise we have $\beta_{\mf}^0=\alpha_\mf^0+\sum_{j=1}^r\iota_j\beta_{\mf_j}^0$, where $\mf_j$ are children nodes of $\mf$ and $\iota_j\in\{0,1\}$ are chosen arbitrarily.
\end{itemize}
Now, for fixed $\alpha_\mf^0$ there are only finitely many choices for $\beta_\mf^0$, and for each choice we have that
\[\sum_{q_\mf^0\in \Sc_\mf}\sup_{|\alpha_\mf^0-q_\mf^0|\leq 1}\prod_{\mf}\langle \beta_{\mf}^0\rangle^{-1}\lesssim(\log R)^{r(\Qc)},\] by taking supremum and summation $\alpha_\mf^0$ in the order of $\mf$ that is opposite from the integration in $t_\mf$, and using the elementary inequality that
\[\sup_{\#\Sc\leq R^{20},\,\Sc\subset\Zb}\sup_\beta\sum_{q\in\Sc}\sup_{|\alpha-q|\leq 1}\frac{1}{\langle \alpha+\beta\rangle}\lesssim \log R.\] This completes the proof, where we note also that $r(\Qc)$ is the total number of branching nodes $\mf$ in $\Qc$, that is, the total number of variables $\alpha_\mf^0$.
\end{proof}

\smallskip
\subsection{Base counting estimates}\label{seccount}
For $r\in\{2,3\}$ and $\zeta_j\in\{\pm\}\,(1\leq j\leq r)$, $k^0\in\mathbb{Z}_R^r$, $k_*\in\mathbb{Z}_R$, and $\beta\in\mathbb{R}$, consider the sets
\begin{equation}\label{countingset}\Sigma_r:=\bigg\{(k_1,\cdots,k_r)\in(\mathbb{Z}_R)^r:|k_j-k_j^0|\leq 1,\sum_{j=1}^r\zeta_jk_j=k_*,\bigg|\sum_{j=1}^r\zeta_j(|k_j|^{1/2}+\epsilon^2\Gamma_0(k_j))-\beta\bigg|\leq T_1^{-1}\bigg\}.
\end{equation} Then we have the following estimates for $\#\Sigma_r$:

\begin{prop}[$2$ vector counting]\label{2vecprop}
Suppose $r=2$ and $\zeta_1=-\zeta_2$, then we have
\begin{equation}\label{2vec1}\#\Sigma_2\lesssim R\langle k^0\rangle^2\cdot\min(1,T_1^{-1}\cdot |k_*|^{-1}).
\end{equation}
Suppose $r=2$ and $\zeta_1=\zeta_2$, then we have
\begin{equation}\label{2vec2}\#\Sigma_2\lesssim R\langle k^0\rangle^2T_1^{-1/2}.
\end{equation} \end{prop}
\begin{proof} Consider (\ref{2vec1}). First trivially $\#\Sigma_2\lesssim R$ because $k_1$ has only $\approx R$ choices and $k_2$ is uniquely determined by $k_1$. Now, suppose $k_*\neq 0$ and (say) $(\zeta_1,\zeta_2)=(+,-)$, then we have
\[|(|k_1|^{1/2}+\epsilon^2\Gamma_0(k_1))-(|k_1-k_*|^{1/2}+\epsilon^2\Gamma_0(k_1-k_*))-\beta|\leq T_1^{-1}.\] Recalling the definition \eqref{defgauge} we can estimate
\begin{equation*}
\big|\partial_{k_1}\{(|k_1|^{1/2}+\epsilon^2\Gamma_0(k_1))-(|k_1-k_*|^{1/2}+\epsilon^2\Gamma_0(k_1-k_*))\}
	\big|\gtrsim |k_*|\langle k^0\rangle^{-2}.
\end{equation*} 
Therefore (\ref{countingset}) implies that $k_1$ belongs to a set of measure $\lesssim T_1^{-1}\cdot |k_*|^{-1}\langle k^0\rangle^2$ (which is a union of finitely many intervals), hence we also have $\#\Sigma_2\lesssim RT_1^{-1}|k_*|^{-1}\langle k^0\rangle^{2}$.

Now consider (\ref{2vec2}), assume $\zeta_1=\zeta_2=+$, then we have
\[|(|k_1|^{1/2}+\epsilon^2\Gamma_0(k_1))+(|k_*-k_1|^{1/2}+\epsilon^2\Gamma_0(k_*-k_1))-\beta|\leq T_1^{-1}.\] Moreover $k_1$ also belongs to a fixed unit interval, and the second derivative
\[\big|\partial_{k_1}^2\{(|k_1|^{1/2}+\epsilon^2\Gamma_0(k_1))+(|k_*-k_1|^{1/2}+\epsilon^2\Gamma_0(k_*-k_1))\}\big|\gtrsim \langle k^0\rangle^{-2}.\] This implies that $k_1$ belongs a set of measure $\lesssim T_1^{-1/2}$ which is a union of finitely many intervals, and therefore $\#\Sigma_2\lesssim RT_1^{-1/2}\langle k^0\rangle^2$.
\end{proof}

\begin{prop}[$3$ vector counting]\label{3vecprop} Suppose $r=3$, then we have
\begin{equation}\label{3vec}\#\Sigma_3\lesssim R^2\langle k^0\rangle^2T_1^{-1}\log R.
\end{equation}
\end{prop}
\begin{proof} First assume the signs $\zeta_1,\zeta_2,\zeta_3$ are not all the same, say $\zeta_1=-\zeta_2$. Then with $k_3$ fixed, the number of choices for $(k_1,k_2)$ is bounded by $R\langle k^0\rangle^2T_1^{-1}|k'|^{-1}$ due to (\ref{2vec1}), where $k'$ is a shift of $k_3$ by a constant vector (either $k_*$ or $-k_*$). Summing over $k_3$ then yields (\ref{3vec}), as
\[\sum_{k_3}|k_3\pm k_*|^{-1}\lesssim\log R\] for summation in $k_3$ over any fixed unit interval.

Now assume (say) $\zeta_1=\zeta_2=\zeta_3=+$, then we have $k_1+k_2+k_3=k_*$ and
\[(|k_1|^{1/2}+\epsilon^2\Gamma_0(k_1))+(|k_2|^{1/2}+\epsilon^2\Gamma_0(k_2))+(|k_3|^{1/2}+\epsilon^2\Gamma_0(k_3))=\beta+O(T_1^{-1}).\] Let $\max(|k_1-k_2|,|k_2-k_3|,|k_1-k_3|)\approx L$, and assume $|k_1-k_2|\approx L$. Then by the same argument as in the proof of (\ref{2vec2}) in Proposition \ref{2vecprop}, and using the bound
\[\big|\partial_{k_1}\{(|k_1|^{1/2}+\epsilon^2\Gamma_0(k_1))+(|k_*-k_3-k_1|^{1/2}+\epsilon^2\Gamma_0(k_*-k_3-k_1))\}\big|\gtrsim L\langle k^0\rangle^{-2},\] we see that for fixed $k_3$, the number of choices for $(k_1,k_2)$ is bounded by $RT_1^{-1}L^{-1}\langle k^0\rangle^{2}$. Moreover by assumption we also have $|k_3-k_*/3|\lesssim L$, so the number of choices of $(k_1,k_2,k_3)$ with fixed $L$ is $\lesssim RT_1^{-1}L^{-1}\langle k^0\rangle^{2}\cdot RL=R^2T_1^{-1}\langle k^0\rangle^{2}$. 
Finally summing over $L$ yields an extra logarithmic factor, which proves (\ref{3vec}).
\end{proof}


\medskip
\section{Irregular chains}\label{sec.irre} 
In this section we define and analyze the \emph{irregular chains}, 
which were first introduced in \cite{DeHa1}. 
In particular we will exhibit an important cancellation between these structures.
In addition to the arguments in \cite{DeHa1} we need to use a $C^{1/2}$ regularity property of 
the symbols of the water waves equations after our normal form transformations,
which is quite non-trivial to verify; see \eqref{symbolcubic} and the proof of Proposition \ref{normalprop}.

\begin{df}\label{defirrechain} An \emph{irregular chain} $\Hc$ is 
a sequence of branching nodes $(\nf_0,\cdots,\nf_{q})$, such that
\begin{itemize}
\item Each $\nf_j\,(0\leq j\leq q)$ is an I-branching node of exactly $3$ children nodes, and $\nf_{j+1}$ is a child node of $\nf_j$;
\item Each $\nf_{j}$ has a child node $\mf_j$, which has opposite sign with $\nf_{j+1}$, and is paired with a child node $\pf_{j+1}$ of $\nf_{j+1}$ as leaves.
\end{itemize}
\end{df}
  \begin{figure}[h!]
  \includegraphics[scale=0.45]{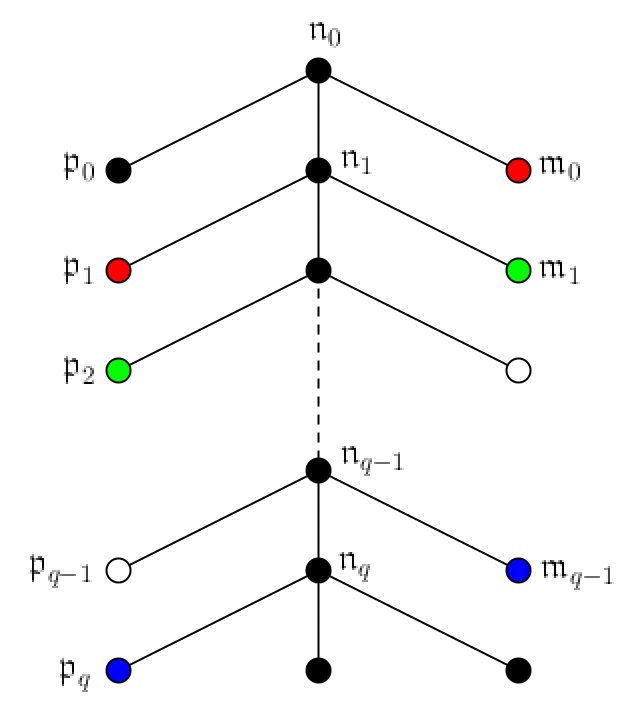}
  \caption{An example of an irregular chain (Definition \ref{defirrechain}).}
  \label{fig:irrechain}
\end{figure}

\begin{df}[Congruence relation]\label{defcong} 
Let $\Hc$ be an irregular chain as in Figure \ref{fig:irrechain}. 
By definition, we know that $\pf_j$ is a child of $\nf_j$ 
that has the same sign with $\nf_j$, thus it is either the first (left) child or the second (middle) child.
Let us fix these relative positions; then, let the sign $\zeta_{\nf_j}=\zeta_j$. 
It is clear that, given the sign $\zeta_0$ and the relative positions of $\pf_j$, 
the structure of $\Hc$ is uniquely determined by the signs $\zeta_j\,(1\leq j\leq q)$ 
(this is because, if we fix $\zeta_{j+1}/\zeta_j\in\{\pm\}$, 
then this fixes the relative position of $\nf_{j+1}$ as a child of $\nf_j$). See Figure \ref{fig:defcong}.

We define two irregular chains to be \emph{congruent} (denoted $\Hc'\equiv \Hc$), 
if the sign $\zeta_0$ and the relative position of each $\pf_j$ as a child of $\nf_j$ 
is the same for these two irregular chains. In other words, given a congruence class, 
the irregular chains $\Hc$ in this congruence class is uniquely determined by the signs 
$\zeta_j\,(1\leq j\leq q)$. For two congruent irregular chains, we also call them \emph{twists} of each other.
Finally, suppose $\Hc_j\,(1\leq j\leq r)$ are disjoint irregular chains in a given couple $\Qc$, 
then we can perform twists separately and arbitrarily at each $\Hc_j$, starting from $\Qc$; 
we define any couple $\Qc'$ formed in this way to be \emph{congruent} to $\Qc$ 
(relative to the irregular chains $\Hc_j$), denoted $\Qc'\equiv\Qc$.

\begin{figure}[h!]
\includegraphics[scale=0.22]{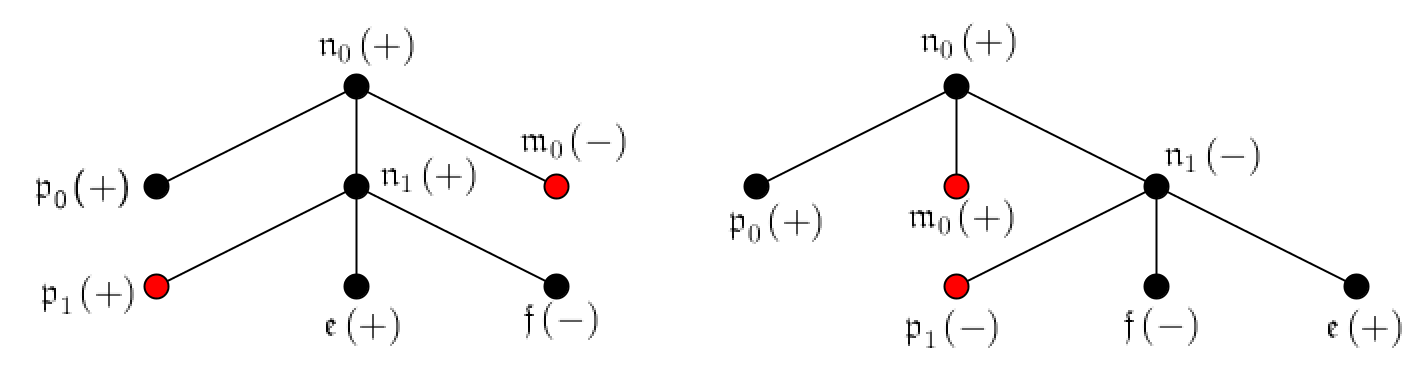}
\caption{An example of two congruent irregular chains (Definition \ref{defcong}) which are twists of each other. 
Here the signs of the involved nodes are shown in the picture.}
\label{fig:defcong}
\end{figure}
\end{df}

Now we fix one irregular chain $\Hc$ in a couple $\Qc$, and consider a decoration of $\Qc$. 
Then the restriction of this decoration to nodes in $\Hc$ can be uniquely determined by the $2q+4$ vectors $(k_j,\ell_j):0\leq j\leq q+1$, which are defined as follows: for each $0\leq j\leq q$ we have $\{k_j,\ell_j\}=\{k_{\nf_j},k_{\pf_j}\}$, and $(k_j,\ell_j)=(k_{\nf_j},k_{\pf_j})$ if $\zeta_j=+$, and $(k_j,\ell_j)=(k_{\pf_j},k_{\nf_j})$ if $\zeta_j=-$. 
For $j=q+1$ we have $(k_{q+1},\ell_{q+1})=(k_\mathfrak{e},k_{\mathfrak{f}})$ where $\mathfrak{e}$ and $\mathfrak{f}$ are the two other children nodes of $\nf_q$ other than $\pf_q$, such that $\mathfrak{e}$ has sign $+$ and $\mathfrak{f}$ has sign $-$. In fact we have the following equalities:
\begin{equation}\label{commdiff}
k_0-\ell_0=k_1-\ell_1=\cdots =k_{q+1}-\ell_{q+1}:=h,
\end{equation} 
and
\begin{equation}\label{commdiff2}
\begin{aligned}
&T_1\Omega_{\nf_j}(t)=\alpha_j^0\cdot t+G_j(t),\\&\alpha_j^0
	= T_1\cdot\big(|k_j|^{1/2}-|k_j-h|^{1/2}-|k_{j+1}|^{1/2}+|k_{j+1}-h|^{1/2}\big)
	\\
&\qquad\qquad\qquad\quad+\epsilon^2T_1
	\cdot\big(\Gamma_0(k_j)-\Gamma_0(k_j-h)-\Gamma_0(k_{j+1})+\Gamma_0(k_{j+1}-h)\big);
\\
&G_j(t)=\int_{0}^{t}\big(\Gamma_1(t',k_j)-\Gamma_1(t',k_j-h)-\Gamma_1(t',k_{j+1})+\Gamma_1(t',k_{j+1}-h)\big)\,\mathrm{d}t'.
\end{aligned}\end{equation} 
Note that these are independent of $\zeta_j$, i.e. 
they depend only on the congruence class which is invariant under twists.

\medskip
With the above discussions, we can consider an irregular chain $\Hc$ in a couple $\Qc$ and all its twists. Note that if we twist $\Hc$, then this twist can be trivially extended to $\Qc$, as long as we keep the remaining part of the couple $\Qc$ unchanged (note that the positions of $\ef$ and $\ff$ may be switched, but the sub-trees attached at $\ef$ and $\ff$ are unchanged). Consider the expression $\Kc_\Qc$ defined in (\ref{defofkq}) and the associated sub-expression where we restrict the sum to only those $k_\nf$ for $\nf$ in $\Hc$ (but excluding $\nf\in\{\nf_0,\pf_0,\ef,\ff\}$) and restrict the integration to only those $t_\mf$ for $\mf$ in $\Hc$ (but excluding $\mf=\nf_0$); denote this sub-expression by $\Kc_\Hc$. 
We will assume $\zeta_{\nf_0}=+$ without loss of generality. Then, the combination
\[\sum_{\Hc'\equiv\Hc}\Kc_{\Hc'},\] where the sum is taken over all twists $\Hc'$ of $\Hc$, can be rearranged into the following form:
\begin{align}\label{twistformula}
\begin{split}
& \sum_{\Hc'\equiv\Hc}\Kc_{\Hc'} =: \mathcal{G}\big( (k_0,\ell_0,k_{q+1},\ell_{q+1}),({t}_{\nf_0},{t}_{\ef_0},{t}_{\ff_0}) \big)
\\
& \qquad = \int_{\max({t}_\ef,t_\ff)<{t}_q<\cdots<{t}_1<{t}_{\nf_0}}
	\sum_{k_1,\cdots,k_q}\prod_{j=0}^q e^{i(\alpha_j^0\cdot t_j+G_j(t_j))} \,\mathrm{d}{t}_1\cdots\mathrm{d}{t}_q
	\cdot in_3(\ell_0,k_1,\ell_1)
\\
& \qquad \times \prod_{j=1}^q\big[ \psi(\ell_j)
	\varphi_{\leq K_{\mathrm{tr}}}(k_j) in_{3}(\ell_j,k_{j+1},\ell_{j+1})-\psi(k_j)\varphi_{\leq K_{\mathrm{tr}}}(\ell_j)in_{3}(k_j,\ell_{j+1},k_{j+1})\big].
\end{split}
\end{align}
Here note that, in the case $\zeta_{\nf_0}=-$, the factor $n_3(\ell_0,k_1,\ell_1)$ in \eqref{twistformula}
should be replaced by $n_3(k_0,\ell_1,k_1)$, but this does not affect the proof of Proposition \ref{prop.cancellation} below.

\begin{prop}[Explicit statement of the cancellation estimate]\label{prop.cancellation} 
Note that the function
\[\mathcal{G}:=\sum_{\Hc'\equiv\Hc}\Kc_{\Hc'}\] defined in \eqref{twistformula}
is a function of $(k_0,\ell_0,k_{q+1},\ell_{q+1})$ and $(t,t',t''):=(t_{\nf_0},t_\ef,t_\ff)$, where \[(t,t',t'')\in O:=\{(t,t',t''):0<\max(t',t'')<t<1\}.\] 
We may also assume each of $(k_0,\ell_0,k_{q+1},\ell_{q+1})$ belongs to a fixed unit ball.

Then there exists an extension of $\Gc$ that equals $\Gc$ on $O$, which we still denote by $\Gc$, such that
\begin{equation}\label{irrechainbound}
\big\| \mathcal{F}\big(\mathcal{G}\cdot e^{-iT_1\Omega_*(t)}\big) \big\|_{L_{(\xi,\xi',\xi'')}^1\ell_{(k_0,\ell_0,k_{q+1},\ell_{q+1})}^\infty}\leq \epsilon^{2q/3}(R\epsilon^{-2}T_1^{-1})^q,
\end{equation} where $\widehat{\mathcal{G}}$ is the Fourier transform of $\mathcal{G}$ in the $(t,t',t'')$ variables with the corresponding Fourier variables being $(\xi,\xi',\xi'')$. 
The function $\Omega_*(t)$ is defined by 
\begin{equation}\label{defOmega'}
\begin{aligned}
&T_1\Omega_*(t)=\alpha_*^0\cdot t+G_*(t),\\&\alpha_*^0
	= T_1\cdot\big(|k_0|^{1/2}-|k_0-h|^{1/2}-|k_{q+1}|^{1/2}+|k_{q+1}-h|^{1/2}\big)
	\\
&\qquad\qquad\qquad\quad+\epsilon^2T_1\cdot\big(\Gamma_0(k_0)-\Gamma_0(k_0-h)
	-\Gamma_0(k_{q+1})+\Gamma_0(k_{q+1}-h)\big);
\\
&G_*(t)=\int_{0}^{t}\big(\Gamma_1(t',k_0)-\Gamma_1(t',k_0-h)-\Gamma_1(t',k_{q+1})+\Gamma_1(t',k_{q+1}-h)\big)\,\mathrm{d}t'.
\end{aligned}\end{equation}
Finally, the bound (\ref{irrechainbound}) 
is uniform in the choice of the unit balls containing each of 
the quartets $(k_0,\ell_0,k_{q+1},\ell_{q+1})$.
\end{prop}

\begin{proof} We may assume $\zeta_{\nf_0}=+$ (the case $\zeta_{\nf_0}=-$ is the same; 
note that the twisting does not affect $\zeta_{\nf_0}$, see Definition \ref{defcong}). 
Let $h$ be defined as in (\ref{commdiff}), and let $\psi_{O}$ 
be a fixed compactly supported smooth cutoff function that equals $1$ on $O$. 
Also fix $K\gg N_0$ to be a sufficiently large constant, and note that $q\leq N_0$. 
We may form an extension of $\Gc$ by multiplying by 
$\psi_O$; 
after this we make the linear change of variables and define $\Uc(t,t_1,t_2)=\Gc(t,t+t_1,t+t_2)$. 
Then, by making the substitution $t_j=t+s_j$ in \eqref{twistformula} and using that $\alpha_*^0=\sum_j \alpha_j^0$, we have
\begin{align}\label{irrechain_0}
\begin{split}
\Uc(t,t_1,t_2) & \cdot e^{-i(\alpha_*^0\cdot t+G_*(t))} = 
	\int_{\max(t_1,t_2)<{s}_q<\cdots<{s}_1<0}
	\sum_{k_1,\cdots,k_q}\prod_{j=1}^q e^{i(\alpha_j^0\cdot s_j+G_j(s_j+t))} \mathrm{d}{s}_1\cdots\mathrm{d}{s}_q
	\\
	& \times in_3(\ell_0,k_1,\ell_1) \cdot e^{i(G_0(t)-G_*(t))}
\\
& \times \prod_{j=1}^q\big[ \psi(\ell_j)\varphi_{\leq K_{\mathrm{tr}}}(k_j) in_{3}(k_{j+1},\ell_{j+1},\ell_j)
	- \psi(k_j)\varphi_{\leq K_{\mathrm{tr}}}(\ell_j) in_{3}(\ell_{j+1},k_{j+1},k_j)\big].
\end{split}
\end{align}
Define the left hand side of \eqref{irrechain_0} to be $\Vc(t,t_1,t_2)$, then we have 
\begin{align}\label{irrechain1}
\begin{split}
\widehat{\Vc}(\xi,\xi_1,\xi_2) & = \sum_{k_1,\cdots,k_q}\prod_{j=1}^q\big[ \psi(\ell_j)\varphi_{\leq K_{\mathrm{tr}}}(k_j)
	in_{3}(k_{j+1},\ell_{j+1},\ell_j) - \psi(k_j)\varphi_{\leq K_{\mathrm{tr}}}(\ell_j) in_{3}(\ell_{j+1},k_{j+1},k_j)\big]
	\\
&\times in_3(\ell_0,k_1,\ell_1)\int_{t_1,t_2<s_q<\cdots<s_1<0;\,t\in\Rb}e^{i(\xi t+\xi_1t_1+\xi_2t_2)}\cdot\prod_{j=1}^qe^{i\alpha_j^0s_j}\cdot\prod_{j=1}^q e^{iG_j(s_j+t)}\,
\\
&\times e^{i(G_0(t)-G_*(t))}\cdot \psi \cdot \mathrm{d}t\mathrm{d}t_1\mathrm{d}t_2\prod_{j=1}^q\mathrm{d}s_j.
\end{split}
\end{align}

First assume $\max(|\xi|,|\xi_1|,|\xi_2|)\leq R^K$. In (\ref{irrechain1}), we can integrate by parts first in $(s_q,\cdots,s_1)$ in this order, then in $(t_1,t_2)$, and finally in $t$; by a similar argument to the proof of Proposition \ref{timeintest}, we can bound the absolute value of the integral in (\ref{irrechain1}) by a sum of at most $2^q$ terms which are bounded by
\begin{equation}\label{upperbound}
\prod_{j=1}^q\langle \beta_j\rangle^{-1}\cdot \langle \xi_1+\gamma_1\rangle^{-1}\cdot\langle\xi_2+\gamma_2\rangle^{-1}\cdot\langle\xi+\gamma\rangle^{-1},
\end{equation} 
where each $\beta_j$ equals $\alpha_j^0$ plus an algebraic sum of the $(\alpha_k^0)_{k<j}$, and similarly $\gamma_1$ equals an algebraic sum of $(\alpha_j^0)$, and $\gamma_2$ equals an algebraic sum of $(\alpha_j^0)$ and $\gamma_1$, and $\gamma_3$ equals an algebraic sum of $(\alpha_j^0)$ and $(\gamma_1,\gamma_2)$. Here we note that previous integration by parts may lead to factors of form $G_j'(s_j+t)$ for certain $j$ (if we choose the bulk term instead of the boundary term), but these terms come with integration in $s_j$, so while integrating by parts in $t$ we can choose the vector field $\partial_t-\sum_j\partial_{s_j}$ instead of $\partial_t$ to avoid derivatives falling on these $G_j'$ factors. In the end this will not affect the validity of (\ref{upperbound}).

Now, for each $\alpha_j^0$, we may fix $q_j^0\in\Zb$ such that $|\alpha_j^0-q_j^0|\leq 1$, and $|q_j^0|\leq R^2$; this allows us to decompose $\widehat{\Vc}=\sum_{(q_j^0)}\Wc_{(q_j^0)}$, where $\Wc_{(q_j^0)}$ is defined as in (\ref{irrechain1}) but with the vector $(k_j,h)$ restricted to the set where $|\alpha_j^0-q_j^0|\leq 1$ for each $j$. For fixed $(q_j^0)$, by (\ref{irrechain1}) and (\ref{upperbound}) we then have
\begin{align}\label{irrechain2}
\begin{split}
& \|\Wc_{(q_j^0)}\|_{L_{(\xi,\xi_1,\xi_2)}^1L_{(k_0,k_{q+1},h)}^\infty}
\\
& \qquad \lesssim
\prod_{j=1}^q\langle \beta_j\rangle^{-1}\cdot\int_{\max(|\xi|,|\xi_1|,|\xi_2|)\leq R^K}
\langle \xi_1+\gamma_1\rangle^{-1}\langle \xi_2+\gamma_2\rangle^{-1}\langle \xi+\gamma\rangle^{-1}
\, \mathrm{d}\xi\mathrm{d}\xi_1\mathrm{d}\xi_2
\\
& \qquad \times\sum_{\substack{(k_1,\cdots,k_q)\\|\alpha_j^0-q_j^0|\leq 1}}
\prod_{j=1}^q \Big| \psi(\ell_j)\varphi_{\leq K_{\mathrm{tr}}}(k_j)
	\cdot n_{3}(k_{j+1},\ell_{j+1},\ell_j) - \psi(k_j)
	\varphi_{\leq K_{\mathrm{tr}}}(\ell_j)\cdot n_{3}(\ell_{j+1},k_{j+1},k_j) \Big|;
\end{split}
\end{align} 
in \eqref{irrechain2} the $\beta_j$ can also be replaced by an algebraic sum of $(q_j^0)$. 
Now in (\ref{irrechain2}), the integral is bounded by $(\log R)^C$, 
and then the factor $\prod_{j=1}^q\langle \beta_j\rangle^{-1}$ depends only on $(q_j^0)$, 
and the summation of this factor over all choices of $(q_j^0)$ gives another $(\log R)^C$. 
As such, to bound (\ref{irrechain2}) we only need to bound the final sum in $(k_j)$.

Consider now the sum in $(k_j)$ in (\ref{irrechain2}); 
in view of the decay of $\psi$, we may restrict each $k_j$ to a unit interval $|k_j-k_j^0|\leq 1$;
then we claim that 
\begin{align}\label{C1/2bound}
\big| \psi(\ell_j) \varphi_{\leq K_{\mathrm{tr}}}(k_j)
  \cdot n_{3}(k_{j+1},\ell_{j+1},\ell_j) - \psi(k_j)\varphi_{\leq K_{\mathrm{tr}}}(\ell_j)
  \cdot n_{3}(\ell_{j+1},k_{j+1},k_j)\big|\lesssim |h|^{1/2}.
\end{align}
To see this, it suffices to recall that we are assuming that $\psi$ satisfies \eqref{databoun}, and that the symbol $n_3$ 
satisfies the bound in \eqref{symbolcubic}, 
together with the fact that $k_j-\ell_j = h$.

Putting all these together, we can bound the summation in $(k_j)$ in (\ref{irrechain2}) by $|h|^{q/2}\cdot\Af$,
where $\Af$ is the number of choices for $(k_1,\cdots,k_q)$ such that $|\alpha_j^0-q_j^0|\leq 1$ for each $j$.
Finally, to bound this $\Af$, we note that $(k_0,\ell_0)$ has been fixed; for each $j$, 
we inductively construct $(k_j,\ell_j=k_j-h)$ using the condition that $\alpha_j^0$ belongs to a fixed unit interval. 
Then, Proposition \ref{2vecprop} gives that, for each fixed $(k_{j-1},\ell_{j-1})$, 
the number of choices for $(k_j,\ell_j)$ is bounded by $R\cdot T_1^{-1/2}|h|^{-1/2}$. 
By iterating this for each $j$, we get the upper bound
\[|h|^{q/2}\cdot\Af\lesssim (R\cdot T_1^{-1/2})^{q}\lesssim (R\epsilon^{-2}T_1^{-1})^q\epsilon^{2q/3},\] 
using that $T_1\leq \epsilon^{-8/3}$, which completes the proof of (\ref{irrechainbound}) 
by summing over all $(q_j^0)$ in (\ref{irrechain2}).

Finally, assume $\max(|\xi|,|\xi_1|,|\xi_2|)\geq R^K$. In this case we can first fix the values of $(k_j,h)$ at a loss of $R^{10N}\ll R^M$, using the trivial bound
\[\|\Wc\|_{L_{(\xi,\xi_1,\xi_2)}^1L_{(k_j,h)}^\infty}
	\leq \|\Wc\|_{L_{(\xi,\xi_1,\xi_2)}^1L_{(k_j,h)}^1}=\|\Wc\|_{L_{(k_j,h)}^1L_{(\xi,\xi_1,\xi_2)}^1}
	\leq R^{10N}\|\Wc\|_{L_{(k_j,h)}^\infty L_{(\xi,\xi_1,\xi_2)}^1}\] 
for any function $\Wc$.

Note that with fixed $k_j$, the function $G_j(t)$, see \eqref{defOmega'},
 is compactly supported on $\Rb$ (by suitably extending it for $t \geq 1$)
and belongs to $C_t^1$; by Plancherel and Cauchy-Schwartz, we know that
\[\|\langle \lambda\rangle^{1/4}\widehat{G_j}(\lambda)\|_{L^1_\lambda}\lesssim1,\] 
thus we can write $G_j(t)$ as a linear superposition of functions of form $\langle\lambda\rangle^{-1/4} e^{i\lambda t}$ with summable coefficients in $\lambda$. Applying this for each $G_j$ and $(G_0,G_*)$, we can get rid of these $e^{iG_j(s_j+t)}$ factors in (\ref{irrechain1}). 
Then, by fixing all the $s_j\,(j\neq q)$, integrating in $(t_1,t_2)$, then in $s_q$ and finally integrating by parts in $t$ twice, we get that
\[|\widehat{\Vc}|\leq R^{10N}\cdot \langle \xi_1\rangle^{-1}\langle \xi_2\rangle^{-1}\cdot\langle \xi_1+\xi_2+\alpha_q^0+\gamma_1\rangle^{-1}\cdot\langle \xi+\gamma_2\rangle^{-2}\cdot\max(\langle \gamma_1\rangle,\langle \gamma_2\rangle)^{-1/4}.\] Using that $\max(|\xi|,|\xi_1|,|\xi_2|)\geq R^K$, we can easily bound the above in $L_{(\xi,\xi_1,\xi_2)}^1$ norm by $R^{-K/8}$, which overcomes the loss $R^{10N}$ and is sufficient to prove (\ref{irrechainbound}). Therefore, the proof of (\ref{irrechainbound}) is complete in all cases.
\end{proof}
\begin{df}[Splicing irregular chains]\label{defsplice} Let $\Qc$ be a couple and $\Hc$ be an irregular chain. We define the new couple $\Qc_{\mathrm{sb}}$ as follows. As in Definition \ref{defirrechain}, in $\Qc_{\mathrm{sb}}$ we remove all the nodes in Figure \ref{fig:irrechain} except $\nf_0$, $\pf_0$ and the two children nodes of $\nf_q$ other than $\pf_q$. Then we turn these two children nodes of $\nf_q$ into children nodes of $\nf_0$ (the left children node of $\nf_0$ is still $\pf_0$), and connect these four nodes to the rest of the couple as in $\Qc$. We call this the \emph{splicing} of $\Qc$ at $\Hc$. Similarly, for any set of disjoint irregular chains $\Hc_j$, we can define the splicing at these chains $\Hc_j$.
\end{df}

\begin{prop}\label{irrereduce}
Let $\Qc$ be a couple and $\Hc_j\,(1\leq j\leq r)$ be disjoint irregular chains; for each $j$ we have two nodes $\nf_{0j}$ and $\pf_{0j}$ as in Definition \ref{defirrechain}. 
Let $\Qc'\equiv \Qc$ run over all couples that are congruent to $\Qc$ relative to $\Hc_j$ (Definition \ref{defcong}). 
Moreover, in the definition of $\Kc_{\Qc'}$ in (\ref{defofkq}), we restrict (by inserting indicator functions of the vector variables $k_\nf$) that for each $j$ we have $|k_{\nf_{0j}}-k_{\pf_{0j}}|\leq \epsilon^{2-\theta}$; clearly this restriction is not affected by twisting. Let $\Qc_{\mathrm{sb}}$ be the result of splicing the irregular chains $\Hc_j$, as in Definition \ref{defsplice}. Then we have
\begin{equation}
\label{reducekq}
\begin{aligned}
\sum_{\Qc'\equiv\Qc}\Kc_{\Qc'} = \epsilon^{2q_{\mathrm{tot}}/3}
(\epsilon R^{-1/2})^{r(\Qc_{\mathrm{sb}})-2}(T_1)^{n_I(\Qc_{\mathrm{sb}})}\sum_{(k_\nf)}
\prod_{\lf\in\Lc_{\mathrm{sb}}}^{(+)}\widetilde{\psi}(k_{\lf})
\prod_{\nf\in\Qc_{\mathrm{sb}}}\varphi_{\leq K_{\mathrm{tr}}}(k_\nf)
\\
 \times\int_{\Ec_{\mathrm{sb}}}\Gc\cdot\prod_{\mf\in\Ic_{\mathrm{sb}}} e^{iT_1\Omega_\mf(t_\mf)}\,\mathrm{d}{t}_\mf.
\end{aligned}
\end{equation} 

Here $(k_\nf)$ is a $k$-decoration of $\Qc_{\mathrm{sb}}$, and all notions $(\Lc_{\mathrm{sb}},\Ic_{\mathrm{sb}},\Ec_{\mathrm{sb}})$ and $\Omega_\mf$ etc. are the same as in (\ref{defofkq}) 
but associated with the couple $\Qc_{\mathrm{sb}}$. 
Moreover $q_{\mathrm{tot}}$ is the sum of the lengths $q_j$ of the irregular chains $\Hc_j$ (cf. Definition \ref{defirrechain}).
Finally the function $\Gc$ is a function of all vector variables $(k_{\Qc_{\mathrm{sb}}})$ and time variables $(t_{\Ic_{\mathrm{sb}}})$, such that:
\begin{enumerate}
\item let $\widehat{\Gc}$ be the time Fourier transform of $\Gc$, then the norm 
\begin{equation}\label{allchains}\|\widehat{\Gc}\|_{L_{\xi_{\Ic_{\mathrm{sb}}}}^1L_{k_{\Qc_{\mathrm{sb}}}}^\infty}\lesssim 1,\end{equation} where $\xi_{\Ic_{\mathrm{sb}}}$ are the Fourier variables corresponding to $t_{\Ic_{\mathrm{sb}}}$;
\item For each irregular chain $\Hc_j$ in $\Qc$, after splicing, we get two nodes $\nf_{0j}$ and $\pf_{0j}$ in $\Qc_{\mathrm{sb}}$, see Definition \ref{defsplice}. Then, in the support of $\Gc$, we have that $|k_{\nf_{0j}}-k_{\pf_{0j}}|\leq\epsilon^{2-\theta}$ for each $j$;

\item The $\widetilde{\psi}$ is another rapidly decaying function which may be different from the $\psi$ in \eqref{defofkq}.

\end{enumerate}

\end{prop}

\begin{proof} 
We start from the expression $\sum_{\Qc'\equiv\Qc}\Kc_{\Qc'}$ 
in (\ref{reducekq}), where $\Kc_\Qc$ is as in (\ref{defofkq}) with the restriction $|k_{\nf_{0j}}-k_{\pf_{0j}}|\leq \epsilon^{2-\theta}$ for each $j$ described as above. We first sum and integrate in the variables associated with each irregular chain $\Hc_j$ (i.e. those occurring in the expressions of $\Kc_{\Hc_j}$ as in (\ref{twistformula})), and treat the other variables as fixed; note that these frozen variables are in one-to-one correspondence with the variables $(k_\nf)_{\nf\in \Qc_{\mathrm{sb}}}$ and $(t_\mf)_{\mf\in\Ic_{\mathrm{sb}}}$. Then, for each $\Hc_j$ we obtain the factor which is
\[\sum_{\Hc_j'\equiv \Hc_j}\Kc_{\Hc_j'},\] and by (\ref{irrechainbound}) in Proposition \ref{prop.cancellation}, we can reduce it to $e^{iT_1\Omega_*(t_{\nf_{0j}})}$ multiplied by some function whose time Fourier transform is bounded in the $L_\xi^1L_k^\infty$ norm, The desired expression (\ref{reducekq}) then follows by putting together these factors for each $j$. Here we note that:
\begin{enumerate}
\item All the 
symbols $\widetilde{n}(\cdots)$ and $\widetilde{a}(\cdots)$ can be absorbed by the rapidly decaying factors $\psi(k_\lf)$ 
in view of (\ref{symbolbound}), at the price of replacing $\psi$ with another rapidly decaying function $\widetilde{\psi}$.
\item The $\Gc$ factor in (\ref{reducekq}) is the product of all the factors coming from application of (\ref{irrechainbound}) for each $j$, and the bound (\ref{allchains}) follows from the tensor product of the bounds (\ref{irrechainbound}).
\item The support condition $|k_{\nf_{0j}}-k_{\pf_{0j}}|\leq\epsilon^{2-\theta}$ follows from the extra restrictions defined above.
\item Finally, it is easy to see that
\[r(\Qc)=r(\Qc_{\mathrm{sb}})+2q_{\mathrm{tot}},\quad n_I(\Qc)=n_I(\Qc_{\mathrm{sb}})+q_{\mathrm{tot}},\] so the powers of $\epsilon R^{-1/2}$ and $T_1$ in (\ref{reducekq}) and (\ref{defofkq}) exactly match in view of the right hand side of (\ref{irrechainbound}).
\end{enumerate}
The proof of Proposition \ref{irrereduce} is now complete.
\end{proof}

\section{Molecules and counting algorithm}\label{sec.molecule}
In this section we prove an estimate for the expression \eqref{reducekq}, 
which is a key step in the proof of Proposition \ref{propjr}. 
The first main goal is to reduce this estimate 
to a counting problem associated with combinatorial structures called \emph{molecules},
which are essentially directed graphs with weights and linear equations
associated to their vertices (called atoms) and edges (called bonds).
Molecules have been introduced and used in the context of the cubic NLS in \cite{DeHa1}
and then used in a series of subsequent works.
We define molecules and related concepts in Subsection \ref{ssecmolecule}.
In Subsection \ref{choiceofchain} we first form molecules from couples 
after splicing irregular chains (see Section \ref{sec.irre}),
and then, in Proposition \ref{countingred1},
we give an upper bound for the main quantity that need to be estimated in Proposition \ref{propjr}.
This upper bound depends on combinatorial/graph theoretical quantities associated to the molecule 
(such as its {\it circuit rank}),
and on a {\it counting problem} for approximate resonance equations at each atom.
To efficiently carry out this complex counting problem we use a cutting operation,
described in Subsection \ref{sec.cutting}, and perform a cutting algorithm
which decomposes molecules in relatively simpler objects, keeping track of how the counting
problem changes at each step. 
Eventually we complete the necessary counting estimates in Subsection \ref{sseccountmol}
which is the more technical part of the proof, by `removing' atoms one by one from a molecule until it becomes trivial.
The proof of the main Proposition \ref{propjr} is then completed at the end of the section.


\smallskip
\subsection{Basic definitions and notations}\label{ssecmolecule}
We begin with the general definition of a molecule and related concepts.

\begin{df}[Molecules]\label{defmol} 
A \emph{molecule} $\Mb$ is a directed graph, formed by vertices, called \emph{atoms}, 
and edges, called \emph{bonds}, where multiple and self-connecting bonds are allowed. 
The atoms are classified into two types, the I-atoms and N-atoms; 
denote by $\Mb_I$ and $\Mb_N$ the set of all I-atoms and N-atoms. 
We will write $v\in \Mb$ and $\ell\in\Mb$ for atoms $v$ and bonds $\ell$ in $\Mb$, 
and write $\ell\sim v$ if $v$ is one of the two endpoints of $\ell$. 
For a molecule $\Mb$ we define $V$ (resp. $V_I$ and $V_N$) to be the number of atoms 
(resp. I- and N-atoms), $E$ the number of bonds and $F$ the number of components. 
Define $$\chi:=E-V+F$$ 
to be its \emph{circuit rank}.
For any atom $v \in \Mb$ we define its degree $d(v)$ as the number of bonds $\ell\sim v$.
\end{df}

\begin{df}[Molecules from couples]\label{defcplmol} Given a couple $\Qc$, define the molecule $\Mb=\Mb(\Qc)$ associated with $\Qc$, as follows. The atoms of $\Mb$ are the branching nodes $\mf\in\Bc$ of $\Qc$ (including both I- and N-branching nodes). For any two atoms $\mf_1$ and $\mf_2$, we connect them by a bond if either (i) $\mf_1$ is the parent of $\mf_2$, 
or (ii) a child of $\mf_1$ is paired to a child of $\mf_2$ as leaves. We fix the direction of each bond as follows: in case (i) the bond should go from $\mf_1$ to $\mf_2$ if $\mf_2$ has sign $-$, and from $\mf_2$ to $\mf_1$ if $\mf_2$ has sign $+$; 
in case (ii) the bond should go from the $\mf_j$ whose paired child has sign $-$ to the one whose paired child has sign $+$.
Note that in case (ii) we allow $\mf_1 = \mf_2$ hence a self-connecting bond in $\Mb$.
Finally, we add one bond from the root with sign $+$ to the root with sign $-$; 
if one of the roots is a leaf then this bond should be between the other root 
and the parent node of the root paired with the leaf.
\end{df}

\begin{df}[Decorations of molecules and restricted decorations]\label{defdecmol}
Let $\Mb$ be a molecule, suppose we also fix the vectors 
$c_v\in\Zb_R$ for each $v\in\Mb$. we then define a $(c_v)$-\emph{decoration} (or just a decoration) of $\Mb$ to be a set of vectors $(k_\ell)$ for all bonds $\ell\in\Mb$, such that $k_\ell\in\Zb_R$ and
\begin{equation}\label{decmole1}
\sum_{\ell\sim v}\zeta_{v,\ell}k_\ell=c_v
\end{equation} for each atom $v\in\Mb$. Here the sum is taken over all bonds $\ell\sim v$, and $\zeta_{v,\ell}$ equals $1$ if $\ell$ is outgoing from $v$, and equals $-1$ otherwise. 
For each such decoration and each I-atom $v$, define also the {\it resonance factor}
\begin{equation}\label{defomegadec}
\Gamma_v = \sum_{\ell\sim v}\zeta_{v,\ell}(|k_\ell|^{1/2}+\epsilon^2\Gamma_0(k_\ell)).
\end{equation}
Suppose $\Mb=\Mb(\Qc)$ comes from a couple $\Qc=(\Tc^+,\Tc^-)$ as in Definition \ref{defcplmol}. Let $\rf_\pm$ be the root of $\Tc_\pm$,  we define a $k$-decoration of $\Mb$ to be a $(c_v)$-decoration where $c_v=0$ for all $v$, and $k_{\ell_0}=k$ where $\ell_0$ is the bond connecting two roots (cf. Definition \ref{defmol}).

Given any $k$-decoration of $\Qc$ in the sense of Definition \ref{defdec}, 
define a $k$-decoration of $\Mb(\Qc)$, such that for each bond $\ell$:

\begin{itemize}

\item If $\ell$ corresponds to a leaf pair in $\Qc$ between leaves $\lf$ and $\lf'$, then define $k_\ell:=k_\lf=k_{\lf'}$.

\item If $\ell$ corresponds to a branching node $\nf$ which is a child of another branching node $\mf$,
then define $k_\ell:=k_\nf$;

\item If $\ell$ corresponds to the two roots, the define $k_\ell=k$.
\end{itemize} 
It is easy to verify that this $k_\ell$ is well-defined, and gives a one-to-one correspondence between 
$k$-decorations of $\Qc$ and $k$-decoration of $\Mb(\Qc)$. Moreover for such decorations we have
\begin{equation}\label{molegammav}
\Gamma_v=-\zeta_{\nf(v)}\Omega_{\nf(v)}
\end{equation} 
for each I-atom $v$ (see Remark \ref{remG0vsG} below for the slight abuse of notation that we are making here).

Finally, given $\beta_v\in\Rb$ for each $v\in\Mb_I$ and $k_\ell^0\in\Zb_R$
for each $\ell\in\Mb$, we define a decoration $(k_\ell)$ to be \emph{restricted by} 
$(\beta_v)$ and/or $(k_\ell^0)$, if we have $|\Gamma_v-\beta_v|\leq T_1^{-1}$ 
for each I-atom $v$ and/or $|k_\ell-k_\ell^0|\leq 1$ for each bond $\ell$.
In what follows we will assume, without loss of generality, that $|k^0_\ell| \lesssim 1$
by using the fast decay of the functions $\psi=\psi(k_\lf)$ (see \eqref{databoun} and \eqref{defofkq}).
Similarly, we can also disregard the factors of $\langle k^0 \rangle$ which appear in the counting estimates
\eqref{2vec1}-\eqref{2vec2} and \eqref{3vec}.
\end{df}

\begin{df}[Counting Problem]\label{defC}
For any molecule $\Mb$, define the counting problem (abbreviated CP) denoted by $\mathfrak{C}$, to be the 
supremum in the parameters $\beta_v$ and $k_{\ell}^0$, of the number of restricted decorations 
in the sense of Definition \ref{defdecmol}. 
\end{df}

\begin{rem}[Resonance factors]\label{remG0vsG}
Note how we are only using $\epsilon^2 \Gamma_0$ to modify the frequencies of oscillations in 
\eqref{defomegadec}, instead of the full factor in the definition of $\Gamma$ in \eqref{defGamma}. 
Using only $\epsilon^2 \Gamma_0$ 
is equivalent to using the whole $\Gamma$, since the size of $\Gamma(t) - \epsilon^2 t \Gamma_0$ 
is less than $T_1^{-1}$ (see also Proposition \ref{fixedrenorm}), and therefore it does not impact the oscillation
and counting estimates.
In the rest of the argument in this section, e.g. in \eqref{molegammav},
we are then dropping the distinction between $\Gamma(t)$ and $\epsilon^2 t\Gamma_0$;
cfr. \eqref{molegammav} with \eqref{resfactor} and \eqref{ddefomegan}.
\end{rem}

\begin{rem}[Restriction on decoration after splicing]\label{remextra}
For the molecule $\Mb_{\mathrm{sb}}=\Mb(\Qc_{\mathrm{sb}})$ 
which is the result of splicing any set of irregular chains $\Hc_j$ as in Definition \ref{defsplice}, 
where the couple $\Qc_{\mathrm{sb}}$ is defined as in Definition \ref{defsplice}, 
in the definition of $\Cf$ above we shall further restrict $|k_{\nf_{0j}}-k_{\pf_{0j}}|\leq \epsilon^{2-\theta}$, 
as in Proposition \ref{irrereduce}. 
Here we note that, under the one-to-one correspondence between decorations of $\Qc_{\mathrm{sb}}$ and $\Mb_{\mathrm{sb}}$, 
the condition $|k_{\nf_{0j}}-k_{\pf_{0j}}|\leq \epsilon^{2-\theta}$
is translated to $|k_{\ell_1}-k_{\ell_2}|\leq \epsilon^{2-\theta}$ 
for some two bonds $\ell_1$ and $\ell_2$ of opposite sign at 
the I-atom $v_j\in\Mb_{\mathrm{sb}}$ corresponding to the I-branching node $\nf_{0j}\in\Qc_{\mathrm{sb}}$. 
For more details see Definition \ref{defgapchain}.
\end{rem}


\smallskip
\subsection{Choice of the irregular chains $\Hc_j$}
\label{choiceofchain} To prove the estimate for (\ref{reducekq}), 
in this subsection we will choose a set of disjoint irregular chains $\Hc_j$ in $\Qc$, and splice them as in Definition \ref{defsplice}, to form the new couple $\Qc_{\mathrm{sb}}$. 
Then, using Proposition \ref{irrereduce}, we can reduce (\ref{reducekq}) to an expression of 
the form (\ref{reducekq}) which involves the couple $\Qc_{\mathrm{sb}}$ instead of $\Qc$. 
We start with a lemma concerning chains of double bonds in molecules.

\begin{lem}\label{doublebond} 
Let $\Qc$ be a couple, and $\Mb=\Mb(\Qc)$ be the corresponding molecule as in Definition \ref{defcplmol}. Suppose there exists a chain of double bonds in $\Mb$, 
i.e. a chain of molecules $v_1,\cdots, v_q$ \emph{of degree 4}, 
and each $v_j$ and $v_{j+1}$ are connected by a double bond. Then exactly one of the followings is true:

\begin{itemize}
\item[(i)] The atoms $\nf_j\,(1\leq j\leq q)$ form an irregular chain in the sense of Definition \ref{defcong}; 
or
\item[(ii)] The atoms $\nf_j\,(1\leq j\leq r)$ and the atoms $\nf_j\,(r+1\leq j\leq q)$ 
form two irregular chains, and $\nf_r$ and $\nf_{r+1}$ each has two children nodes that form two leaf pairs.
\end{itemize}
We shall call the chain in case (i) an irregular chain of \emph{type (CL)}
and the chain in case (ii) an irregular chain of \emph{type (CN)}. 
Note that if we remove all the type (CN) chains from $\Mb(\Qc)$, then this molecule still remains connected.
\end{lem}

\begin{proof} 
This is proved in \cite{DeHa2} 
which actually only deals with the case where all branching nodes in $\Qc$ have 3 children, 
but since all the branching nodes $\nf_j$ considered here do have 3 children, the proof is the same.
\end{proof}

\begin{rem}[Structure of chains in molecules]\label{irrechaindouble} 
By Lemma \ref{doublebond} we know that (i) each irregular chain in $\Qc$ corresponds 
to a chain of double bonds in $\Mb=\Mb(\Qc)$; and (ii) each chain of double bond in $\Mb$ corresponds to either one irregular chain or two irregular chains connected by two leaf pairs. 
Moreover, for any irregular chain $\Hc$ 
and the corresponding double bond chain $\Hb$ in $\Mb$, 
if $\Qc_{\mathrm{sb}}$ is the result of splicing $\Hc$ as in Definition \ref{defsplice}, then $\Mb_{\mathrm{sb}}=\Mb(\Qc_{\mathrm{sb}})$ can be formed from $\Mb$ by merging all atoms in $\Hb$ to a single atom.
\end{rem}

\begin{df}\label{defV} Given a couple $\Qc$ and the corresponding molecule $\Mb=\Mb(\Qc)$, define $\Vs=\{\Hc_j\}$ to be the set of all maximal irregular chains (i.e. those that are not proper subsets of any other irregular chain) in $\Qc$. Note that these $\Hc_j$ must be disjoint. Subsequently, define the congruence relation $\equiv$ as in Definition \ref{defcong} associated with these $\Hc_j$.
\end{df}
Using Definition \ref{defV}, we can state the main estimate in this section as follows:

\begin{prop}\label{propcong_2} 
Consider the expression
\begin{equation}\label{defI}
\Ic(t,k) := \sum_{\Qc' \equiv \Qc} \Kc_{\Qc'}(t,k)
\end{equation} 
as in (\ref{reducekq}), 
where the congruence relation is defined using the set of irregular chains $\Vs=\{\Hc_j\}$
in Definition \ref{defV}. Then we have the following estimate
\begin{equation}\label{est_total_I}
R^{-1}\sum_k\langle k\rangle^{10}|\Ic(t,k)|\lesssim \epsilon^{-2}T_1^{-1} \cdot \epsilon^{\theta r/8}
\end{equation} 
uniformly in $t$, where $r=r(\Qc)$ is the rank of $\Qc$ (cf. Definition \ref{deftree}), 
and we assume $\Qc$ is nontrivial.
\end{prop}


\begin{df}[Gaps and tame atoms]\label{defgapchain} 
Given any irregular chain $\Hc\subset\Qc$ and \emph{any decoration} $(k_\nf)$ of $\Qc$, define the \emph{gap} of $\Hc$ to be $r:=|k_{\nf_0}-k_{\pf_0}|$ in the notations of Definition \ref{defirrechain}. We say $\Hc$ is \emph{small gap} (SG) or \emph{large gap} (LG) with respect to this decoration, if $r\leq \epsilon^{2-\theta}$ or $r>\epsilon^{2-\theta}$. Note also that $r=|k_{\ell_1}-k_{\ell_2}|$ where $\ell_1$ and $\ell_2$ are the two bonds at one endpoint of $\Hb$ other than the double bond (and $\Hb$ is the double bond chain in $\Mb$ corresponding to $\Hc$), cf. Remark \ref{remextra}.

More generally, let $\Mb$ be a molecule, $v$ be an atom and $(\ell_1,\ell_2)$ 
be two bonds at $v$ with opposite directions. Given a decoration $(k_\ell)$ of $\Mb$, 
we define the \emph{gap} of the atom $v$ at the two bonds $(\ell_1,\ell_2)$ relative to this decoration, to be
\[\rho:=|k_{\ell_1}-k_{\ell_2}|.\] 
We say the gap is \emph{large} (large gap, or LG) 
if $|\rho|\geq \epsilon^{2-\theta}$, and is \emph{small} (small gap, or SG) if $|\rho|<\epsilon^{2-\theta}$. 
If $v$ is a degree 4 I-atom, we say $v$ is \emph{tame} if $v$ is SG \emph{in two different ways}; 
that is, if the gap at $v$ at the bonds $(\ell_1,\ell_2)$ is small, then there exists $\ell_3 \sim v$
such that the gap at the bonds $(\ell_1,\ell_3)$ or at the bonds $(\ell_2,\ell_3)$ is also small.
\end{df}

\medskip
Using Definition \ref{defgapchain}, we can further reduce the quantity $\Ic$ defined in Proposition \ref{propcong_2} as follows. First, the definition of $\Ic$ in (\ref{defI}) involves summation over all decorations $(k_\nf)$ of couples $\Qc'$ where $\Qc'\equiv\Qc$; by inserting a partition of unity, for each irregular chain $\Hc_j\in\Vs$ we may fix it to be either small gap (SG) or large gap (LG) in the sense of Definition \ref{defgapchain}. Let $\Vs_0\subset\Vs$ be the set of irregular chains that are SG, then the sum (\ref{defI}) over all congruent couples can be divided into finitely many sub-sums, such that each sub-sum is taken over all couples that are congruent \emph{with respect to the irregular chains in $\Vs_0$}. We shall keep this new notation for congruence for the rest of the proof.

Then, using Proposition \ref{irrereduce}, we can reduce the quantity $\Ic$
in (\ref{defI}) to a finite sum of terms that have the same form as the right hand side of (\ref{reducekq}).
The next proposition studies such quantities and reduce their estimate 
to the counting problem for restricted decorations of the molecule $\Mb_{\mathrm{sb}}:=\Mb(\Qc_{\mathrm{sb}})$,
as defined in Definitions \ref{defdecmol} and \ref{defC}

\begin{prop}\label{countingred1} 
The quantity $\Ic$ defined in \eqref{defI} satisfies the bound
\begin{equation}\label{molbound}
R^{-1}\sum_k\langle k\rangle^{10}|\Ic(t,k)|\leq \epsilon^{-2}T_1^{-1}\cdot\epsilon^{2q_{\mathrm{tot}}/3}\cdot (R\epsilon^{-2}T_1^{-1})^{-\chi}\cdot T_1^{-(E-2V+V_N)}\cdot (\log R)^{C}\cdot \Cf,
\end{equation} 
where the relevant quantities $E,V,\chi$ etc. are as in Definitions \ref{decmole1}-\ref{defcplmol} 
and are all associated with the molecule $\Mb_{\mathrm{sb}}$, 
and $\Cf$ is (the supremum of) the number of restricted decorations for $\Mb_{\mathrm{sb}}$ 
as in Definition \ref{defdecmol}, with the extra restrictions as in Remark \ref{remextra}.
\end{prop}

\begin{proof}
As discussed above, to estimate $\Ic$ we only need to estimate each term that has the form of the right hand side of (\ref{reducekq}). First, using the bound (\ref{allchains}) for $\Gc$, we can expand $\Gc$ in (\ref{reducekq}) as a linear combination, with summable coefficients, of functions of form
\[\Gc_1(k_{\Qc_{\mathrm{sb}}})\cdot \prod_{\nf\in\Ic_{\mathrm{sb}}}e^{i\lambda_\nf\cdot t_\nf},\] where $\lambda_\nf\in\Rb$ and $\Gc_1$ is a uniformly bounded function of the $(k_\nf)_{\nf\in\Qc_{\mathrm{sb}}}$ variables. Then we integrate in the time variables $(t_\nf)_{\nf\in \Ic_{\mathrm{sb}}}$ using the bounds (\ref{timeintest1})--(\ref{timeintest2}); note that these bounds do not change with the frequency shifts $e^{i\lambda_\nf t_\nf}$ and are uniform in $(\lambda_\nf)$, which is easily seen by examining the proof of Proposition \ref{timeintest}.

Now, we may replace the integral in the $(t_\nf)_{\nf\in \Ic_{\mathrm{sb}}}$ 
variables by $\Ac((\alpha_\nf^0)_\nf)$ in \eqref{timeintest1}; 
using \eqref{timeintest2}, with a loss of at most $(\log R)^C$,
we can restrict the value of each $(\alpha_\nf^0)_\nf$ to a unit interval, which means that each 
$\Omega_\nf$ is restricted to a fixed interval of length $T_1^{-1}$. 
Moreover, using the decay of $\widetilde{\psi}(k_\lf)$ in $k_\lf$ 
in (\ref{reducekq}), we may absorb the weight $\langle k\rangle^{10}$ 
and restrict the value of $k_\lf$ for each leaf $\lf$ to a unit ball; 
subsequently we may also restrict $k_\nf$ to a unit ball for each node $\nf\in\Qc$.

Then, by Definition \ref{defdecmol}, we then know that the decoration 
$(k_\nf)_{\nf\in\Qc_{\mathrm{sb}}}$ is in one-to-one correspondence with a decoration $(k_\ell)$ of $\Mb_{\mathrm{sb}}$ which is a \emph{restricted} decoration in the sense of Definition \ref{defdecmol}. Therefore the summation of all the decorations $(k_\nf)$ and integration in all the time variables $(t_\mf)$ in (\ref{reducekq}) contributes at most $\Cf$ as in Definition \ref{defdecmol}. Here note that summing over all $k$-decorations and the summing over all $k$ is equal to summing over all decorations.

Finally we calculate the powers of $R$, $\epsilon$ and $T_1$ in \eqref{reducekq} 
and match them with the powers in \eqref{molbound};
in view of the definition \eqref{defI} it suffices to show that
\begin{equation}\label{matchpower}
(\epsilon R^{-1/2})^{r(\Qc_{\mathrm{sb}})-2}(T_1)^{n_I(\Qc_{\mathrm{sb}})} = R\epsilon^{-2}T_1^{-1}
  \cdot (R\epsilon^{-2}T_1^{-1})^{-\chi}\cdot T_1^{-(E-2V+V_N)}.
\end{equation} 
To prove (\ref{matchpower}), first note that the molecule $\Mb_{\mathrm{sb}}$ is connected (i.e. $F=1$), 
since otherwise each of the two trees forming $\Qc_{\mathrm{sb}}$ must have all its leaves paired by themselves, 
which implies that $k_\rf=k=0$ for the root $\rf$ of each tree; 
but this implies that the contribution of this term to the right hand side of 
(\ref{reducekq}) must vanish due to the symbols $\widetilde{a}_{\zeta_{\mf_1}\cdots\zeta_{\mf_r}}$ and 
$\widetilde{n}_{\zeta_{\mf_1}\cdots\zeta_{\mf_r}}$ in (\ref{defofkq}) (which enter into the $\Gc$ factor in (\ref{reducekq})) 
and the fact that they vanish for $k_\rf=k=0$ as stated in Proposition \ref{normalprop}. 
We then notice that the powers of $T_1$ match in \eqref{matchpower} since
the power on the left hand-side is $n_I(\Qc_{\mathrm{sb}}) = V_I = V-V_N$ 
and the one on the right-hand side is $\chi-1-E+2V-V_N = F-1 +V-V_N$.

To match the powers of $R$ and $\epsilon$, it suffices to show that 
\begin{equation}\label{identity_1}
r(\Qc_{\mathrm{sb}})-2 = 2\chi-2. 
\end{equation} 
To verify \eqref{identity_1} observe that, by induction, we have
\[r(\Qc_{\mathrm{sb}})-2 = \sum_{\nf\in\Bc}(r(\nf)-1)\]
where $r(\nf)$ is the number of children nodes of the branching node $\nf$;
then, note that $r(\nf)=d(v)-1$ where $v=v(\nf)$ is the atom corresponding to the branching node $\nf$,
and that 
\[ \sum_{\nf\in\Bc}(d(v)-2) = 2E - 2V = 2\chi - 2F.\]
The equality \eqref{matchpower} then follows from these observations, and  this completes the proof of \eqref{molbound}.
\end{proof}

\smallskip
Recall that we denote by $\Mb_{\mathrm{sb}}=\Mb(\Qc_{\mathrm{sb}})$ 
the molecule obtained from the couple $\Qc_{\mathrm{sb}}$, which in turn is obtained from $\Qc$ 
by splicing the irregular chains $\Hc_j$, see Definition \ref{defsplice} and Proposition \ref{irrereduce}. 
In the next proposition, we list the properties of this molecule which will be used in the proof below.

\begin{prop}[Properties of $\Mb_{\mathrm{sb}}$]\label{countingred2} 
Assume that the couple $\Qc$ occurring in the definition (\ref{defI}) of $\Ic$ is admissible
according to Definition \ref{defadm}. 
Then, the molecule $\Mb_{\mathrm{sb}}=\Mb(\Qc_{\mathrm{sb}})$ is connected and satisfies the following properties:

\smallskip
\begin{enumerate}
\item (Double bonds)
For each double bond in $\Mb_{\mathrm{sb}}$ that is SG under the decoration (see Definition \ref{defgapchain}), 
it satisfies one of the following $4$ properties:

\begin{enumerate}

\smallskip
\item One of its endpoints is an N-atom;

\smallskip
\item One of its endpoints is an I-atom \emph{of degree at least 5};

\smallskip
\item At each endpoint $v$ of this double bond, we can find two bonds of 
opposite direction that is also SG, such that \emph{exactly one} of these two bonds belong to the given double bond;
in other words each endpoint $v$ is `tame', according to Definition \ref{defgapchain}.

\smallskip
\item This double bond belongs to a set $\Yc$ of double bonds 
such that the molecule $\Mb_{\mathrm{sb}}$ \emph{remains connected} after removing all double bonds in $\Yc$.
\end{enumerate}

\smallskip
\item (Connectivity property)
For each degree 4 atom $v$ and two bonds $(\ell_1,\ell_2)$ of opposite signs at $v$, 
\emph{any other atom} must be connected to $v$ by a path that \emph{does not include $\ell_1$ or $\ell_2$}.


\end{enumerate}
\end{prop}

\begin{proof} 
{\it Proof of (1)}. Suppose we have an SG double bond $(\ell,\ell')$ between two atoms $v_1$ and $v_2$, 
such that both $v_1$ and $v_2$ are I-atoms of degree 4 (otherwise we are in case (a) or (b)). 
Now, suppose this same double bond $(\ell,\ell')$ that exists in $\Mb_{\mathrm{sb}}$ 
\emph{does not exist} in the original molecule $\Mb$ before merging the double bond chains 
(cf. Remark \ref{irrechaindouble}), then the two bonds $(\ell,\ell')$ 
must belong to two different ends of the double bond chain before merging, 
which then implies that we are in case (c), that is, $\ell$ must be SG with a bond other than $\ell'$
(see Definition \ref{defgapchain}).
Suppose instead that this double bond $(\ell,\ell')$ 
exists in the original molecule $\Mb$ and is preserved under merging the double bond chains. 
Then we shall put this double bond into the set $\Yc$, 
and it suffices to prove that $\Mb_{\mathrm{sb}}$ remains connected after removing the double bonds in $\Yc$. 
The proof is basically the same as in \cite{DeHa2};
note that by Lemma \ref{doublebond} and our choice of $\Vs_0$, any SG double bond in $\Yc$ 
must be coming from two leaf pairs (it is of type CN) as in Lemma \ref{doublebond}.
If we remove all double bonds in $\Yc$,
then the remaining molecule is still connected because all the bonds corresponding 
to parent-child relations are still present, as well as the bond connecting the two roots (cf. Definition \ref{defcplmol}).

\smallskip
{\it Proof of (2)}.
Suppose $(v,\ell_1,\ell_2)$ are as stated, and define $Z$ to be the set of atoms other than $v$ that are connected to $v$ by a path not including $\ell_1$ nor $\ell_2$. If $Z\neq\Mb_{\mathrm{sb}}$, then for \emph{any decoration}, by adding up the equations (\ref{decmole1}) for all atoms in $Z\cup\{v\}$,  we conclude that $k_{\ell_1}=k_{\ell_2}$. Since this equality holds for \emph{any decoration}, which may come from any decoration of $\Qc_{\mathrm{sb}}$, and note also that the branching node $\nf$ corresponding to $v$ has exactly 3 children, we then see that: there exist two children nodes $\nf',\nf''$ of $\nf$ with opposite signs, such that $k_{\nf'}=k_{\nf''}$ for all decorations. As the decorations of leaf pairs are free variables and the decorations of branching nodes are completely determined by them as in Definition \ref{defdec}, we conclude that the leaves in the subtrees at $\nf'$ and $\nf''$ must be completely paired, which contradicts the admissibility assumption, cf. Definition \ref{defadm}.
\end{proof}

\begin{rem}[Trees admissibility and molecule connectivity]\label{remadm1}
We saw in the proof of Proposition \ref{countingred2}  that the connectivity property in (2)
holds in view of the `admissibility' of the trees that are forming our couples (hence the molecules).
This property is important in what follows to avoid unfavorable scenarios in the cutting algorithm
and counting estimates in the rest of the section. 
Importantly, the `admissibility' of the tress is a consequence of our renormalization procedure
which involves the `gauge' $\Gamma_1$ in \eqref{fixedrenorm} (see also \eqref{defGamma}-\eqref{profileeqn}).
\end{rem}

\smallskip
\begin{lem}[Degrees and power of $T_1$]\label{lempowT1}
For each atom $v\in\Mb_{\mathrm{sb}}$ we have $d(v)\geq 4$ if $v$ is I-atom, 
and $d(v)\geq 3$ if $v$ is N-atom, where $d(v)$ is the degree of $v$ in $\Mb_{\mathrm{sb}}$. 
Moreover, for the quantity $T_1^{-(E-2V+V_N)}$ in (\ref{molbound}), we have
\begin{equation}\label{decompose_sum}
T_1^{-(E-2V+V_N)} =  \prod_{v\in \Mb_I} T_1^{-(d(v)-4)_+/2} \prod_{v\in \Mb_N} T_1^{-1/2-(d(v)-3)_+/2}.
\end{equation}  
\end{lem}

\begin{proof}
The inequality $d(v)\geq 4$ for I-atoms and $d(v)\geq 3$ for N-atoms follow directly because if $v$ 
corresponds to the branching node $\nf$ in $\Mb_{\mathrm{sb}}$, 
then $d(v)$ is equal to the number of children nodes of $\nf$ plus $1$.
In the equality \eqref{decompose_sum} we may then remove the positive part symbol; 
\eqref{decompose_sum} 
then follows from the fact that $\sum_v d(v)=2E$ and, therefore,
$E-2V+V_N = E - 2V_I - V_N = (1/2)\sum_v d(v) - 2\sum_{v \in \Mb_I}1 - \sum_{v \in \Mb_N}1$.
\end{proof}

\smallskip
\begin{df}[The main Counting Quantity]\label{defAA}
For the rest of this section we define the main Counting Quantity (abbreviated CQ) associated to a molecule 
$\Mb$ contained in $\Mb_{\mathrm{sb}}$ as
\begin{equation}\label{defA}
\Af := \Cf \cdot( R\epsilon^{-2}T_1^{-1})^{-\chi}\cdot \prod_{v\in \Mb_I}
  T_1^{-(d_0(v)-4)_+/2} \prod_{v\in \Mb_N}T_1^{-1/2-(d_0(v)-3)_+/2},
\end{equation}
where $d_0(v)$ is the (initial) degree of the atom $v$ in $\Mb_{\mathrm{sb}}$; 
in other words, we understand that even if the degree of $v$ changes in subsequent operations 
(e.g. cutting of some bonds or removal of some atoms, see later sections) 
the value of $d_0(v)$ will stay the same. $\Cf$ is the counting problem from Definition \ref{defC}
associated to $\Mb$.
\end{df}

With the above definition, in view of Lemma \ref{lempowT1} and \eqref{molbound}, we have (recall \eqref{defI})
\begin{equation}\label{molbound'}
R^{-1} \sum_k\langle k\rangle^{10} |\Ic(t,k)| \leq \epsilon^{-2}T_1^{-1} \cdot \epsilon^{2q_{\mathrm{tot}}/3}
   (\log R)^{C} \cdot \Af_{sb};
\end{equation} 
consistently with Proposition \ref{countingred1}, here $\Af_{\mathrm{sb}}$ 
is the CQ associated to the molecule $\Mb_{\mathrm{sb}}$.

\begin{rem}[About the symbols]\label{remsymbols}
Notice that in \eqref{molbound'} we are disregarding the sizes of the symbols of the interactions associated
to each node, that is, the $\widetilde{n}$ in the formula \eqref{defofkq}, and the $\mathcal{G}$ in \eqref{reducekq}.
This is natural in view of the restrictions on decorations in Definition \ref{defdecmol}, and of \eqref{allchains}.
At the same time, when the size of the frequencies is smaller than $1$, the symbols $\widetilde{n}$ can 
potentially help the estimates in view of \eqref{symbolbound}.
We are in fact going to use this property for the cubic symbol $n_{++-}$ in the proof of Proposition \ref{prop_cutincr} below.
However, we are not going to carry explicitly the dependence on this symbol throughout the proof, as 
it does not play a role elsewhere.
\end{rem}

\subsection{The cutting operation and cutting algorithm}\label{sec.cutting} 
We now define the operation of \emph{cutting}.

\begin{df}[$\alpha$- and $\beta$-cuts]\label{defcut} 
Let $\Mb$ be a molecule $v$ be an atom of degree 4. Note that $v$ must have two incoming bonds and two outgoing bonds.
Let $(\ell_1,\ell_2)$ be two bonds at $v$ and $(\ell_3,\ell_4)$ be the other two bonds.

We define the \emph{cut} at the atom $v$, relative to $(\ell_1,\ell_2)$, to be the following operation: 
we split $v$ into two atoms $v_1$ and $v_2$. Then we connect $v_1$ to the two atoms that are end points of $\ell_1$ and $\ell_2$ other than $v$, and connect $v_2$ to the two atoms that are end points of $\ell_3$ and $\ell_4$ other than $v$.
The rest of the molecule remains unchanged.
In particular in the resulting molecule, both $v_1$ and $v_2$ have degree $2$. See Figure \ref{fig:cut} for an illustration.

We define this cut to be an \emph{$\alpha$-cut}, and call the resulting atoms \emph{$\alpha$-atoms}, 
if this cutting operation does not create a new connected component. 
In this case we \emph{further require that $(\ell_1,\ell_2)$ (and thus $(\ell_3,\ell_4)$) 
are in opposite directions and $v$ is SG with respect to them.} 

If this cutting creates a new connected component, then we define it to be \emph{$\beta$-cut}, 
and call the resulting atoms \emph{$\beta$-atoms}. 
In this latter case we do not make any restrictions on the directions $\ell_j$ and the SG/LG of the atom $v$.
\end{df}

  \begin{figure}[h!]
  \includegraphics[scale=0.55]{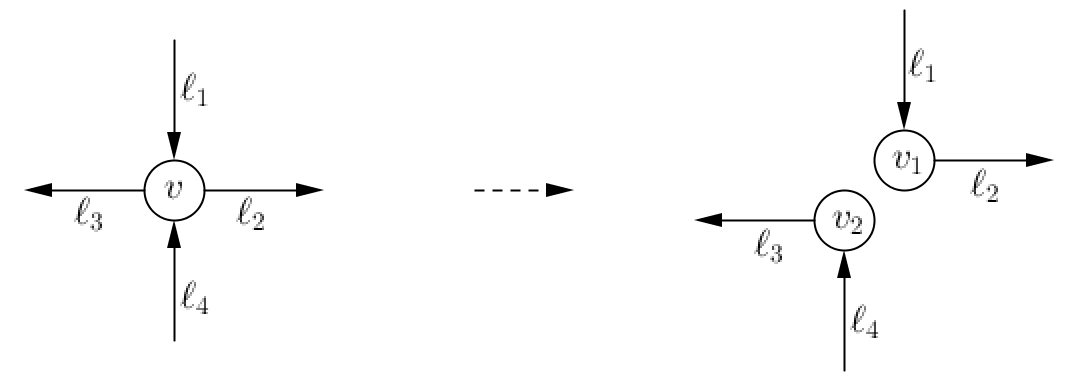}
  \caption{An example of a cut as in Definition \ref{defcut}.}
  \label{fig:cut}
\end{figure}

\smallskip
\begin{df}\label{defdeccut} In connection with Definition \ref{defcut}, 
we define the following convention for the restricted decorations of molecules with $\alpha$ and $\beta$-atoms:
\begin{itemize}

\smallskip
\item For each $\alpha$-atom $v$ with two bonds $(\ell_1,\ell_2)$, we require that both equations \eqref{decmole1} and 
$|\Gamma_v-\beta_v|\leq T_1^{-1}$ (with $\Gamma_v$ defined in (\ref{defomegadec})) hold at $v$;

\smallskip
\item For each $\beta$-atom $v$ with two bonds $(\ell_1,\ell_2)$, 
we \emph{only require} \eqref{decmole1}
holds at $v$ and do not make requirements about $\Gamma_v$;
\item Nevertheless, if we naturally turn all the $\beta$-atoms into pairs 
(where the two $\beta$ atoms created in the same $\beta$ cutting are naturally paired), then in the restricted decoration we also require that $\Gamma_v+\Gamma_{v'}=(\mathrm{const})+O(T_1^{-1})$ for any \emph{pair} of $\beta$-atoms $v$ and $v'$, 
where $\Gamma_v$ is defined in (\ref{defomegadec}).

\end{itemize}

\smallskip
In what follows, when we talk about restricted decorations of molecules that contain $\alpha$- and $\beta$-atoms 
(i.e. formed after cutting in Definition \ref{defcut}), we always take the interpretation as described above, 
and define the quantities $\Cf$ and $\Af$, as in Definitions \ref{defC} and \ref{defAA}, accordingly.
\end{df}


Now we describe the cutting algorithm. 
Start with the molecule $\Mb_{\mathrm{sb}}=\Mb(\Qc_{\mathrm{sb}})$. 
Recall the subset $\Yc$ of double bonds in $\Mb_{\mathrm{sb}}$ defined in Proposition \ref{countingred2} (1d).

\begin{itemize}

\smallskip
\item[] {\it Step 1}: for each double bond in $\Yc$ such that both endpoints are degree 4 I-atoms,
we remove it (so that both endpoints become degree 2 atoms).
We call each of the two endpoints an $\alpha$-atom;
note that this is consistent with Definition \ref{defcut} and Proposition \ref{countingred2}. 

\smallskip
\item[] {\it Step 2}: For each of the remaining degree 4 I-atoms \emph{that is SG and not tame} 
(cf. Definition \ref{defgapchain}), we perform a cutting operation at this atom along the two bonds that are SG;
by this we mean that if $\ell_1$ and $\ell_2$ are SG at such a degree 4 I-atom, we cut $v$ into
$v_1$ and $v_2$, so that $\ell_1, \ell_2 \sim v_1$.
Repeat this process until there are no more degree 4 I-atoms that are SG and not tame.

\smallskip
\item[] {\it Step 3}: For each of the remaining degree 4 I-atoms \emph{that are not tame} (and thus are LG),
where it is possible to perform a \emph{$\beta$-cut}, we then perform a $\beta$-cut at this atom.
Repeat this process until 
this is no longer possible.
\end{itemize}

\smallskip
The next proposition will summarize how the counting quantity $\mathfrak{A}$ defined in \eqref{defA} changes 
with each operation in Steps 1--3 defined above.

\begin{prop}[Change in the Counting Quantity]\label{prop_cutincr} 
Consider each of the Steps 1--3 defined above; for each of them, 
consider the value of $\mathfrak{A}$ for the molecule before and after this step, 
which we denote by $\mathfrak{A}_{\mathrm{pre}}$ and $\mathfrak{A}_{\mathrm{pos}}$. 
Then we have
\begin{equation}\label{cutincr}
\mathfrak{A}_{\mathrm{pre}}\lesssim\mathfrak{A}_{\mathrm{pos}}
  \cdot \epsilon^{\theta |\Yc|}
  \cdot \prod_{(\alpha)}(\lambda^{1/2}\epsilon^{2-\theta}T_1)\cdot (\epsilon^{-2+\theta}T_1^{-1})^{V_\beta-2\Delta F}
\end{equation} 
where $|\Yc|$ is the number of edges in $\Yc$,
$\prod_{(\alpha)}$ is taken over all $\alpha$ atoms,
(and each $\alpha$-atom is associated with a gap $\lambda\leq\epsilon^{2-\theta}$),
$V_\beta$ is the number of $\beta$-atoms generated by this step, 
and $\Delta F$ is the increment of the number of components $F$ under after this step.
\end{prop}

Similarly to $\Delta F$, we will use $\Delta E$, $\Delta V$, $\Delta \chi$ to denote the increment
of the respective quantities after each step in the cutting algorithm.

\begin{rem}[SG at $\alpha$-atoms]
Note that each $\alpha$-atom in the product in \eqref{cutincr} 
is associated with a small  gap $\lambda\leq\epsilon^{2-\theta}$.
Indeed, if two $\alpha$-atoms $v_1$ and $v_2$ are created in Step 1 
this means that the double bond existing between them before the `removal' operation
belonged to the set $\Yc$ 
so it is SG by definition; therefore, the remaining two bonds at 
$v_1$ and $v_2$ must also be SG since \eqref{decmole1} holds at both ends with $c_{v_1}=c_{v_2}=0$.
If instead $\alpha_1$ and $\alpha_2$ are created in Step 2, the gap is small by the definition of the operation.
\end{rem}

\begin{proof}[Proof of Proposition \ref{prop_cutincr}] 
We consider each of the Steps 1--3 and first note that, by the definition \eqref{defA}, we have
\begin{align}\label{Aratio}
\frac{ \mathfrak{A}_{\mathrm{pre}} }{ \mathfrak{A}_{\mathrm{pos}} }
  \lesssim \frac{ \Cf_{\mathrm{pre}} }{ \Cf_{\mathrm{pos}} } \cdot( R\epsilon^{-2}T_1^{-1})^{\Delta \chi}
\end{align}
since we are only dealing with degree 4 I-atoms.

\begin{enumerate}

\smallskip
\item In Step 1, we cut out a double bond in $\Yc$. By Lemma \ref{doublebond}, after cutting out this double bond, the molecule still remains connected; therefore after this cutting we have
\[\Delta E=-2,\quad \Delta V=0,\quad \Delta F=0,\quad V_\beta=0,\] and thus $\Delta\chi=-2$.

Now, to fix the decoration of the molecule $\Mb_{\mathrm{pre}}$, 
we need to fix the two vectors $(k_{\ell_1},k_{\ell_2})$ corresponding to the edges $(\ell_1,\ell_2)$ in this double bond. 
By the SG assumption, since one of them will have $\lesssim R$ choices, 
the other one will have $\lesssim R\lambda$ choices with gap $\lambda\leq \epsilon^{2-\theta}$. 
This gives $\lesssim R^2\lambda$ choices for $(k_{\ell_1},k_{\ell_2})$, and once they are fixed, 
we are reduced to a counting problem for $\Mb_{\mathrm{pos}}$, and thus
\[\mathfrak{C}_{\mathrm{pre}}\leq R^2\lambda\cdot\mathfrak{C}_{\mathrm{pos}}.\]
By \eqref{Aratio} we obtain
\begin{align}
\mathfrak{A}_{\mathrm{pre}}/\mathfrak{A}_{\mathrm{pos}}
   \lesssim R^2 \lambda \cdot ( R\epsilon^{-2}T_1^{-1})^{-2} = (\lambda^{1/2} \epsilon^{2-\theta} T_1 )^2 \epsilon^{2\theta},
\end{align}
which is exactly what we need for \eqref{cutincr}
since this cutting produced two $\alpha$ atoms and removes a double bond (hence, two edges) in $\mathcal{Y}$.

\smallskip
\item In Steps 2 and 3, each time we are performing one cut at one atom. 
If this cut is a $\beta$-cut, let the two edges involved be $\ell$ and $\ell'$. 
By definition, after performing the cut, the molecule breaks into two connected components (so $\Delta F=1$); 
by adding the equation \eqref{decmole1} for \emph{all atoms in one of these components}, 
we conclude that the value $k_{\ell} \pm k_{\ell'}$ is \emph{fixed} 
for any decoration (where the $\pm$ sign depends on the 
directions of $(\ell,\ell')$ at this $\beta$ cut). 
In view of the convention for decorations in Definition \ref{defdeccut},
this means that the counting problem for $\Mb_{\mathrm{pre}}$ and $\Mb_{\mathrm{pos}}$ 
are equivalent to each other, hence $\mathfrak{C}_{\mathrm{pre}}\leq\mathfrak{C}_{\mathrm{pos}}$, 
and thus $\mathfrak{A}_{\mathrm{pre}}\leq\mathfrak{A}_{\mathrm{pos}}$, noting that $\Delta\chi=0$ after this cut
as $\Delta E=0$, $\Delta V=1$ and $\Delta F=1$ and $V_\beta=2$.


Alternatively, if this cut is an $\alpha$-cut, then we have $\Delta\chi=-1$. 
Let us call $v$ the atom in $\Mb_{\mathrm{pre}}$ and $\alpha_1,\alpha_2$ the two $\alpha$-atoms 
in $\Mb_{\mathrm{pos}}$ generated by the cut. 
Let the bonds at $v$ be $\ell_1,\ell_1'$ and $\ell_2,\ell_2'$ so that $\ell_j,\ell_j' \sim \alpha_j$, $j=1,2$
and recall that $(\ell_j,\ell_j')$ are SG. Let us also assume, without loss of generality, that $|\ell_1| \gtrsim |\ell_1'|$.
%
%
In this case we shall fix the value of $k_{\ell_1}-k_{\ell'_1}$ so that $|k_{\ell_1}-k_{\ell_1'}| \lesssim \lambda$,  
and fix the interval of length $T_1^{-1}$ that contains 
$\Phi(\ell_1,\ell_1') := |k_{\ell}|^{1/2} + \epsilon^2\Gamma_0(k_\ell) - (|k_{\ell'}|^{1/2} + \epsilon^2\Gamma_0(k_{\ell'}))$; 
making these choices reduces the counting problem for $\Mb_{\mathrm{pre}}$ to that of $\Mb_{\mathrm{pos}}$.
The number of choices for $k_{\ell}-k_{\ell'}$ are $\lesssim R\lambda$.
Once this choice has been made the value of $\Phi(\ell_1,\ell_1')$ lies in a union of intervals of length $T_1^{-1}$
whose total length is $\lesssim  |\Phi(\ell_1,\ell_1')| + T_1^{-1} \lesssim 
\lambda (\sqrt{|k_{\ell_1}|} + \sqrt{|k_{\ell_1^\prime}|})^{-1} + T_1^{-1}$.
This gives approximately $T_1 \lambda |k_{\ell_1}|^{-1/2}$ choices. 
To compensate for the (possibly large) factor of $|k_{\ell_1}|^{-1/2}$, we use the vanishing 
of the symbols as stated in \eqref{symbolbound} in Proposition \ref{normalprop}; see 
also the explanation in Remark \ref{remsymbols}.
Therefore, taking into account the size of the symbol, 
we get that $\mathfrak{C}_{\mathrm{pre}} \lesssim \mathfrak{C}_{\mathrm{pos}} \cdot R\lambda \cdot \lambda T_1,$
and thus, from \eqref{Aratio},
\begin{align}\label{prop_cutincrpr10}
\mathfrak{A}_{\mathrm{pre}} / \mathfrak{A}_{\mathrm{pos}} \lesssim 
  (R\epsilon^{-2} T_1^{-1})^{-1} \cdot  R\lambda^2T_1 = \lambda^2\epsilon^{2}T_1^2.
\end{align}
Recalling that the gap $\lambda\leq\epsilon^{2-\theta}$
we see that \eqref{prop_cutincrpr10} is better than what we need for \eqref{cutincr}, 
which is $(\lambda^{1/2}\epsilon^{2-\theta}T_1)^2$,
since this cut generates two $\alpha$-atoms. \qedhere
\end{enumerate}

\end{proof}

\smallskip
After completing the above cutting operations, the molecule $\Mb_{\mathrm{sb}}=\Mb(\Qc_{\mathrm{sb}})$
is then reduced to a new molecule $\Mb'$. The properties of $\Mb'$ are described in the following proposition.

\begin{prop}[Properties of the reduced molecule]\label{prop_m0}
The molecule $\Mb'$ may contain multiple connected components, let $\Mb_0$ be any such component. 
All the atoms in $\Mb_0$ that are not $\alpha$- or $\beta$-atoms are called $\varepsilon$-atoms.

\begin{enumerate}

\smallskip
\item Unless $\Mb'$ has only one component, 
each component $\Mb_0$ contains at least one $\beta$-atom. 
Moreover, no component $\Mb_0$ is formed by one SG double bond only.

\smallskip
\item Each degree 4 I-atom in $\Mb_0$ (which must be an $\varepsilon$-atom) that is not tame must have
LG in the sense of Definition \ref{defgapchain}. 
Moreover, there is \emph{no} degree 4 I-atom in $\Mb_0$ which is not tame and 
at which one can perform a $\beta$-cut in the sense of Definition \ref{defcut}.

\smallskip
\item Suppose $\Mb_0$ contains at least one $\varepsilon$-atom. Then all the $\alpha$ and $\beta$-atoms form several chains; for each chain, either it has two endpoints at two different $\varepsilon$-atoms, or its two endpoints coincide at one $\varepsilon$-atom, and in the latter case this $\varepsilon$-atom must be an N-atom, or tame, or has degree at least 5.
\end{enumerate}

\end{prop}

\begin{proof} 
(1) Note that all the cuts in Step 1 do not create new component, by Proposition \ref{countingred2}. 
Now, unless there is no $\beta$ cut (in which case no new component is formed), each new component that is created must contain a $\beta$-atom which is formed precisely at the cut when this new component is separated from its complement. 

Next, note that the SG double bonds in $\Yc$ (cf. Proposition \ref{countingred2}) have been removed in Step 1. 
For the other SG double bonds in $\Mb_{\mathrm{sb}}$, either one of its endpoints 
is an N-atom or has degree at least 5 so it will not be cut in Steps 1--3, 
or both its endpoints are tame due to Proposition \ref{countingred2} (1c), 
so neither of them will be cut in Steps 1--3. 
In any case this double bond will not be separated from its complement.

\smallskip
(2) This follows from the definition of Steps 2 and 3; 
if any such degree 4 I-atom $v$ exists, then this will have been cut either in Step 2 or Step 3 
and will not remain degree 4 in $\Mb'$.

\smallskip
(3) Note that each $\alpha$- or $\beta$-atom must have degree 2, thus they must form several chains. 
Each such chain must have two endpoints at two $\varepsilon$-atoms. 
Assume that the endpoints of a chain of $\alpha$- and $\beta$-atoms coincide
at an $\varepsilon$-atom, and call this atom $v$.
If $v$ is not an $N$-atom or does not have degree more than 5, it is a degree 4 I-atom; 
then, by (2) above, $v$ must be tame, because otherwise it would be a non-tame degree 4 I-atom at which
we can perform a $\beta$-cut to disconnect the chain.
\end{proof}

\subsection{Counting estimates for reduced molecules}\label{sseccountmol}
Here is the main result of this section.

\begin{prop}\label{comp_est_2} 
For each of the connected component $\Mb_0$ of the reduced molecule $\Mb'$ after the above `cutting' steps, 
consider the counting problem associated with $\Mb_0$. 
Then for each $\Mb_0$ we have
\begin{equation}\label{comp_est}
\mathfrak{A}\lesssim \epsilon^{\theta E/4}
  \prod_{(\alpha)}(\lambda^{-1/2}\epsilon^{-2+\theta}T_1^{-1})
  \cdot(\epsilon^{-2+\theta}T_1^{-1})^{2G-V_\beta},
\end{equation} 
where the product is taken over all $\alpha$-atoms, and $V_\beta$ is the number of $\beta$-atoms in $\Mb_0$. 
Finally $E$ is the number of edges in $\Mb_0$, and $G$ equals $1$ 
for each component $\Mb_0$ if $\Mb'$ contains at least two components, and $G=0$ otherwise.
\end{prop}

Before proving Proposition \ref{comp_est_2} we want to prove a bound
concerning the increment of $\Af$ under the 
operation of {\it removing one single atom} at a time from $\Mb_0$ (or intermediate molecules).
We are going to need the following simple lemma about restricted decorations of $\Mb_0$:

\begin{lem}\label{lemGammav}
Recall the convention about restricted decorations of molecules with $\alpha$ and $\beta$-atoms
as in Definition \ref{defdeccut}. Then, for any $\beta$-atom $v$ in $\Mb_0$, 
the value of $\Gamma_v$ (cf. (\ref{defomegadec})) must equal a fixed algebraic sum 
of $\Gamma_{v'}$ for \emph{some set of N-atoms} $v'$, plus a constant, up to error $O(T_1^{-1})$.
\end{lem}

We postpone the proof of Lemma \ref{lemGammav} to the end of this subsection,
and proceed to prove the following:

\begin{prop}[Atom removal]\label{comp_est_0} 
Consider the process of reducing $\Mb_0$, as defined in Proposition \ref{prop_m0}, 
where we remove one atom each time. 
For each step in this process, consider the $\Af$ value before and after this step,
which we denote by $\Af_{\mathrm{pre}}$ and $\Af_{\mathrm{pos}}$. 
Note that in the definition \eqref{defA} of $\Af$, 
our convention is to keep the value of $d_0(v)$ fixed 
under the reduction of $\Mb_0$, so that it remains the initial degree 
of the atom $v$ in $\Mb_0$ (and $\Mb_{\mathrm{sb}}$).

Suppose we remove an atom $v_0$ in the given step, let $d_0(v_0)$ be the degree of $v_0$ in $\Mb_0$, 
and $d(v_0)$ be the degree of $v_0$ in $\Mb_{\mathrm{pre}}$. 
We also assume that:

\begin{enumerate}

\item[(i)] if $v_0$ is an I-atom, then either $d_0(v_0)=4$ and $\Mb_{\mathrm{pre}}=\Mb_0$ 
(i.e. we are at the first operation step) or $\Mb_{\mathrm{pre}}$ contains no N-atoms at all; and 

\item[(ii)] if $v_0$ is an I-atom and $d_0(v_0)=4$, then it must be either LG or tame in $\Mb_0$. 
\end{enumerate}

\smallskip
Then, if $\Delta\chi := \chi_{\mathrm{pos}} -  \chi_{\mathrm{pre}}= 0$ in this step,
we have $\Af_{\mathrm{pre}}\leq\Af_{\mathrm{pos}}$. 

If, instead, $\Delta\chi<0$ we have
\begin{equation}\label{removect}
\Af_{\mathrm{pre}}\leq\Af_{\mathrm{pos}}\cdot\epsilon^{\theta d/2}\cdot\mu,
\end{equation} 
where $d$ is the degree of $v_0$ in the molecule $\Mb_{\mathrm{pre}}$ before this step, 
and $\mu = 1$ if

\begin{enumerate}
\item $v_0$ is an N-atom and $d(v_0)=d_0(v_0)=3$, or

\smallskip
\item $v_0$ is an I-atom and $d(v_0)\in\{4,5\}$ and $d_0(v_0)=5$, or

\smallskip
\item $v_0$ is an I-atom, $d_0(v_0)=4$, and either $d(v_0)<4$ 
or $\Mb_{\mathrm{pre}}$ has a $\beta$-cut for $\Mb_{\mathrm{pre}}$, 

\end{enumerate} 
whereas $$\mu = \epsilon^{-2}T_1^{-1}
$$  
otherwise.

\end{prop}

\smallskip
\begin{proof} 
First, let us record that 
\begin{equation}\label{prcompest1}
\frac{\Af_{\mathrm{pre}}}{\Af_{\mathrm{pos}}} \leq  
  \frac{ \Cf_{\mathrm{pre}} }{ \Cf_{\mathrm{pos}} } \cdot( R\epsilon^{-2}T_1^{-1})^{\Delta \chi}
  \prod_{v \in (\Mb_\mathrm{pre})_I \smallsetminus (\Mb_\mathrm{pos})_I } T_1^{-(d_0(v)-4)_+/2}
  \prod_{v\in (\Mb_\mathrm{pre})_N \smallsetminus (\Mb_\mathrm{pos})_N }T_1^{-1/2-(d_0(v)-3)_+/2},
\end{equation}
and recall that $T_1=\epsilon^{-8/3+\theta_0}$ with $\theta_0 \gg \theta$ 
(which will be used in the numerology calculations below). 
Let $d:=d(v_0)$ and $d_0:=d_0(v_0)$. 
Suppose the connected component containing $v_0$ becomes $q$ connected components after removing $v_0$, 
then $1\leq q\leq d$ and $\Delta\chi=q-d$ since $\Delta E = -d$, $\Delta V=-1$, and $\Delta F = q-1$.
If $q=d$ (i.e. $\Delta\chi=0$) then each of these new components has exactly one bond $\ell$ connecting to $v_0$, 
and by adding the equations \eqref{decmole1} for all atoms in each component, 
we get that the value of each $k_\ell$ must be fixed.
This implies that $\Cf_{\mathrm{pre}}\leq\Cf_{\mathrm{pos}}$ and thus $\Af_{\mathrm{pre}}\leq\Af_{\mathrm{pos}}$ as desired.

\smallskip
Now assume $\Delta\chi<0$.
Consider the $q$ components after removing $v_0$, and for the $j$-th component, consider the set of bonds connecting atoms in this component to $v_0$, let this set be $A_j$ and let $p_j := |A_j|$, so that $\sum_j p_j=d$. 
By adding the equations \eqref{decmole1} for all atoms in the $j$-th component, we get
\begin{equation}\label{decmole1_1}
\sum_{\ell\in A_j} \zeta_{v,\ell} k_\ell=(\mathrm{const}), \qquad j=1,\dots,q.
\end{equation} 

If $v_0$ is N-atom, then \eqref{decmole1_1} is the only equation(s) we will rely on. 
In this case the total number of solutions to \eqref{decmole1_1} is $\prod_j R^{p_j-1} = R^{d-q}$
and thus $\Cf_{\mathrm{pre}} \leq \Cf_{\mathrm{pos}} R^{d-q}$;
therefore by \eqref{prcompest1} we have
\begin{align}\label{prcompest2}
\Af_{\mathrm{pre}} /  \Af_{\mathrm{pos}} \leq
  R^{d-q} \cdot (R\epsilon^{-2} T_1^{-1})^{q-d}\cdot T_1^{-(d_0-2)/2}
  = (\epsilon^{2} T_1)^{d-q} \cdot T_1^{-(d_0-2)/2}.
\end{align}
If $q=1$ and $d=d_0=3$, the above right-hand side is $\epsilon^4 T_1^{3/2} \leq \epsilon ^{3\theta_0/2}$.
In all other cases (since $1\leq q < d\leq d_0\in\{3,4\}$) 
the right hand side of \eqref{prcompest2} can be improved to 
$\epsilon^{\theta_0 d/2}\cdot (\epsilon^{-2}T_1^{-1})$. 
This proves \eqref{removect} when $v_0$ is an $N$-atom.

Now assume $\Delta\chi<0$ and $v_0$ is I-atom, and $d_0=d=4$, and $\Mb_{\mathrm{pre}}=\Mb_0$. 
If $v_0$ is tame (i.e. it is SG in two different ways), 
using the tame condition we see that the number of solutions to \eqref{decmole1_1} 
is at most $R^{4-q}\epsilon^{2-\theta}$ if $q\geq 2$ and $R^{3}(\epsilon^{2-\theta})^2$ if $q=1$;
in either case we have
\begin{align*}
\Af_{\mathrm{pre}} /  \Af_{\mathrm{pos}} \leq
  (\epsilon^2T_1)^{4-q}\cdot (\epsilon^{2-\theta})^{1+\mathbf{1}_{q=1}}\leq \epsilon^{2\theta}\cdot (\epsilon^{-2}T_1^{-1})
\end{align*} 
which again proves \eqref{removect}. 
If $v_0$ is not tame, then by (2) in Proposition \ref{prop_m0} it is LG
and we must have that either 
$q=1$ and $p_1=4$, or $q=2$ and $(p_1,p_2)=(1,3)$; in both cases we must also have
\begin{equation}\label{defomegadec_1}
  \sum_{\ell\in A_j}\zeta_{v,\ell}(|k_\ell|^{1/2}+\epsilon^2\Gamma_0(k_\ell))=(\mathrm{const})+O(T_1^{-1})
\end{equation} 
for each $j$. 
In the case $q=1$, by using Proposition \ref{3vecprop} we get that the number of solutions
to \eqref{decmole1_1} and \eqref{defomegadec_1} is bounded by $R \cdot R^2T_1^{-1} \log R$, which implies that
\begin{equation*}
\Af_{\mathrm{pre}} / \Af_{\mathrm{pos}} \leq (\epsilon^{2}T_1)^{3}T_1^{-1} \epsilon^{-\theta}
  \leq \epsilon^{2\theta_0}\cdot (\epsilon^{-2}T_1^{-1}),
\end{equation*}
which matches the desired estimate \eqref{removect}.
In the case $q=2$, with $(p_1,p_2)=(1,3)$ we can use again Proposition \ref{3vecprop} 
to get that solutions to \eqref{decmole1_1} and \eqref{defomegadec_1}
are at most $R^2T_1^{-1} \log R$, which implies
\begin{equation*}
\Af_{\mathrm{pre}} / \Af_{\mathrm{pos}} \leq (\epsilon^{2}T_1)^{2}T_1^{-1} \epsilon^{-\theta}
  \leq \epsilon^{\theta_0}\cdot (\epsilon^{-2}T_1^{-1}),
\end{equation*}
which suffices for \eqref{removect}.

Finally, assume that $v_0$ is I-atom and $\Mb_{\mathrm{pre}}$ has no N-atom. 
In this case we want to use Lemma \ref{lemGammav}.
Noticing that all the N-atoms have been removed in previous steps when forming $\Mb_{\mathrm{pre}}$, 
we conclude that in $\Mb_{\mathrm{pre}}$, the value of $\Gamma_v$ belongs 
to a fixed interval of length $O(T_1^{-1})$ for each $v$. 
By adding these values for all atoms $v$ in the $j$-th connected component 
(created after removing $v_0$)
we conclude that \eqref{defomegadec_1} holds for each $j$ in addition to (\ref{decmole1_1}).
Using this we can then apply Propositions \ref{2vecprop} and \ref{3vecprop} 
to bound the number of 
solutions to the systems \eqref{decmole1_1}--\eqref{defomegadec_1} as follows;
recalling that $p_j := |A_j|$, $j=1,\dots,q < d$, is the number of bonds connecting the atom 
$v_0$ to the $j$-th connected component created by its removal, we have:

\begin{itemize}
 
 \item If $p_j=1$ the decoration of this single edge is fixed (i.e., there is only one solution);
 
 \smallskip
 \item If $p_j=2$ and $d_0 \neq 4$ we bound the number of possible solutions trivially by $R = R^{p_j-1}$; 
 
 \smallskip
 \item If $p_j=2$ and $d_0 = 4$ we use the $2$-vector counting from Proposition 
 \ref{2vecprop} to bound the number of solutions by $R (T_1\lambda)^{-1}$
 where $\lambda$ is the gap; by the LG hypothesis in (ii) of the statement (see also Proposition \ref{prop_m0} (2)) the 
 number of choices is bounded by $R (T_1 \epsilon^2)^{-1} \epsilon^{\theta} 
 = (R T_1^{-1} \epsilon^{-2})^{p_j-1} \epsilon^{\theta}$;

 \smallskip
 \item If $p_j \geq 3$ we first pick $p_j-3$ edges and then apply the $3$-vector counting from Proposition \ref{3vecprop} 
 to bound the number of decoration by $R^{p_j-3} \cdot R^2 T_1^{-1} \log R \leq R^{p_j-1} T_1^{-1} \epsilon^{-\theta}$.
 
\end{itemize}

From \eqref{prcompest1}, and since $\sum_{j=1}^q p_j-1 = d-q$, it follows that
\begin{align}\label{prcompest10}
\begin{split}
\frac{\Af_{\mathrm{pre}}}{\Af_{\mathrm{pos}}} & \leq  (R\epsilon^{-2}T_1^{-1})^{q-d} \cdot T_1^{-(d_0-4)/2} \cdot
  \prod_{j=1}^d R^{p_j-1} 
  \big(T_1^{-1} \epsilon^{-2} \epsilon^{\theta}\big)^{\mathbf{1}(p_j=2,d_0=4)}
  \big(T_1^{-1} \epsilon^{-\theta}\big)^{\mathbf{1}(p_j\geq 3)}
\\
& = T_1^{-( d_0-4)/2}\cdot\prod_{j=1}^q\sigma_j, \,\quad \mathrm{with} \quad \sigma_j := \left\{
\begin{aligned}
& 1, & \mathrm{if\ } p_j=1;
\\
& \epsilon^{2}T_1, & \mathrm{if\ } p_j=2 \mathrm{\ and\ }d_0\neq 4;
\\
& \epsilon^{\theta}, & \mathrm{if\ } p_j=2 \mathrm{\ and\ }d_0=4;
\\
& (\epsilon^2T_1)^{p_j-1} T_1^{-1} \epsilon^{-\theta}, & \mathrm{if\ }p_j\geq 3.
\end{aligned}
\right.
\end{split}
\end{align}
Now, using that $T_1 = \epsilon^{-8/3 + \theta_0}$, we see that $\prod_{j=1}^q\sigma_j$ 
is maximized in the following scenarios:

\begin{itemize}

\smallskip
\item $d_0=4$, $d=4$ and $(p_1,p_2) = (2,2)$, in which case \eqref{prcompest10} 
is upper bounded by $\epsilon^{2\theta}$,
consistently with the desired \eqref{removect} with $\mu=1$; 
in this case there is a $\beta$-cut at $v_0$ as in (3) in the statement;

\smallskip
\item $d_0=4$, no $\beta$-cut is possible at $v_0$, and $d<4$, so that in particular 
$p_1=d$ and \eqref{prcompest10} is upper bounded by $1$ if $d=1$,
by $\epsilon^{\theta}$ if $d=2$; 
these case are in (3) of the statement, and the bounds are consistent with \eqref{removect} with $\mu=1$;

\smallskip
\item $d_0=5$, and either $d=4$ and $(p_1,p_2) = (2,2)$, or $d=5$ and $(p_1,p_2,p_3) = (2,2,1)$
which corresponds to the scenario (2) in the statement;
in both cases \eqref{prcompest10} is upper bounded by $T_1^{-1/2} (\epsilon^2T_1)^2 = \epsilon^{2\theta_0}$
which is consistent with \eqref{removect} with $\mu=1$.

\end{itemize}

In all the cases that are not those considered above (so that are not (2) and (3) in the statement)
we want to show that one has \eqref{removect} with the improvement $\mu=\epsilon^{-2}T_1^{-1}$. 
Indeed, the remaining cases are
\begin{itemize}
 
\item[(a)] $d_0(v_0) = 5$ and $d_0(v_0) < 4$;

\item[(b)] $d_0(v_0) > 5$;

\item[(c)] $d_0(v_0) = 4 = d_0(v_0)$ and there is no $\beta$-cut at $v_0$.

\end{itemize}

\noindent
In case (a), using \eqref{prcompest1} when $d=3$, and considering all the possible 
cases $(p_1,p_2,p_3) = (1,1,1)$, $(p_1,p_2) = (2,1)$, and $p_1 = 3$, 
we have the bound
\begin{align*}
T_1^{-1/2} \cdot 
  \epsilon^2 T_1 \lesssim \epsilon^{(3/2)\theta_0} \epsilon^{-2}T_1^{-1}
\end{align*}
in view of our choice of $T_1$.
When $d=2$ and $p_1=2$ or $(p_1,p_2)=(1,1)$, we get the same bound above, and when $d=1$ an even better one.

In case (b), from the general bound \eqref{prcompest1}
we can see that the worst case scenario is when $d=d_0$ and all $p_j=2$ (thus $q=d_0/2$) 
in which case we obtain a bound by
\begin{align*}
T_1^{-(d_0-4)/2} (\epsilon^2T_1)^{d_0/2} = T_1^2 \epsilon^{d_0} \lesssim \epsilon^{-2} T_1^{-1},
\end{align*}
where the last inequality is again due to our choice of $T_1$ and $d_0 \geq 6$.

In the last case (c), since there is no $\beta$-cut, we only have one connected component ($p_1=4$) 
and the bound from \eqref{prcompest1} reads
\begin{align*}
\frac{\Af_{\mathrm{pre}}}{\Af_{\mathrm{pos}}} \lesssim (\epsilon^2T_1)^3 T_1^{-1} \epsilon^{-\theta}
\end{align*}
which again suffices for \eqref{removect} in view of our choice of $T_1$.
This concludes the proof. 

\end{proof}

\smallskip
Now we can prove Proposition \ref{comp_est_2}.

\begin{proof}[\bf{Proof of Proposition \ref{comp_est_2}}] 
We consider several cases. We may assume $G=1$, otherwise the proof below will be easier 
when $G=0$ (and thus $V_\beta=0$).

\smallskip
{\it Case 1: Cycles}. Suppose $\Mb_0$ contains only $\alpha$- and $\beta$-atoms, then it must be a cycle. 
If it is a cycle of length $2$ (i.e. a double bond), by Proposition \ref{prop_m0} (1),
it must be LG, and thus must have two $\beta$-atoms. 
By Proposition \ref{2vecprop} we can control the number of choices of decorations for this double bond
by $RT_1^{-1}\epsilon^{-2+\theta}$. 
Plugging this into \eqref{defA}, since $\chi=1$, gives the bound on the QC
\begin{align*}
\Af \leq  R^{-1}\epsilon^2T_1 \cdot RT_1^{-1}\epsilon^{-2+\theta} = \epsilon^\theta;
\end{align*}
this is smaller than the right-hand side of \eqref{comp_est} since $V_\beta=2,G=1,E=2$ in this case.


Next, assume $\Mb_0$ contains at least 3 $\beta$-atoms. 
Bounding trivially the number of decorations for the cycle by $R$, 
we obtain
\begin{align*}
\Af \leq (R^{-1}\epsilon^2T_1)^{\chi} \cdot R = \epsilon^2T_1
  \leq \epsilon^{\theta E/4}\cdot\prod_{(\alpha)}
  (\lambda^{-1/2}\epsilon^{-2+\theta}T_1^{-1})\cdot(\epsilon^{-2+\theta}T_1^{-1})^{2-V_\beta},
\end{align*}
which is consistent with \eqref{comp_est},
having used for the last inequality that $\beta\geq 3$ 
and $\lambda^{-1/2}\epsilon^{-2+\theta}T_1^{-1} \geq \epsilon^{-1/3}$
(using that $\lambda\leq \epsilon^{2-\theta}$).

Next, assume $\Mb_0$ contains at least 2 $\beta$-atoms and at least one $\alpha$-atom with gap $\sim\lambda_1$
By Proposition \ref{2vecprop} we get
\begin{align*}
\Af \lesssim R^{-1}\epsilon^2T_1\cdot \min(R,R T_1^{-1}\lambda_1^{-1}) = \epsilon^2\min(T_1,\lambda_1^{-1}).
\end{align*}
Then note that
\begin{align*}
\epsilon^2\min(T_1,\lambda_1^{-1})\lesssim\epsilon^2\lambda_1^{-1/2}T_1^{1/2}
  \leq \lambda_1^{-1/2}\epsilon^{-2+\theta}T_1^{-1}
\end{align*}
by our choice of $T_1$,
and since $V_\beta\geq 2$, it follows that 
\begin{align*}
\Af \lesssim \epsilon^{\theta E/4}\cdot\prod_{(\alpha)}(\lambda^{-1/2}\epsilon^{-2+\theta}T_1^{-1})\cdot(\epsilon^{-2+\theta}T_1^{-1})^{2-V_\beta},
\end{align*}
having used again that each factor in the above product is larger than $\epsilon^{-1/3}$.

Finally, assume $\Mb_0$ contains at least one $\beta$-atom 
and at least 2 $\alpha$-atoms with gaps $\lambda_1$ and $\lambda_2$.
Then, again by Proposition \ref{2vecprop} and similar to the above, we first bound
\begin{align*}
\Af \lesssim\epsilon^2\min(T_1,\lambda_1^{-1},\lambda_2^{-1}) & \lesssim (\lambda_1\lambda_2)^{-1/2}\epsilon^2
  \\
  & \lesssim (\lambda_1^{-1/2}\epsilon^{-2+\theta}T_1^{-1})
  \cdot (\lambda_2^{-1/2}\epsilon^{-2+\theta}T_1^{-1})\cdot (\epsilon^{-2+\theta}T_1^{-1})
\end{align*}
having used that $T_1 = \epsilon^{-8/3+\theta_0}$ for the last inequality.
Using that $V_\beta\geq 1$ we see that 
\begin{align*}
\Af \lesssim \epsilon^{\theta E/4}\cdot\prod_{(\alpha)}(\lambda^{-1/2}\epsilon^{-2+\theta}T_1^{-1})\cdot(\epsilon^{-2+\theta}T_1^{-1})^{2-V_\beta}.
\end{align*}
This completes the proof of \eqref{comp_est} in the case when $\Mb_0$ is a cycle.

\begin{rem}
Note how the last two bounds above are tight; these are the cases of triangles with at least a $\beta$-atom
and SG of size $\approx T_1^{-1}$ at the $\alpha$-atoms.
This is consistent with the `bad' scenario presented in Section \ref{SecDiv}, see Figure 
\ref{fig:count_mol}.
\end{rem}

\smallskip
{\it Case 2: At least one $\varepsilon$-atom}.
Suppose $\Mb_0$ contains at least one $\varepsilon$-atom. 
In this case, we perform the following ``selection for removal'' operations: 
first select all the $\varepsilon$-atoms in a certain order to be specified shortly below.
Then, in the order of selection, each time remove one $\varepsilon$-atom, as in Proposition \ref{comp_est_0}; 
if any $\alpha$- or $\beta$-atom becomes degree 1 then we also remove them 
(for which we always have $\Delta\chi=0$, see Proposition \ref{comp_est_0}), 
and then proceed to remove the next $\varepsilon$-atom in the order, and so on.
For the selection ordering we require that:

\begin{enumerate}[{(a)}]

\smallskip
\item If $\Mb_0$ contains one atom $v$ that is I-atom of degree 4, then we order it \emph{first}; 
then we order all the \emph{N-atoms}, and then all the \emph{remaining I-atoms}.

\smallskip
\item Otherwise, we first order all the N-atoms, and then all the I-atoms. 
If some I-atom has degree 5 and is connected to other $\varepsilon$-atoms 
by \emph{at least two} bonds or by two chains of $\alpha$- and $\beta$-atoms, then we order it \emph{last}.

\smallskip
\item If none of (a)--(b) hold, then we set the order arbitrarily.

\end{enumerate}

\smallskip
We then want to apply Proposition \ref{comp_est_0} repeatedly at each step. 
Letting $j=1,\dots,K$ be the steps, and letting $\Af$ denote the counting quantity for the
molecule $\Mb_0$, we have
\begin{align}\label{proof_0}
\Af \lesssim \prod_{j=1}^K \epsilon^{(d_j/2) \, \mathbf{1}(\Delta \chi_j < 0)} \, \mu_j
\end{align}
where $d_j$ is the degree of the atom removed at step $j$, and we denote $\mu_j \in \{1,\epsilon^{-2}T_1^{-1} \}$ 
the quantity $\mu$ from the statement of Proposition \ref{comp_est_0} associated to the step $j$,
and $\Delta \chi_j$ is the increment in $\chi$ at step $j$.

\smallskip
\noindent
{\it Sub-case 2.1: $V_\beta \geq 1$ or $V_\beta = 2$ and $\mu_j\neq 1$ for some $j$}.
Noticing that $\Delta\chi \geq -d$, we see that $\sum_j d_j \geq \chi_0$ where $\chi_0$
is the initial circuit rank of $\Mb_0$. Since each atom is at least degree $2$, we have 
$E \geq 2V$ and thus $\chi_0 \geq E/2$; then, from \eqref{proof_0} we obtain 
\begin{equation}\label{proof_1}
\Af \lesssim \epsilon^{E/4} \prod_{j=1}^K \mu_j.
\end{equation}
Since 
$V_\beta\geq 1$ (from Proposition \ref{prop_m0} (1)), 
we see that \eqref{proof_1} is already a stronger bound than the right-hand side of \eqref{comp_est} 
if \emph{at least one} $\mu_j = \epsilon^{-2}T_1^{-1}$ in \eqref{proof_1} (recall $G=1$).
We also see that \eqref{proof_1} is already sufficient for \eqref{comp_est} if $V_\beta\geq 2$.

\smallskip
\noindent
{\it Sub-case 2.2: $V_\beta = 1$, $\mu_j = 1$ for all $j$}.
Let us now look at the remaining case with $V_\beta=1$ and $\mu_j=1$ for all $j=1,\dots K$ in \eqref{proof_1}. 
By Proposition \ref{comp_est_0}, this implies the following properties for $\Mb_0$:

\begin{enumerate}[{(i)}]

\item $\Mb_0$ has no I-atoms of degree $4$;

\item $\Mb_0$ has no I-atoms of degree $6$ or more;

\item $\Mb_0$ has no N-atoms of degree $4$ and the degree (which is $3$) of $N$-atoms stays constant
through all the steps;

\item No two degree $3$ N-atoms in $\Mb_0$ can be connected by a bond or chain of $\alpha$- and $\beta$-atoms; 

\item Each degree $5$ I-atom in $\Mb_0$ can be connected to other $\varepsilon$-atoms 
by at most one bond or chain of $\alpha$- and $\beta$-atoms.
\end{enumerate}

Let us briefly verify these properties.
Property (i) follows from (3) in Proposition \ref{comp_est_0} since no $\beta$-cuts are allowed 
(this would increase $V_\beta$ by $2$),
and since an atom $v_0$ with initial degree $d_0(v_0)=4$ in $\Mb_0$ 
is ordered first to be removed according to Step (a) 
of the selection process described at the beginning of this Case 2 (so that $d(v_0)=4$).

Property (ii) follows directly from (2) in Proposition \ref{comp_est_0}.
while property (iii) follows from (1) there.

To verify property (iv) notice that if two $N$-atoms are connected by a bond then the removal of one of them
will decrease the degree of the other, contradicting (iii). 
Similarly, if they are connected via a chain
of $\alpha$ or $\beta$ atoms, removing one of them will force the removal of the chain
according to the selection algorithm (since the $\alpha$ or $\beta$ atoms become degree 1)
and eventually the degree of the $N$-atom at the other end of the chain will decrease, again contradicting (iii).

For property (v), assume by contradiction that a degree $5$ I-atom in $\Mb_0$, denoted $I_0$,
is connected to another $\varepsilon$-atom, denoted $\varepsilon_0$, 
by more than one bond or chain of $\alpha$ or $\beta$ atoms; 
then, according to (b) of the selection algorithm both $I_0$ and $\varepsilon_0$ are degree $5$.
Removing one of them, say $I_0$, will force the degree of $\varepsilon_0$ to decrease by at least two
which contradicts (2) of Proposition \ref{comp_est_0}.

\smallskip
We can now proceed to conclude the proof of the proposition.
Consider a degree 3 N-atom $v$ in $\Mb_0$ (if there is any). 
If $\mu=1$ for the step of removing this atom then, 
according to the argument after \eqref{prcompest2} in the proof of Proposition \ref{comp_est_0},
we must have $q=1$ in this step.
By properties (iv) and (v) above,
this means that $v$ is connected to 3 different degree 5 I-atoms by 
3 different chains of $\alpha$ or $\beta$ atoms (or bonds). 
In this case we know there exist at least 3 degree 5 I-atoms. 
On the other hand, if there is no degree 3 N-atom, then there also exist at least 2 degree 5 I-atoms
(by our assumption that one $\varepsilon$ atom exists, and properties (i) and (ii)).

Now consider each of the at least two degree 5 I-atom in $\Mb_0$. 
By property (v) it must have at least two self-connecting bonds 
or chains of $\alpha$- and $\beta$-atoms connecting to itself.
Since $V_\beta=1$, we can find one degree 5 I-atom such that these chains do not contain any $\beta$-atom. 
There are then two possibilities: 
(a) $v$ has two self-connecting bonds, or 
(b) $v$ is connected to itself by a chain containing at least one $\alpha$-atom.

In case (a), this atom $v$ must correspond to a branching node $\nf$ in $\Qc_{\mathrm{sb}}$ 
with its 4 children nodes forming two pairs, but this implies that 
$k_\nf = \zeta_{\nf_1}k_{\nf_1}+\cdots+\zeta_{\nf_4}k_{\nf_4}=0$ in the decoration,
and thus the contribution of this term vanishes because the corresponding 
symbol $\widetilde{n}_{\zeta_{\nf_1}\cdots\zeta_{\nf_4}}$ vanishes thanks to the property \eqref{symbolbound} 
which gives that $\widetilde{n}_{\zeta_1\cdots\zeta_4}(\eta_1,\dots,\eta_4)=0$
for $\eta_1+\cdots+\eta_4=0$.

Finally, let us look at case (b), where the degree $5$ I-atom $v$ has a self-connecting chain with 
at least one $\alpha$-atom.
Let the gap at this $\alpha$-atom be $\lambda$, and consider the step of removing $v$ as in Proposition \ref{comp_est_0}. 
According to \eqref{prcompest10}, with the same notation there, for this step we have
\begin{align*} 
\Af_{\mathrm{pre}} \leq \Af_{\mathrm{pos}} \cdot T_1^{-1/2} \cdot \prod_{j=1}^q\sigma_j.
\end{align*}
Recall also the notation for the sets of bonds $A_j$ connecting $v$ to the connected components 
created after its removal, and that $p_j = |A_j|$.
Let $A_1$ be the set of bonds connecting $v$ to the chain containing the $\alpha$-atom;
in particular $p_1=2$, and $p_2+\dots+p_q=3$.
By first counting the two vectors associated with this $\alpha$-atom 
and using Proposition \ref{2vecprop}, we obtain
\begin{align*} 
\Af_{\mathrm{pre}} \leq \Af_{\mathrm{pos}} \cdot T_1^{-1/2} \cdot \min(1,T_1^{-1}\lambda^{-1}) \cdot \prod_{j=2}^q\sigma_j.
\end{align*}
According to \eqref{prcompest10} the product in the above right-hand side is either $1$ (in the case of $p_2=p_3=p_4=1$)
or $\varepsilon^2T_1$ (in the case $\{p_2,p_3\}=\{2,1\}$,
or $(\varepsilon^2T_1)^2 T_1^{-1} \epsilon^{-\theta}$ (in the case $p_2=3$).
Given our choice of $T_1$, the worst case is the second one, for which we get
\begin{equation*} 
\frac{\Af_{\mathrm{pre}}}{\Af_{\mathrm{pos}}} \leq T_1^{-1/2} \cdot \min(1,T_1^{-1}\lambda^{-1}) \cdot \varepsilon^2T_1
  \leq \lambda^{-1/2} \varepsilon^2.
\end{equation*}
This is then sufficient to obtain \eqref{comp_est}, because $V_\beta = 1$ and
\begin{equation*}
\lambda^{-1/2} \epsilon^{2} \leq (\lambda^{-1/2}\epsilon^{-2+\theta}T_1^{-1}) \cdot (\epsilon^{-2+\theta}T_1^{-1}) 
\end{equation*}
again thanks to our choice of $T_1$. The proof of the proposition is completed.
\end{proof}


\smallskip\begin{proof}[Proof of Lemma \ref{lemGammav}]
The proof follows by induction on the number of $\beta$ cuts:
the case of no $\beta$ cut is trivial by summing the conditions \eqref{defomegadec} over all atoms in $\Mb_0$. 
Suppose we perform one more $\beta$ 
cut which forms two paired $\beta$-atoms $v'$ and $v''$ from one degree 4 I-atom $v$ 
(cf. Definition \ref{defdeccut}), then by adding the values of $\Gamma_w$ 
for all atoms $w$ in one of the newly created components, and applying the induction hypothesis,
we get that $\Gamma_{v'}$ equals the desired form; 
by the assumption $\Gamma_{v'}+\Gamma_{v''}=(\mathrm{const})+O(T_1^{-1})$, 
we get that $\Gamma_{v''}$ also has the desired form.
\end{proof}

\subsection{Proof of Proposition \ref{propjr}} 
With all the above ingredients we now complete the proof of Proposition \ref{propcong_2},
and then that of Proposition \ref{propjr}.

\begin{proof}[Proof of Proposition \ref{propcong_2}] 
By combining \eqref{molbound} in Proposition \ref{countingred1}, 
\eqref{decompose_sum}, and the definition of the CQ in \eqref{defA}, we obtain first 
\begin{align*}
R^{-1}\sum_k\langle k\rangle^{10}|\Ic(t,k)|\leq \epsilon^{-2}T_1^{-1} \cdot \epsilon^{2q_{\mathrm{tot}}/3} \cdot 
  \Af \cdot (\log R)^{C}
\end{align*} 
where, recall, $q_{\mathrm{tot}}$ is the sum of the lengths $q_j$ of the irregular chains $\Hc_j$ 
(cf. Definition \ref{defirrechain} and Proposition \ref{irrereduce}).
Applying then 
the result of the cutting algorithm \eqref{cutincr} from Proposition \ref{prop_cutincr} we get
\begin{align}\label{prjrest1}
\begin{split}
R^{-1}\sum_k\langle k\rangle^{10}|\Ic(t,k)| \leq \epsilon^{-2}T_1^{-1} \cdot \epsilon^{2q_{\mathrm{tot}}/3} 
  \cdot (\log R)^{C} \cdot \mathfrak{A}'
  \\
  \cdot \prod_{(\alpha)}(\lambda^{1/2}\epsilon^{2-\theta}T_1)\cdot (\epsilon^{-2+\theta}T_1^{-1})^{V_\beta-2\Delta F},
\end{split}
\end{align} 
where $\mathfrak{A}'$ is the CQ associated to the reduced molecule $\Mb'$
described in Proposition \ref{prop_m0}.
Finally, applying the counting estimate \eqref{comp_est} to each connected component $\Mb_0$ in $\Mb'$,
noticing that the last two factors in \eqref{prjrest1} exactly cancel the last two factors in \eqref{comp_est},
we arrive at
\begin{equation}\label{prjrest2}
R^{-1}\sum_k\langle k\rangle^{10}|\Ic(t,k)| \leq \epsilon^{-2}T_1^{-1} \cdot \epsilon^{2q_{\mathrm{tot}}/3} 
  \cdot \epsilon^{\theta |\Yc|} \cdot (\log R)^{C}
  \cdot \epsilon^{\theta E/4},
\end{equation}
where, recall, $E$ is the number of edges in $\Mb'$ and $|\Yc|$ the number of edges in the set of double bonds $\Yc$.
Finally, to show that \eqref{prjrest2} implies \eqref{est_total_I} 
is suffices to verify that the rank of the original couple $\Qc$ satisfies 
\begin{align}\label{prjrest3}
r(\Qc) \leq 2E + 4 |\Yc| + \frac{8}{3\theta} q_{\mathrm{tot}}.
\end{align}
Indeed, let $\Mb$ be the original molecule, 
and $\Yc$ the set of double bounds removed in Step 1
of the cutting algorithm in Subsection \ref{sec.cutting}.
Denoting with $\chi(\Mb)$ the circuit rank of a molecule $\Mb$, 
we have
\begin{align*}
r(\Qc) = 2\chi(\Mb) = 2 \big( \chi(\Mb_{\mathrm{sb}}) + 2q_{\mathrm{tot}} \big)
  = 2 \big( \chi(\Mb') + 2|\Yc| + 2q_{\mathrm{tot}} \big)
  \leq 2 \big( E + 2|\Yc| + 2q_{\mathrm{tot}} \big)
\end{align*}
which is sufficient.
\end{proof}

\begin{proof}[Proof of Proposition \ref{propjr}]

\smallskip
{\it Proof of \eqref{jrest1}.}
First, using the formula (\ref{ansatz2}) for $(\Jc_r)_k(t)$ and (\ref{defofjt}) for $(\Jc_\Tc)_k(t)$, 
we note that $\Eb|(\Jc_r)_k(t)|^2$ can be written as a sum $\sum_\Qc\Kc_\Qc(t,k)$ over all admissible couples $\Qc$ 
such that the rank of each tree forming $\Qc$ is equal to $r$. 
Note that this set of couples $\Qc$ is invariant under the congruence relation 
from Definition \ref{defcong},
so this sum can be decomposed into finitely many sub-sums of form \eqref{defI} 
in Proposition \ref{propcong_2}. 
The bound \eqref{jrest1} then follows from (\ref{est_total_I}) in Proposition \ref{propcong_2}.

{\it Proof of \eqref{fixedrenorm2} and \eqref{fixedrenorm3}}. 
Note that the definition of $\Xc[\Gamma_1]$ from \eqref{fixedrenormeqn} 
involves the (weighted) sum of couples $\Kc_\Qc(t,k_1)$ over $\Qc$ and $k_1$
and, again, can be reduced to summation over $\Qc$ in a congruence class. 
Therefore \eqref{fixedrenorm2} follows from the same proof we did above to obtain \eqref{est_total_I}. 

Finally, the contraction mapping property \eqref{fixedrenorm3} 
also follows from the same arguments above, so we just indicate some of the 
steps needed to adapt the previous estimates to bound $\Xc[\Gamma_1]-\Xc[\Gamma_1']$.
Notice that this last expression is a (weighted) sum over $k_1$ of finitely many terms of the form
\begin{align*}
\sum_{\Qc \equiv \Qc_0} \Big( \Kc_{\Qc}[\Gamma_1](t,k_1) - \Kc_{\Qc}[\Gamma_1'](t,k_1) \Big)
\end{align*}
where the sum is over all admissible couples $\Qc_0$ such that the rank of each tree forming $\Qc_0$ is equal to $r$,
and where $\Kc_{\Qc}[\Gamma](t,k)$ denotes the expression for couples $\Kc_{\Qc}(t,k)$ from \eqref{defofkq}
with the phases $\Omega_{\mf}$ defined by \eqref{resfactor} and \eqref{ddefomegan} 
using the given `gauge' $\Gamma=\Gamma_1$ or $\Gamma_1'$. 
From \eqref{defofkq} we see that the only terms that involve the gauge factors $\Gamma_1$ and $\Gamma_1'$
are in the time integrals, so we only need to pay attention to the differences of those,
the rest of the formula being identical for $\Kc_{\Qc}[\Gamma_1](t,k_1)$ and $\Kc_{\Qc}[\Gamma_1'](t,k_1)$.
More precisely, let us denote by 
\begin{equation}\label{defI2'}
\Ic[\Gamma_1] := \int_\Ec\prod_{\mf} e^{i\alpha_\mf^0 t_\mf}\cdot e^{i\Psi_\mf[\Gamma_1](t_\mf)}\,\mathrm{d}t_\mf,
\end{equation} 
the expression in \eqref{defI2} with $\Psi_\mf[\Gamma_1]$ defined as in \eqref{defPsi}.
The difference of the oscillatory time-integrals that appear in 
$\Kc_{\Qc}[\Gamma](t,k_1)-\Kc_{\Qc}[\Gamma'](t,k_1)$ is given by
\begin{equation}\label{I2diff}
\Ic[\Gamma_1] - \Ic[\Gamma_1'] 
  = \int_\Ec\prod_{\mf} e^{i\alpha_\mf^0 t_\mf}\cdot \big( e^{i\Psi_\mf[\Gamma_1](t_\mf)}
  - e^{i\Psi_\mf[\Gamma_1'](t_\mf)} \Big) \,\mathrm{d}t_\mf.
\end{equation}
We then claim that a similar result as in Proposition \ref{timeintest} holds, namely 
\begin{equation}\label{timeintest1'}
\big| \Ic[\Gamma_1] - \Ic[\Gamma_1']| \leq \Ac\big((\alpha_\mf^0)_{\mf}\big) 
  \| \Gamma_1 - \Gamma_1'\|_{C^0_t L^\infty_k},
\end{equation} 
where $\Ac$ is an explicit function depending on $\alpha_\mf^0$ only, and satisfying \eqref{timeintest2}.

To see \eqref{timeintest1'} we first notice that an integration by parts in
\eqref{I2diff}, as done in \eqref{IBP}, generates a bulk term with a factor
\begin{equation}\label{IBPdiff}
\partial_t \big( e^{i\Psi_\mf[\Gamma_1](t_\mf)} - e^{i\Psi_\mf[\Gamma_1'](t_\mf)} \big) 
  = i \frac{d}{dt} \Psi_\mf[\Gamma_1](t_\mf) e^{i\Psi_\mf[\Gamma_1](t_\mf)} 
  - i \frac{d}{dt} \Psi_\mf[\Gamma_1'](t_\mf) e^{i\Psi_\mf[\Gamma_1'](t_\mf)};
\end{equation}
this is clearly bounded by $\|\Psi-\Psi'\|_{C^1_tL^\infty_k} \lesssim \|\Gamma_1-\Gamma_1'\|_{C_t^0L_k^\infty}$, 
consistently with the desired contraction property.
Moreover, the boundary terms generated by the integration by parts are of the form (cfr. \eqref{IBPbdr})
\begin{equation}\label{IBPbdr'}
\frac{1}{i\alpha_\mf^0} e^{i\alpha_\mf^0t_{\qf}} \cdot 
  \big(e^{i\Psi_\mf[\Gamma_1](t_\qf)} - e^{i\Psi_\mf[\Gamma_1'](t_\qf)} \big),
\end{equation}
and can be carried out to the next integration in time as in the proof of Proposition \ref{timeintest}.
In particular, the factor \eqref{IBPbdr} will 
shift the phases in the next integration to $\alpha_\mf^0t_{\qf} + \alpha_\qf^0t_{\qf}$,
exactly as explained after \eqref{IBPbdr}.
The rest of the arguments in the proof of Proposition \ref{timeintest} carry out verbatim
(since they are insensitive to the choice of $\Gamma_1$, as long as $\|\Gamma_1 \|_{C_t^0L_k^\infty} \leq 1$),
and one always picks up a factor like \eqref{IBPdiff} for some node $\mf$, 
or a factor bounded by the difference
$\Psi_\mf[\Gamma_1](t_\qf)-\Psi_\mf[\Gamma_1'](t_\qf)$ for some node $\mf$ belonging to a sub-tree rooted at $\qf$.



The bound \eqref{timeintest1'} above, with \eqref{timeintest2}, 
can then be used in Proposition \ref{countingred1}, to obtain a bound as in \eqref{molbound}
where the factor $(\log R)^C$ is replaced by $(\log R)^C \|\Gamma_1 - \Gamma_1'\|_{C^0_t L^\infty_k}$.
The proof can then proceed as in the rest of 
the section, and we obtain \eqref{est_total_I} with an extra factor of $\|\Gamma_1 - \Gamma_1'\|_{C^0_t L^\infty_k}$
on the right-hand side; this is the desired bound \eqref{fixedrenorm3}.
The proof of Proposition \ref{propjr} is complete.
\end{proof}




\section{Treatment of the remainder}\label{sec.rem} 
In this section we prove Proposition \ref{propdiff}. 
To start, we first prove that $(a_k)_{\mathrm{app}}(t)$ is an approximate solution to the equation 
\eqref{profileeqn} for the modified profile $(a_k)_{\mathrm{tr}}$.

\begin{prop}\label{proofapprox} 
We have
\begin{equation}\label{profileeqn_2}
\begin{aligned}
(a_k)_{\mathrm{app}}({t}) & =  (\epsilon R^{1/2})^{-1} \widehat{w_{\mathrm{in}}^{\mathrm{tr}}}(k)+i\sum_{r=3}^A(\epsilon R^{-1/2})^{r-1}T_1\cdot\sum_{\zeta_j\in\{\pm\}}\sum_{\zeta_1k_1+\cdots+\zeta_rk_r=k}\int_0^{t} e^{iT_1\cdot \Omega_r({t}')}
  \\
  &\times i\widetilde{n}_{\zeta_1\cdots\zeta_r}(k_1,\cdots,k_r)\cdot\varphi_{\leq K_{\mathrm{tr}}}(k)\prod_{j=1}^r(a_{k_j})_{\mathrm{app}}({t}')^{\zeta_j}\,\mathrm{d}{t}'
\\
& +iT_1\int_0^{t}(\epsilon^2\Gamma_0(k)+T_1^{-1}\Gamma_1({t}',k)) 
  (a_k)_{\mathrm{app}}({t}') \mathrm{d}{t}'
  + \mathcal{R}_k(t),
\end{aligned}
\end{equation} 
where, with the exception of a set of probability $\lesssim e^{-2R^{\eta_0}}$, 
the remainder term $\mathcal{R}_k(t)$ satisfies that
\begin{equation}\label{proofapprox_2}
\|\mathcal{R}_k(t)\|_{L_t^\infty L_k^1}\leq \epsilon^{N\theta/10}.
\end{equation}
\end{prop}

\begin{proof}
Recall that
\begin{equation}\label{proofapprox_3}
(a_{k})_{\mathrm{app}}(t)=\sum_{\Tc}(\Jc_\Tc)_k(t),
\end{equation} 
where the summation is taken over all admissible paired trees or rank $r(\Tc)\leq N$, 
and $\Jc_\Tc$ is defined as in \eqref{defofjt}, see \eqref{ansatz}.

Consider the right hand side of \eqref{profileeqn_2} with the exclusion of the remainder $\Rc_k$; 
by definition of initial data \eqref{wtrin} and \eqref{data_w}, 
we see that the first term, which corresponds to $w_{\mathrm{in}}^{\mathrm{tr}}$, is equal to
\begin{equation}\label{proofapprox_4}
I(k,t) =\sum_{\Tc}^{(I)}(\Jc_\Tc)_k(t),
\end{equation} where the summation is taken over all admissible paired trees 
$\Tc$ of sign $+$ \emph{that do not have I-branching node}.

Then consider the second term on the right hand side of \eqref{profileeqn_2}, 
which corresponds to the multilinear expressions of $a_{\mathrm{app}}$. 
By the definition of the trees $\Jc_\Tc$ in the ansatz \eqref{ansatz} for $a_{\mathrm{app}}$, 
this second term is equal to
\begin{equation}\label{proofapprox_5}
II(k,t) =\sum_{\Tc}^{(II)}(\Jc_\Tc)_k(t),
\end{equation}
where the summation is taken over all paired trees $\Tc$ such that:
\begin{enumerate}
\item The root $\rf$ of $\Tc$ is an I-branching node of sign $+$, and $\Tc$ is formed by subtrees $\Tc_j\,(1\leq j\leq r)$ where $3\leq r\leq A$;
\item Each $\Tc_j$ is an admissible paired tree of rank $\leq N$.
\end{enumerate}

Finally, consider the third term, which corresponds to the gauge. By definition of $\Gamma_0$ (\ref{defgauge}) and $\Gamma_1$ (\ref{fixedrenorm}), we see that the third term is equal to
\begin{equation}\label{proofapprox_6}
III(k,t) =-\sum_{\Tc}^{(III)}(\Jc_\Tc)_k(t),
\end{equation} 
where the summation is taken over all paired trees $\Tc$ such that:

\begin{enumerate}[resume]

\item The root of $\Tc$ is an I-branching node of sign $+$ with 3 children, 
and $\Tc$ is formed by by subtrees $\Tc_j\,(1\leq j\leq 3)$ of signs $(+,+,-)$;

\item Each $\Tc_j$ is an admissible paired tree of rank $\leq N$, 
and the leaves of $\Tc_3$ are completely paired with the leaves of $\Tc_i$ for some $i\in\{1,2\}$.

\end{enumerate}

Now we put $I\sim III$ together and compare with the left hand side 
\begin{equation}\label{proofapprox_7}
(a_{k})_{\mathrm{app}}(t) 
 = \sum_\Tc^{(IV)}(\Jc_\Tc)_k(t)
\end{equation} where the summation is taken over all admissible paired trees with rank $\leq N$. Then we know that:
\begin{itemize}
\item If $\Tc$ occurs in \eqref{proofapprox_7} 
then it either has no I-branching node 
(and thus occurs in $I$), or its root $\rf$ is an I-branching node and it occurs in $II$;
\item If $\Tc$ occurs in $II$, then $\Tc$ is \emph{not admissible} 
(cf. Definition \ref{defadm}) if and only if $\Tc$ occurs in the summation defining in $III$ since, by definition, 
this forces $\rf$ to be an I-branching node with 3 children and two 
of its children have their subtree leaves completely paired; 
note that the gauge factor $\Gamma_0$, respectively $\Gamma_1$, in $III$ exactly correspond
to the cases when the (completely paired) trees $\Tc_3$ and $\Tc_i$
are both trivial or not. 
\end{itemize}

Therefore, we conclude that
\begin{equation}\label{proofapprox_8}\Rc_k(t)=
(a_{\mathrm{app}})_{k}(t) - \big[ I(k,t)  + II(k,t)  + III(k,t)  \big] =\sum_{\Tc}(\Jc_\Tc)_k(t),
\end{equation} 
where the summation is taken over a \emph{subset} $\Ac$ of paired trees $\Tc$, 
such that each $\Tc\in\Ac$ \emph{has rank $>N$}, 
and \emph{the subtree rooted at any child node of the root of $\Tc$ is admissible with rank $\leq N$}.
Moreover, this set $\Ac$ is invariant under any twist 
(Definition \ref{defcong}, which can be given also for irregular chains inside a paired tree), 
as long as we \emph{shorten the irregular chain such that it does not contain the root $\rf$}. 

Now, by \eqref{proofapprox_8}, we have that
\begin{equation}\label{proofapprox_9}
\Eb\bigg(\sum_k\langle k\rangle^{10}|\Rc_k(t)|^2\bigg) = \sum_k\langle k\rangle^{10}\sum_\Qc(\Kc_\Qc)_k(t),
\end{equation} 
where the summation is taken over all couples $\Qc=(\Tc^+,\Tc^-)$ with $\Tc^+ \in \Ac$, 
and $\Tc^-$ formed from a paired tree in $\Ac$ by flipping the signs of all nodes. 

Note that for any couple $\Qc$, by the `admissibility' property
of the subtrees rooted at each child of each root, 
we know that it cannot have two children nodes of an I-branching node with their subtree 
leaves completely paired, \emph{unless this I-branching node is a root}. 
This leads to \emph{at most two} exceptional non-admissible instances, 
apart from which the couple $\Qc$ is admissible. 
We can then apply the same proof in Subsection \ref{sec.L1est} and Sections \ref{sec.irre} and \ref{sec.molecule} 
with some small adjustments which we now describe.
In particular, we can write the expressions $\Kc_\Qc$ on right hand side of \eqref{proofapprox_9} 
as a sum of expressions of the form \eqref{defI}, 
with (i) the exception of the two above non-admissible instances, which lead to at most an $O(R^2)$ loss, 
and (ii) there exist at most two irregular chains containing a root, which need to be 
``cut short'' by excluding the root, which leads to at most another $O(R^2)$ loss.
According to this, and the bound \eqref{est_total_I}, we have
\begin{equation}\label{proofapprox_10}
\Eb\bigg(\sum_k\langle k\rangle^{10}|\Rc_k(t)|^2\bigg)\leq R^5 \epsilon^{\theta N/8} \leq \epsilon^{\theta N/9},
\end{equation} 
where we note that the rank $r=r(\Qc)\geq N$.

Finally, notice that $\Rc_k(t)$ is a multilinear Gaussian random variable, 
so that we can get large deviation estimates for $\Rc_k(t)$ for fixed $k$ and $t$ from Proposition \ref{initialprop}; 
this then leads to the large deviation estimates for the norm 
$\|\Rc_k(t)\|_{L_t^\infty L_k^1}$, yielding  \eqref{proofapprox_2} with probability $\geq 1-Ce^{-2R^{\eta_0}}$.
This finishes the proof.
\end{proof}

Next, we analyze the linearized operator associated with the approximate equation \eqref{profileeqn_2}.

\begin{prop}\label{prop.linear} 
Define the $\Rb$-linear operator $\Lc_0$ to be the linearization of the nonlinearity in (\ref{profileeqn}), i.e. 
\begin{equation}\label{def_L0}
\begin{aligned}
(\Lc_0u)_k(t) = i\sum_{r=3}^A\sum_{i=1}^r
  (\epsilon R^{-1/2})^{r-1}T_1\cdot\sum_{\zeta_j\in\{\pm\}}\sum_{\zeta_1k_1+\cdots+\zeta_rk_r=k}\int_0^{t} 
  e^{iT_1\cdot \Omega_r({t}')}
\\ \, \times i\widetilde{n}_{\zeta_1\cdots\zeta_r}(k_1,\cdots,k_r)\cdot\varphi_{\leq K_{\mathrm{tr}}}(k)\prod_{j=1}^r\widetilde{a}_{k_j}({t}')^{\zeta_j}\,\mathrm{d}{t}'
\\
+ \, iT_1\int_0^{t}(\epsilon^2\Gamma_0(k)+T_1^{-1}\Gamma_1({t}',k))u_k({t}')\mathrm{d}{t}',
\end{aligned}
\end{equation} where the functions $\widetilde{a}$ in \eqref{def_L0} are such that exactly one $\widetilde{a}$ 
(i.e. the one corresponding to $j=i$) equals $u$ and all the others equal $a_{\mathrm{app}}$. 
Then, with probability $\geq 1-Ce^{-R^{2\eta_0}}$, the operator $1-\Lc_0$ is invertible, and satisfies
\begin{equation}\label{est_L0}
\|(1-\Lc_0)^{-1}\|_{C_t^0L_k^1\to C_t^0L_k^1}\lesssim R^8.
\end{equation} 
Moreover, 
the norm $C_t^0L_k^1$ can be replaced by any Lebesgue or Sobolev norm with fixed exponents,
such as the $L_t^\infty \dot{W}_{x}^{4,0}$ norm appearing in Proposition \ref{properror1}.
\end{prop}

\begin{proof}
Fix the norm $Z:=C_t^0L_k^1$. We will construct a (random) parametrix $\Xc$ to $1-\Lc_0$, such that 
\begin{equation}\label{parametrix}
\|\Xc\|_{Z\to Z}\lesssim R^8,\qquad\|\Xc(1-\Lc_0)-1\|_{Z\to Z}\leq \epsilon,\qquad \|(1-\Lc_0)\Xc-1\|_{Z\to Z}\leq \epsilon
\end{equation} with probability $\geq 1-e^{-R^{2\eta_0}}$. It then follows from \eqref{parametrix} 
and Von Neumann series that $1-\Lc_0$ is invertible, and 
\[\|(1-\Lc_0)^{-1}\|_{Z\to Z}\leq \|[\Xc(1-\Lc_0)]^{-1}\|_{Z\to Z}\cdot \|\Xc\|_{Z\to Z}\lesssim R^8.\]

To prove (\ref{parametrix}), we need to introduce the notions of \emph{paired flower trees} 
as follows. A paired flower tree, still denoted by $\Tc$,
is a paired tree in the sense of Definition \ref{deftree} but with one specific unpaired leaf 
$\ff$ fixed, called the \emph{flower}. 
The unique path from the root to the flower is called the \emph{stem}. 
We require that all the branching nodes in the stem are I-branching nodes, 
and define the number of such branching nodes to be the \emph{depth} of the (paired) flower tree. 
The paired flower couple is defined in the same way and is formed by two paired flower trees, 
where we require that the two flowers have opposite signs and are paired.

For any paired flower tree $\Tc$, we define the following modification of the $\Jc_\Tc$ in (\ref{defofjt}): the modification $\widetilde{\Jc}_\Tc$ is an operator defined by
\begin{equation}\label{modofjt}
\begin{aligned}
(\widetilde{\Jc}_\Tc(u))_k({t})
  &:=(\epsilon R^{-1/2})^{r(\Tc)-1}(T_1)^{n_I(\Tc)} 
  \cdot \sum_{(k_\nf)}\prod_{\lf\in\Pc}^{(+)} \psi(k_{\lf})
  \prod_{\hf\in\Lc\backslash\Pc} \psi(k_\hf)^{1/2}\prod_{\nf\in\Tc} \varphi_{\leq K_{\mathrm{tr}}}(k_\nf)
  \\
& \times 
  \prod_{\mf\in\Ic} [i\widetilde{n}_{\zeta_{\mf_1}\cdots\zeta_{\mf_r}}(k_{\mf_1},\cdots,k_{\mf_r})]^{\zeta_\mf}
  \cdot \prod_{\mf\in\Nc} \widetilde{a}_{\zeta_{\mf_1}\cdots\zeta_{\mf_r}}(k_{\mf_1},\cdots,k_{\mf_r})^{\zeta_\mf}
  \\
& \int_{\Dc}\prod_{\mf\in\Ic} e^{iT_1\Omega_\mf(t_\mf)} u_{k_\ff}(t_{\ff^+})^{\zeta_{\ff}} \,\mathrm{d}{t}_\mf
  \cdot 
  \Re \bigg(\prod_{\hf\in\Lc\backslash(\Pc\cup\{\ff\})}g_{k_{\hf}}^{\zeta_{\hf}}\bigg),
\end{aligned}
\end{equation} 
where $\ff$ is the flower and $\ff^+$ is the parent node of $\ff$ (which is on the stem). 


Note that if the paired flower tree $\Tc$ has depth $1$, i.e. the parent of the flower is the root, 
then the operator $\widetilde{\Jc}_\Tc$ exactly corresponds to 
the factor that is summed in \eqref{def_L0} for fixed indexes $r$ and $i$ 
with $a_{\mathrm{app}}$ replaced by the standard $\Jc_\Tc$ terms for various paired trees $\Tc$, 
and the function $u$ 
in the $i-th$ position.
Moreover, the last gauging term in (\ref{def_L0}) exactly corresponds to the case when 
the root $\rf$ is an I-branching node with 3 children, and the two trees at the non-flower children 
nodes of $\rf$ have their leaves exactly paired (again the $\Gamma_0$ term corresponds to 
the case when these two trees are both trivial, 
and the $\Gamma_1$ term corresponds to the case when these two trees are not both trivial). 
Therefore we have
\begin{equation}\label{linear_formula}
\Lc_0(u)=\sum_{\Tc}^{(I)}\widetilde{\Jc}_\Tc(u),
\end{equation} where the summation is taken over all admissible paired flower trees $\Tc$ 
of sign $+$ with depth $1$, such that the subtree at each child node of the root has rank $\leq N$.

We define the parametrix $\Xc$ by
\begin{equation}\label{para_formula}
\Xc(u) := \sum_{\Tc}^{(II)}\widetilde{\Jc}_\Tc(u),
\end{equation} 
where the summation is taken over all admissible paired flower trees $\Tc$ of sign $+$ 
with arbitrary depth (including $0$, in which case $\Tc$ is trivial and $\widetilde{\Jc}_\Tc=1$) 
and total rank $\leq N$.

To calculate $\Lc_0\Xc$ and $\Xc\Lc_0$ we use the correspondence between paired flower trees
and their expressions $\widetilde{\Jc_{\Tc}}$.
Note that we can multiply two paired flower trees $\Tc_1$ and $\Tc_2$ by taking $\Tc_1$ and replacing 
the flower by the paired tree $\Tc_2$ rooted at it 
(assuming the flower of $\Tc_1$ has the same sign as the root of $\Tc_2$; 
if not we should attach this $\Tc_2$ after flipping the signs of all nodes),
so the new flower is just the flower of $\Tc_2$. 
Then we can turn this new tree $\Tc$ into a paired tree by
assigning arbitrary pairings, while respecting the existing pairings in $\Tc_j$.
We denote by $[\Tc_1\Tc_2]$ the set of all paired flower trees $\Tc$ obtained in this way.
By these definitions we can see that
\begin{equation}\label{product}
\widetilde{\Jc}_{\Tc_1}\widetilde{\Jc}_{\Tc_2}=\sum_{\Tc\in[\Tc_1\Tc_2]}\widetilde{\Jc}_\Tc.
\end{equation} 
Then, by (\ref{product}), we can calculate the products $\Xc\Lc_0$ and $\Lc_0\Xc$, to get that
\begin{equation}\label{para_rem_1}
\Xc(1-\Lc_0)-1=\sum_\Tc^{(III)}\widetilde{\Jc}_\Tc,
\end{equation}
where the sum is taken over all admissible paired flower trees $\Tc$ of 
sign $+$ such that (i) it has rank $>N$, 
(ii) the subtrees rooted at the siblings of the flower all have rank $\leq N$, 
and (iii) let $\ff^+$ be the parent of the flower $\ff$, then the rank of $\Tc$ becomes $\leq N$ 
if we collapse the subtree rooted at $\ff^+$ to a single node $\ff^+$. Similarly
\begin{equation}\label{para_rem_2}
(1-\Lc_0)\Xc-1=\sum_\Tc^{(IV)}\widetilde{\Jc}_\Tc,
\end{equation} 
where the sum is taken over all admissible paired flower trees $\Tc$ of sign $+$ such that
(i) it has rank $>N$, 
(ii) the subtrees rooted at the children nodes of the root all have rank $\leq N$.

Now consider either \eqref{para_rem_1} or \eqref{para_rem_2} and denote the operator on the right hand side by $\Ic$;
then it can be divided into two parts $\Ic=\Ic^++\Ic^-$, where $\Ic^\zeta$ restricts the summation to 
$\Tc$ whose flower has sign $\zeta$. 
Each such operator 
has a kernel as follows:
\begin{equation}\label{para_rem_3}
(\Ic^\pm u)_k(t) = \sum_{k'}\int_0^t \Ic^\pm (t,s,k,k')u_{k'}(s)\,\mathrm{d}s,
  \qquad \Ic^\pm (t,s,k,k') := \sum_{\Tc} \widetilde{\Jc}_\Tc(t,s,k,k'),
\end{equation} 
where, cfr. \eqref{modofjt},
\begin{equation*}
\widetilde{\Jc}_\Tc(t,s,k,k') 
  := \big( \widetilde{\Jc}_\Tc ( \delta(s- t_{\ff^+}) \mathbf{1}(k_f = k') ) \big)_k (t).
\end{equation*} 
Then we have that
\begin{equation}\label{para_rem_4}
\Eb| 
 \Ic^\pm (t,s,k,k') |^2
  = \Eb \bigg|\sum_{\Tc}\widetilde{\Jc}_\Tc(t,s,k,k')\bigg|^2 = \sum_\Qc \widetilde{\Kc}_\Qc(t,s,k,k'),
\end{equation} 
where we define, for any paired flower couple $\Qc$, a modification of the $\Kc_\Qc$ in \eqref{defofkq}
given by
\begin{equation}\label{modofkq}
\begin{aligned}
\widetilde{\Kc}_\Qc(t,s,k,k')&:=(\epsilon R^{-1/2})^{r(\Qc)-2}(T_1)^{n_I(\Qc)}
  \cdot
  \sum_{(k_\nf)} \prod_{\lf\in\Lc\backslash\{\ff,\ff'\}}^{(+)}
  \psi(k_{\lf})
  \prod_{\nf\in\Qc}\varphi_{\leq K_{\mathrm{tr}}}(k_\nf)
\\
& \times \mathbf{1}_{k_\ff = k_{\ff'} = k'}
   \prod_{\mf\in\Ic} [i\widetilde{n}_{\zeta_{\mf_1} \cdots\zeta_{\mf_r}}(k_{\mf_1},\cdots,k_{\mf_r})]^{\zeta_\mf}
   \cdot \prod_{\mf\in\Nc} \widetilde{a}_{\zeta_{\mf_1}\cdots\zeta_{\mf_r}}(k_{\mf_1},\cdots,k_{\mf_r})^{\zeta_\mf}
\\
& \times\int_{\Ec}\delta(t_{\ff^+}-s)\delta(t_{(\ff')^+}-s)\prod_{\mf\in\Ic} e^{iT_1\Omega_\mf(t_\mf)}\,\mathrm{d}{t}_\mf,
\end{aligned}
\end{equation}
where $\ff$ and $\ff'$ are the flowers and $\ff^+$ and $(\ff')^+$ are the parents of the flowers,
and in \eqref{para_rem_4} the summation is taken over all flower couples $\Qc=(\Tc_1,\Tc_2)$ 
such that each of $\Tc_1$ and $\Tc_2$ satisfy 
the requirements in the definitions of $(III)$ or $(IV)$.

Now it suffices to estimate the right-hand side of \eqref{para_rem_4}.
We can apply the same arguments as in Subsection \ref{sec.L1est}, and Sections \ref{sec.irre}
and \ref{sec.molecule} used to estimate $\Kc_\Qc$, 
with minor adjustments adapted to $\widetilde{\Kc}_\Qc$.
Note that the set of  $\Qc=(\Tc_1,\Tc_2)$ in \eqref{para_rem_4} is invariant under twisting, provided that the irregular 
chain involved in the twisting does not contain any child node of the root or the parent node of the flower. 
These exceptions lead at most to an $O(R^4)$ loss, and the rest of the proof is the same as for the case of $\Kc_\Qc$. 
Since the rank $r(\Qc)>N$, we then get that
\[
 \sum_\Qc \widetilde{\Kc}_\Qc(t,s,k,k') \lesssim R^4 \epsilon^{\theta N/8}.\] 
By the large deviation estimate for 
Gaussians, see Proposition \ref{initialprop}, in the same way as before we get
\[\sup_{t,s,k'k'}|\Ic^\pm(t,s,k,k')| \lesssim R^4 \epsilon^{\theta N/9} \] 
with probability $\geq 1- Ce^{-R^{2\eta_0}}$. 
Since $|k|,|k'|\leq 2^{K_{\mathrm{tr}}} \leq R^{1/A}$
and $0\leq t,s\leq 1$,
the above inequality and \eqref{para_rem_3} imply the second and third bounds in \eqref{parametrix}.

To prove the first bound in \eqref{parametrix}, we write, similarly to the above
\begin{equation}\label{para_rem_3+}
(\Xc u)_k(t) = \sum_{k'}\int_0^t \Xc
(t,s,k,k')u_{k'}(s)\,\mathrm{d}s,
   \qquad \Xc (t,s,k,k') := \sum_{\Tc}^{(II)} \widetilde{\Jc}_\Tc(t,s,k,k'),
\end{equation} 
where the sum is as in \eqref{para_formula}. We can then estimate
\begin{equation*}
\Eb|(\Xc u)_k(t)|^2 = \sum_\Qc \widetilde{\Kc}_\Qc(t,s,k,k') \lesssim R^4,
\end{equation*} 
which leads to the bound for $\Xc$ in \eqref{parametrix} with probability $1-e^{-R^\eta}$.
This completes the proof.
\end{proof}


\subsection{Proof of Proposition \ref{propdiff}}\label{ssecprpropdiff}
Using the above results we can now prove Proposition \ref{propdiff}.
Define $b_k(t) := (a_k)_{\mathrm{tr}}(t)-(a_k)_{\mathrm{app}}(t)$. 
Then by \eqref{profileeqn} and \eqref{profileeqn_2} we get that
\[(1-\Lc_0)b=\Rc+O(|b|^2),\] where $\Lc_0$ is defined as in (\ref{def_L0}), 
$\Rc$ is defined as in (\ref{profileeqn_2}), and $O(|b|^2)$ is a term that is at least quadratic in $b$. 
By Propositions \ref{proofapprox} and \ref{prop.linear} and a fixed point argument, 
we know that, with probability $\geq 1-Ce^{-R^{2\eta_0}}$, this equation for $b$ 
has a unique solution in the space $\|b\|_Z\leq \epsilon^{N\theta/12}$;
the same holds if one replaces the $Z=C_t^0L_k^1$ norm by a norm with a $\langle k\rangle^{10}$ weight. 


\subsection{Proof of \eqref{error1.2'}
}\label{ssecprerror1.2'}
Suppose $W$ solves the equation \eqref{error1.1'}, and apply to it the transform \eqref{profile} 
(which transforms $w_{\mathrm{tr}}$ to $a_{\mathrm{tr}}$) to get 
\begin{align*}
a_k(t) = (R^{1/2} \epsilon)^{-1} e^{iT_1(|k|^{1/2}t + \Gamma(t,k)}\widehat{W}(T_1t,k) =: \Pc(W)_k.
\end{align*}
This profile satisfies the equation
\begin{align*}
a_k(t) = (R^{1/2} \epsilon)^{-1} \widehat{W}(0) + \Lc(a)_k(t) + \int_0^t\Rc'_k(s) ds, \qquad  \Rc'_k(s) := \Pc(\Rc)_k
\end{align*}
for $t\leq 1$, where here $\Lc$ is defined in the same way as $\Lc_0$ in \eqref{def_L0}, but with $a_{\mathrm{app}}$ 
replaced by $a_{\mathrm{tr}}$.
We can then estimate, for $Z = C^0_tL^1_k$,
\begin{align*}
\| (1-\Lc_0) a \|_Z \leq R^{-1/2} \epsilon^{-1} \| \widehat{W}(0) \|_{L^1_k} 
  + \big\| (\Lc - \Lc_0)(a) \big\|_Z + \big\| \Rc' \big\|_Z .
\end{align*}
Using the estimate \eqref{est_L0}, the above equation implies
\begin{align}\label{estWZ}
\| \widehat{W} \|_Z \lesssim R^8 \big( \| \widehat{W}(0) \|_{L^1_k} + \| \Rc  \|_Z \big)
\end{align}
provided we can bound
\begin{align}\label{L-L0}
\big\| (\Lc - \Lc_0)u \big\|_Z \leq R^{-9} \| u \|_Z.
\end{align}
From the definitions of $\mathcal{L}_0$ in \eqref{def_L0} and $\mathcal{L}$ above, we see that
\begin{align*}
(\Lc - \Lc_0)u & = i\sum_{r=3}^A\sum_{i=1}^r
  (\epsilon R^{-1/2})^{r-1}T_1\cdot\sum_{\zeta_j\in\{\pm\}}\sum_{\zeta_1k_1+\cdots+\zeta_rk_r=k}\int_0^{t} 
  e^{iT_1\cdot \Omega_r({t}')}
  \\
  & \times [i \widetilde{n}_{\zeta_1\cdots\zeta_r}(k_1,\cdots,k_r)]^{\zeta_n}
  \cdot\varphi_{\leq K_{\mathrm{tr}}}(k)\prod_{j=1}^r\widetilde{a}_{k_j}({t}')^{\zeta_j}\,\mathrm{d}{t}'
\end{align*}
where, in the last product, exactly one of the $\widetilde{a}$ is $u$, at least one of the $\widetilde{a}$
is $b = a_{\mathrm{tr}}-a_{\mathrm{app}}$, and the others are $a_{\mathrm{app}}$.
Since $\| b \|_Z \lesssim \epsilon^{N\theta/12}$ is small, 
using also \eqref{ansatz3} and \eqref{jrest1} to bound $a_{\mathrm{app}}$, via Proposition \ref{initialprop},
we obtain \eqref{L-L0} with probability $\geq 1 - Ce^{-2R^{\eta_0}}.$ 

Finally, the desired \eqref{error1.2'} follows in the same way as \eqref{estWZ}
using that the bound \eqref{est_L0} for $(1-\Lc_0)^{-1}$ 
also holds if we measure in the $L_t^\infty \dot{W}_x^{4,0}$ norm of $W(t,x)$ 
instead of the $C_t^0L_k^1$ norm of the corresponding profile $a_k(t)=\Pc(W)$, 
as stated in Proposition \ref{prop.linear}, and that the same is true for the bound \eqref{L-L0}.

\section{Normal forms and proof of Proposition \ref{normalprop}}\label{sectionNormal}

We assume in this section that $A\geq 3$ is a fixed integer and $(h,\phi)\in C([T_1,T_2]:H^{N_2}\times \dot{H}^{N_2+1/2,1/2})$ is a solution of the system 
\begin{equation}\label{ww0repeat}
\begin{split}
&\partial_th=G(h)\phi,\\
&\partial_t\phi=-h-\frac{1}{2}(\partial_x\phi)^2+\frac{(G(h)\phi+\partial_xh\cdot\partial_x\phi)^2}{2(1+(\partial_xh)^2)},
\end{split}
\end{equation} 
satisfying the bounds \eqref{aprio}, i.e.
\begin{equation}\label{apriorepeat}
\|(h+i|\partial_x|^{1/2}\phi)(t)\|_{H^{N_2}}\leq B_2,\qquad \|(h+i|\partial_x|^{1/2}\phi)(t)\|_{\dot{W}^{N_\infty,0}}\leq \epsilon_\infty,
\end{equation}
for some parameters $N_2\geq 2A$, $N_\infty\in[4,N_2-4]$, $B_2\geq 1$, $\eps_\infty\ll1$, and for any $t\in[T_1,T_2]$. 

We recall a simple lemma from \cite{DIP}, see Lemmas 2.1 and 2.2:

\begin{lem}\label{touse} 

(i) For any $m\in S^\infty$ we define the associated multiplinear operator $M$ by
\begin{equation}\label{Moutm}
\mathcal{F}\big[M_m(f_1,\ldots,f_d)\big](\xi):=\frac{1}{(2\pi R)^{d-1}}\sum_{\eta_1,\ldots,\eta_d\in\mathbb{Z}/R,\,\eta_1+\ldots+\eta_d=\xi}m(\eta_1,\ldots,\eta_d)\widehat{f}(\eta_1)\cdot\ldots\cdot \widehat{f_d}(\eta_d).
\end{equation}
Then
\begin{equation}\label{Moutm2}
\begin{split}
&M_m(f_1,\ldots,f_d)(x)=\int_{\Tb_R^d}K_m(x-y_1,\ldots,x-y_d)f(y_1)\cdot\ldots\cdot f(y_d)\,dy_1\ldots dy_d,\\
&K_m(z_1,\ldots,z_d):=\frac{1}{(2\pi R)^d}\sum_{\eta_1,\ldots,\eta_d\in\Zb/R}e^{i(z_1\eta_1+\ldots+z_d\eta_d)}m(\eta_1,\ldots,\eta_d).
\end{split}
\end{equation}
In particular $\|K_m\|_{L^1(\Tb_R^d)}={\| m \|}_{S^\infty}$. Therefore, if $p_1,\ldots,p_d,r\in[1,\infty]$ satisfy $1/p_1+\ldots+1/p_d=1/r$ and $f_1,\ldots,f_d\in L^2(\mathbb{T}_R)$ then
\begin{equation}\label{mk6}
\|M(f_1,\ldots,f_d)\|_{L^r(\mathbb{T}_R)} \leq \|m\|_{S^\infty} \|f_1\|_{L^{p_1}(\mathbb{T}_R)}\cdot\ldots\cdot\|f_d\|_{L^{p_d}(\mathbb{T}_R)}.
\end{equation}

(ii) If $m,m'\in S^\infty$ then $m\cdot m'\in S^\infty$ and
\begin{equation}\label{al8}
\|m\cdot m'\|_{S^\infty}\leq \|m\|_{S^\infty}\|m'\|_{S^\infty}.
\end{equation}

(iii) Assume that $d\geq 1$, $m:\Rb^d\to\mathbb{C}$ is a continuous compactly supported function, and $\check{m}\in L^1(\Rb^d)$ is the Euclidean inverse Fourier transform of $m$. Then
\begin{equation}\label{al8.6}
\Big\|\frac{1}{(2\pi R)^d}\sum_{\xi_1,\ldots,\xi_d\in\Zb/R}m(\xi_1,\ldots,\xi_d)e^{i(x_1\xi_1+\ldots+x_d\xi_d)}\Big\|_{L^1_x(\Tb_R^d)}\lesssim \|\check{m}\|_{L^1(\Rb^d)}.
\end{equation}
\end{lem}

The first main issue is to understand the expansion of the Dirichlet-Neumann operator $G(h)$. We prove the following: 

\begin{lem}\label{DNexpansion}

(i) We have $G(h)\phi\in C([T_1,T_2]: H^{N_2-3/2})$ and, for any $t\in[T_1,T_2]$,
\begin{equation}\label{adi7}
\begin{split}
&G(h)\phi=|\partial_x|\phi+\sum_{n=2}^AG_n(h,\phi)+G_{>A},\\
&G_2(h,\phi)=-\partial_x(h\partial_x\phi)-|\partial_x|(h|\partial_x|\phi),\\
&\|G_{>A}\|_{\dot{H}^{N_2-3/2,-3/4}}\lesssim\eps_\infty^{A}B_2,\qquad \|G_{>A}\|_{\dot{W}^{N_\infty-3/2,-3/4}}\lesssim\eps_\infty^{A+1},\\
&\|G_n(h,\phi)\|_{\dot{H}^{N_2-3/2,-3/4}}\lesssim\eps_\infty^{n-1}B_2,\qquad \|G_n(h,\phi)\|_{\dot{W}^{N_\infty-3/2,-3/4}}\lesssim\eps_\infty^{n},\qquad n\in\{2,\ldots,A\}.
\end{split}
\end{equation}
In addition, the functions $G_n(h,\phi)$, $n\in\{2,\ldots,A\}$ are defined by multipliers $q_n$ satisfying
\begin{equation}\label{adi8}
\begin{split}
&\widehat{G_n(h,\phi)}(\xi)=\frac{1}{(2\pi R)^{n-1}}\sum_{\eta_1,\ldots,\eta_n\in\Zb/R,\,\eta_1+\ldots+\eta_n=\xi}q_n(\eta_1,\ldots,\eta_n)\widehat{\phi}(\eta_1)\widehat{h}(\eta_2)\cdot\ldots\cdot\widehat{h}(\eta_n),\\
&q_n(\lambda\underline\eta)=\lambda^nq_n(\underline\eta),\qquad\text{ for any }\,\,\lambda\in (0,\infty),\,\,\underline{\eta}\in\Rb^n,\\
&\big\|q_n(\eta_1,\ldots,\eta_n)\varphi_k(\eta_1+\ldots+\eta_n)\varphi_{k_1}(\eta_1)\cdot\ldots\cdot\varphi_{k_n}(\eta_n)\big\|_{S^\infty}\lesssim 2^{k+\widetilde{k}_1+\ldots+\widetilde{k}_{n-1}}(1+|\widetilde{k}_n-\widetilde{k}_1|)^A.
\end{split}
\end{equation}
The $S^\infty$ inequality in the last line of \eqref{adi8} holds for any $k,k_1,\ldots k_n\in\Zb$, where $\widetilde{k}_1\leq \widetilde{k}_2\leq\ldots\leq \widetilde{k}_n$ denotes the nondecreasing rearrangement of $k_1,\ldots,k_n$.

(ii) As a consequence, the functions $h,\phi$ satisfy the equations 
\begin{equation}\label{adi12.2}
\partial_t h=|\partial_x|\phi+\sum_{n=2}^AG_n(h,\phi)+G_{>A},\qquad\partial_t\phi=-h+\sum_{n=2}^AH_n(h,\phi)+H_{>A},
\end{equation}
where $G_n,G_{>A}$ are as in \eqref{adi7}--\eqref{adi8} and the functions $H_n$ and $H_{>A}$ satisfy
\begin{equation}\label{adi12.3}
\begin{split}
&H_2(h,\phi)=-\frac{1}{2}(\partial_x\phi)^2+\frac{1}{2}(|\partial_x|\phi)^2,\\
&\|H_{>A}\|_{H^{N_2-3/2}}\lesssim\eps_\infty^{A}B_2,\qquad \|H_{>A}\|_{\dot{W}^{N_\infty-3/2,1/10}}\lesssim\eps_\infty^{A+1},\\
&\|H_n(h,\phi)\|_{H^{N_2-3/2}}\lesssim\eps_\infty^{n-1}B_2,\qquad \|H_n(h,\phi)\|_{\dot{W}^{N_\infty-3/2,1/10}}\lesssim\eps_\infty^{n},\qquad n\in\{2,\ldots,A\},
\end{split}
\end{equation}
for any $t\in[T_1,T_2]$. In addition, the functions $H_n(h,\phi)$, $n\in\{3,\ldots,A\}$, are defined by multipliers $r_n$ satisfying
\begin{equation}\label{adi12.4}
\begin{split}
&\widehat{H_n(h,\phi)}(\xi)=\frac{1}{(2\pi R)^{n-1}}\sum_{\eta_1,\ldots,\eta_n\in\Zb/R,\,\eta_1+\ldots+\eta_n=\xi}r_n(\eta_1,\ldots,\eta_n)\widehat{\phi}(\eta_1)\widehat{\phi}(\eta_2)\widehat{h}(\eta_3)\cdot\ldots\cdot\widehat{h}(\eta_n),\\
&r_n(\lambda\underline\eta)=\lambda^nq_n(\underline\eta),\qquad\text{ for any }\,\,\lambda\in (0,\infty),\,\,\underline{\eta}\in\Rb^n,\\
&\big\|r_n(\eta_1,\ldots,\eta_n)\varphi_{k_1}(\eta_1)\cdot\ldots\cdot\varphi_{k_n}(\eta_n)\big\|_{S^\infty}\lesssim 2^{k_1+\ldots+k_n}(1+|\widetilde{k}_n-\widetilde{k}_1|)^A.
\end{split}
\end{equation}

\end{lem}

\begin{proof} (i) We fix $t\in[T_1,T_2]$ and use some of the results in \cite[Section 8]{DIP}. Loss of derivative is not important here, so, for the sake of simplicity, we do not attempt to optimize this aspect.

{\bf{Step 1.}} We start from the formulas (8.18) and (8.14) in \cite{DIP},
\begin{equation}\label{adi1}
G(h)\phi=-\partial_x\psi,\qquad (I+T_1)\psi=(H_0+T_2)\phi,
\end{equation}
where the operators $H_0$ (the Hilbert transform), $T_1$, $T_2$ are defined by the formulas
\begin{equation}\label{na20}
(H_0 f)(\alpha):=\frac{-1}{\pi }\mathrm{p.v.} \int_{\mathbb{R}}\frac{f^\#(\beta)}{\alpha-\beta}\,d\beta,
\end{equation}
\def\be{\beta}
\begin{equation}\label{na21}
\begin{split}
&(T_1 f)(\alpha):=\frac{1}{\pi}\int_{\mathbb{R}}
\frac{h^\#(\alpha)-h^\#(\beta)-\partial_\be h^\#(\beta)(\alpha-\beta)}{|\gamma(\alpha)-\gamma(\beta)|^2}f^\#(\beta)\,d\beta,\\
&(T_2 f)(\alpha):=\frac{1}{\pi}\int_{\mathbb{R}}
\frac{h^\#(\alpha)-h^\#(\beta)-\partial_\be h^\#(\beta)(\alpha-\beta)}{|\gamma(\alpha)-\gamma(\beta)|^2}\frac{h^\#(\alpha)-h^\#(\beta)}{\alpha-\beta}f^\#(\beta)\,d\beta.
\end{split}
\end{equation}
Here $h^\#:\Rb\to\Rb$ and $f^\#:\Rb\to\mathbb{C}$ denote the periodic extensions of the functions $h$ and $f$, and $\gamma(x):=x+ih^\#(x)$. The operators $T_1$ and $T_2$ admit expansions
\begin{equation}\label{na23}
T_1=\sum_{n\geq 0}(-1)^{n}R_{2n},\qquad T_2=\sum_{n\geq 0}(-1)^{n}R_{2n+1}.
\end{equation}
where the homogeneous real operators $R_n$ are given by
\begin{equation}\label{na8}
 (R_nf)(\alpha):=\frac{1}{\pi}\int_{\mathbb{R}}
\frac{h^\#(\alpha)-h^\#(\beta)-\partial_\beta h^\#(\beta)(\alpha-\beta)}{(\alpha-\beta)^2}\Big(\frac{h^\#(\alpha)-h^\#(\beta)}{\alpha-\beta}\Big)^nf^\#(\beta)\,d\beta.
\end{equation}
The operators $R_n$ can be defined in terms of Fourier multipliers, in the form
\begin{equation}\label{na13.3}
\widehat{R_nf}(\xi)=\frac{1}{(2\pi R)^{n+1}}\sum_{\eta_0,\ldots,\eta_{n+1}\in \Zb/R,\,\eta_0+\ldots+\eta_{n+1}=\xi}M_{n+1}(\eta_0,\underline{\eta})\widehat{f}(\eta_0)\widehat{h}(\eta_1)\cdot\ldots\cdot\widehat{h}(\eta_{n+1}),
\end{equation}
where $\underline{\eta}:=(\eta_1,\ldots,\eta_{n+1})$ and 
\begin{equation}\label{na13.2}
M_{n+1}(\eta_0,\underline{\eta}):=\frac{-i\eta_0}{\pi(n+1)}\int_{\mathbb{R}}e^{-i\eta_0\rho}\prod_{l=1}^{n+1}\frac{1-e^{-i\eta_l\rho}}{\rho}\,d\rho.
\end{equation}
See the identities (8.41)-(8.42) in \cite{DIP}. 

The formula \eqref{na13.2} shows that the symbols $M_{n+1}$ are homogeneous of degree $n+1$,
\begin{equation}\label{adi2}
M_{n+1}(\lambda\eta_0,\lambda\underline{\eta})=\lambda^{n+1}M_{n+1}(\eta_0,\underline{\eta}),\qquad\lambda\in (0,\infty),\eta_0\in\Rb,\underline{\eta}\in\Rb^{n+1}.
\end{equation}
Moreover, when $n=0$ we calculate explicitly
\begin{equation}\label{adi6.9}
\begin{split}
&M_1(\eta_0,\eta_1)=\eta_0[\mathrm{sgn}\,(\eta_0+\eta_1)-\mathrm{sgn}\,(\eta_0)],\\
&\widehat{R_0f}(\xi)=\frac{1}{2\pi R}\sum_{\eta_0\in\Zb/R}\big[\mathrm{sgn}(\xi)-\mathrm{sgn}(\eta_0)\big]\eta_0\widehat{f}(\eta_0)\widehat{h}(\xi-\eta_0).
\end{split}
\end{equation}
Notice also that 
\begin{equation}\label{Wie1}
\begin{split}
&M_{n+1}(\eta_0,\eta_1,\ldots,\eta_{n+1})=0\qquad\text{ if }\eta_0\cdot\ldots\cdot\eta_{n+1}=0,\\
&M_{n+1}(\eta_0,\eta_1,\ldots,\eta_{n+1})=0\qquad\text{ if }|\eta_0|>|\eta_1|+\ldots+|\eta_{n+1}|.
\end{split}
\end{equation}
Indeed, the identity in the first line follows directly from the definition \eqref{na13.2}, while the identity in the second line follows by induction over $n$ once we observe that 
\begin{equation*}
\partial_{\eta_1}M_{n+1}(\eta_0,\eta_1,\ldots,\eta_{n+1})=(-i)\frac{n}{n+1}M_{n}(\eta_0+\eta_1,\eta_2,\ldots,\eta_{n+1}).
\end{equation*}

We show now that if $n\geq 1$ and $k_0,\ldots,k_{n+1}\in\overline{\Zb}$ then
\begin{equation}\label{Wie2}
\big\|M_{n+1}(\eta_0,\ldots,\eta_{n+1})\cdot\varphi_{k_0}(\eta_0)\ldots\varphi_{k_{n+1}}(\eta_{n+1})\big\|_{S^\infty}\leq C_1^n2^{k_0}2^{\widetilde{k}_1+\ldots+\widetilde{k}_n}(1+|\widetilde{k}_{n+1}-\widetilde{k}_n|),
\end{equation}
where $\widetilde{k}_1\leq\ldots\leq\widetilde{k}_{n+1}$ is the nondecreasing rearrangement of $k_1,\ldots,k_{n+1}$ and $C_1\geq 1$ is a constant. Indeed, we have
\begin{equation}\label{Wie3}
\begin{split}
&\mathcal{F}^{-1}(M_{n+1,\underline{k}})(x_0,\ldots,x_{n+1})=C_n\int_{\Rb}\widetilde{\psi}_{k_0}(x_0-\rho)\prod_{l=1}^{n+1}\frac{\psi_{k_l}(x_l)-\psi_{k_l}(x_l-\rho)}{\rho}\,d\rho,\\
&\psi_k(x):=\frac{1}{2\pi R}\sum_{\eta\in\Zb/R}e^{ix\eta}\varphi_k(\eta),\qquad \widetilde{\psi}_k(x):=\frac{1}{2\pi R}\sum_{\eta\in\Zb/R}e^{ix\eta}(\eta/2^k)\varphi_k(\eta),
\end{split}
\end{equation}
where $M_{n+1,\underline{k}}(\eta_0,\ldots,\eta_{n+1}):=M_{n+1}(\eta_0,\ldots,\eta_{n+1})\cdot\varphi_{k_0}(\eta_0)\ldots\varphi_{k_{n+1}}(\eta_{n+1})$, $\underline{k}:=(k_0,\ldots,k_{n+1})$. This follows easily from \eqref{na13.2}. Using \eqref{al8.6} we see easily that
\begin{equation*}
\|\psi_k\|_{L^1(\Tb_R)}+\|\widetilde{\psi}_k\|_{L^1(\Tb_R)}\lesssim 1,\qquad \|\psi_k(x)-\psi_k(x-\rho)\|_{L^1_x(\Tb_R)}\lesssim \min(1,2^k|\rho|)
\end{equation*}
for any $k\in\Zb$ and $\rho\in\Rb$. The desired bounds \eqref{Wie2} follow from \eqref{Wie3} by integration in $\rho$.

{\bf{Step 2.}} To prove \eqref{adi7}--\eqref{adi8} we start from the formula \eqref{adi1}, so
\begin{equation}\label{adi6}
\psi=(I-T_1+T_1^2+\ldots+(-1)^{A-1}T_1^{A-1})(H_0+T_2)\phi+(-1)^{A}T_1^{A}\psi.
\end{equation}
We use the expansions \eqref{na23} as well as the bounds
\begin{equation}\label{adi6.5}
\begin{split}
&\|R_nf\|_{\dot{H}^{N_2-1/2,1/10}}\leq B_2(C_2\eps_\infty)^n\||\partial_x|^{1/2}f\|_{\dot{W}^{1,0}},\\
&\|R_nf\|_{\dot{W}^{N_\infty-1/2,1/10}}\leq (C_2\eps_\infty)^{n+1}\||\partial_x|^{1/2}f\|_{\dot{W}^{1,0}},
\end{split}
\end{equation}
for some constant $C_2\geq 1$ and any $n\geq 0$, which follow from \eqref{adi6.9}--\eqref{Wie2} and \eqref{mk6}. Since $\||\partial_x|^{1/2}\psi\|_{\dot{W}^{1,0}}\lesssim\eps_\infty$ (due to estimate (8.20) in \cite{DIP}), we have
\begin{equation}\label{adi6.6}
\|T_1^{A}\psi\|_{\dot{H}^{N_2-1/2,1/10}}\lesssim B_2\eps_\infty^{A},\qquad \|T_1^{A}\psi\|_{\dot{W}^{N_\infty-1/2,1/10}}\lesssim \eps_\infty^{A+1}.
\end{equation}

We examine now the main term in \eqref{adi6}, expand the operators $T_1,T_2$ as in \eqref{na23}, and use the bounds \eqref{adi6.5} to estimate all the terms of homogeneity $\geq A+1$ as in \eqref{adi6.6}. We can thus write
\begin{equation}\label{adi6.7}
\begin{split}
&\psi=H_0\phi-R_0H_0\phi+\sum_{n=3}^A\widetilde{G}_n(h,\phi)+\widetilde{G}_{>A},\\
&\|\widetilde{G}_{>A}\|_{\dot{H}^{N_2-1/2,1/10}}\lesssim B_2\eps_\infty^{A},\qquad \|\widetilde{G}_{>A}\|_{\dot{W}^{N_\infty-1/2,1/10}}\lesssim \eps_\infty^{A+1},\\
&\widetilde{G}_n(h,\phi)=\sum_{l\geq 1,\,\iota\in\{0,1\},\,a_1+\ldots+a_l=n-l-1}c_{\iota,a_1,\ldots,a_l}R_{a_1}\ldots R_{a_l}H_0^\iota\phi,
\end{split}
\end{equation}
for some suitable coefficients $c_{\iota,a_1,\ldots,a_l}\in\Zb$. Then we can finally define
\begin{equation}\label{adi6.10}
G_2(h,\phi):=\partial_xR_0H_0\phi,\qquad G_n(h,\phi):=-\partial_x\widetilde{G}_n(h,\phi),n\geq 3,\qquad G_{>A}:=-\partial_x\widetilde{G}_{>A}.
\end{equation}
The identities and the estimates in \eqref{adi7} follow from \eqref{adi6.7}, \eqref{adi6.5}, and the observation that $\partial_xR_0H_0\phi=-\partial_x(h\partial_x\phi)-|\partial_x|(h|\partial_x|\phi)$. 

The formulas \eqref{na13.3}--\eqref{na13.2} show that if $a_1+\ldots+a_l+l=n-1$ then
\begin{equation}\label{adi6.8}
\begin{split}
&\mathcal{F}\big[R_{a_1}\ldots R_{a_l}f\big](\xi)=\frac{1}{(2\pi R)^{n-1}}\sum_{\eta_0,\ldots,\eta_{n-1}\in\Zb/R,\,\eta_0+\ldots+\eta_{n-1}=\xi}\widetilde{q}_{a_1,\ldots,a_l}(\eta_0,\eta_1,\ldots,\eta_{n-1})\\
&\qquad\qquad\qquad\qquad\times\widehat{f}(\eta_0)\widehat{h}(\eta_1)\cdot\ldots\cdot\widehat{h}(\eta_{n-1}),\\
&\widetilde{q}_{a_1,\ldots,a_l}(\eta_0,\eta_1,\ldots,\eta_{n-1}):=M_{a_1+1}(\eta_0+s_1,\eta_1,\ldots,\eta_{a_1+1})\\
&\qquad\qquad\qquad\times\ldots\times M_{a_l+1}(\eta_0+s_l,\eta_{a_1+\ldots+a_{l-1}+l},\ldots,\eta_{a_1+\ldots+a_{l}+l}).
\end{split}
\end{equation}
where
\begin{equation*}
s_m:=\eta_{a_1+a_2+\ldots+a_m+m+1}+\ldots+\eta_{n-1}\qquad\text{ for }m=1,\ldots,l.
\end{equation*}
The homogeneity property \eqref{adi8} follows from \eqref{adi2}. To estimate the localized $S^\infty$ norms we assume $k,k_0,k_1,\ldots,k_{n-1}\in\overline{\Zb}$ and let $K_{a_1\ldots a_l;\iota}^{k;k_0k_1\ldots k_{n-1}}$ denote the kernels defined by the localized multipliers $\widetilde{q}_{a_1,\ldots,a_l}$, i.e.
\begin{equation}\label{Wie5}
\begin{split}
\mathcal{F}(K_{a_1,\ldots, a_l;\iota}^{k;k_0,\ldots, k_{n-1}})(\eta_0,\ldots,\eta_{n-1})&=\widetilde{q}_{a_1,\ldots,a_l}(\eta_0,\ldots,\eta_{n-1})\varphi_{k_0}(\eta_0)\cdot\ldots\cdot\varphi_{n-1}(\eta_{n-1})\\
&\times\varphi_{k}(\eta_0+\ldots+\eta_{n-1})(i\mathrm{sgn}\,\eta_0)^\iota.
\end{split}
\end{equation}
We will prove that 
\begin{equation}\label{Wie6}
\|K_{a_1,\ldots, a_l;\iota}^{k;k_0,\ldots, k_{n-1}}\|_{L^1(\Tb^n)}\lesssim 2^{k_0+k_1+\ldots+k_{n-1}}2^{-\widetilde{k}_{n-1}}(1+|\widetilde{k}_{n-1}-\widetilde{k}_1|)^l,
\end{equation}
where $\widetilde{k}_1\leq\ldots\leq \widetilde{k}_{n-1}$ denotes the nondecreasing rearrangement of $k_1,\ldots,k_{n-1}$.

We prove the bounds \eqref{Wie6} by induction over $l$. If $l=1$ then $n=a_1+2$, and the bounds \eqref{Wie6} follow from \eqref{adi6.9}--\eqref
{Wie2} and \eqref{mk6}. To prove the induction step, assume that the bounds \eqref{Wie6} hold for all $l'\leq l$. Assume that $a_1,\ldots,a_{l+1}$ are given, $a_1+\ldots+a_{l+1}+l+1=n-1$, $k,k_0,\ldots,k_{n-1}\in\overline{\Zb}$, and let $K_{a_1;0}^{k;p,k_1,\ldots,k_{a_1+1}}$ and $K_{a_2,\ldots, a_{l+1};\iota}^{p;k_0,k_{a_1+2},\ldots, k_{n-1}}$ denote the localized kernels corresponding to the operators $R_{a_1}$ and $R_{a_2,\ldots,a_{l+1}}H_0^\iota$ respectively. The main observation is that
\begin{equation*}
\begin{split}
K_{a_1,\ldots, a_{l+1};\iota}^{k;k_0,\ldots, k_{n-1}}(x_0,\ldots,x_{n-1})&=\sum_{p\in\Z}\int_{\Tb_R} K_{a_1;0}^{k;n,k_1,\ldots,k_{a_1+1}}(\alpha, x_1,\ldots,x_{a_1+1})\\
&\times K_{a_2,\ldots, a_{l+1};\iota}^{n;k_0,k_{a_1+2},\ldots, k_{n-1}}(x_0-\alpha, x_{a_1+2}-\alpha, \ldots, x_{n-1}-\alpha)\,d\alpha.
\end{split}
\end{equation*}
In particular
\begin{equation}\label{Wie7}
\|K_{a_1,\ldots, a_l;\iota}^{k;k_0,\ldots, k_{n-1}}\|_{L^1(\Tb^n)}\leq \sum_{p\in\Z}\|K_{a_1;0}^{k;p,k_1,\ldots,k_{a_1+1}}\|_{L^1(\Tb^{a_1+2})}\|K_{a_2,\ldots, a_{l+1};\iota}^{p;k_0,k_{a_1+2},\ldots, k_{n-1}}\|_{L^1(\Tb^{n-a_1-1})}.
\end{equation}
Moreover, $K_{a_1;0}^{k;p,k_1,\ldots,k_{a_1+1}}\equiv 0$ unless $2^p\lesssim 2^{k_{max,1}}$ where $k_{max,1}:=\max(k_1,\ldots,k_{a_1+1})$ and 
$K_{a_2,\ldots, a_{l+1};\iota}^{p;k_0,k_{a_1+2},\ldots, k_{n-1}}\equiv 0$ unless $2^p\lesssim 2^{k_{max,2}}$ where $k_{max,2}:=\max(k_{a_1+2},\ldots,k_{n-1})$, due to \eqref{Wie1}. Using \eqref{Wie7} and the induction hypothesis on the kernels $K_{a_1;0}^{k;p,k_1,\ldots,k_{a_1+1}}$ and $K_{a_2,\ldots, a_{l+1};\iota}^{p;k_0,k_{a_1+2},\ldots, k_{n-1}}$, we can finally estimate
\begin{equation*}
\begin{split}
\|K_{a_1,\ldots, a_l;\iota}^{k;k_0,\ldots, k_{n-1}}\|_{L^1(\Tb^n)}&\lesssim\sum_{2^p\lesssim 2^{\min(k_{max,1},k_{max,2})}}\frac{2^p2^{k_1+\ldots+k_{a_1+1}}}{2^{k_{\max,1}}}\frac{2^{k_0}2^{k_{a_1+2}+\ldots+k_{n-1}}}{2^{k_{\max,2}}}(1+|\widetilde{k}_{n-1}-\widetilde{k}_1|)^{l+1}\\
&\lesssim 2^{k_0+k_1+\ldots+k_{n-1}}2^{-\max(k_{max,1},k_{max,2})}(1+|\widetilde{k}_{n-1}-\widetilde{k}_1|)^{l+1}.
\end{split}
\end{equation*}
This completes the induction step and the proof of \eqref{Wie6}.

(ii) The functions $\partial_th$ and $\partial_t\phi$ can be expressed using the main equations \eqref{ww0repeat} and the expansion \eqref{adi7}. The bounds in \eqref{adi12.3} follow from the bounds in \eqref{adi7} and simple algebra properties (see \cite[Lemma 2.4]{DIP}), while the $S^\infty$ bounds in \eqref{adi12.4} follow from \eqref{adi8}.
\end{proof} 

\subsection{The first normal form} As in \cite[Lemma 5.1]{IoPu2} we define the new variables
\begin{align}\label{defv}
\begin{split}
&v:= h^\ast+i|\partial_x|^{1/2}\phi^\ast,\\
&h^\ast:= h+|\partial_x|\big(H_0h \cdot H_0h\big)/2,\qquad \phi^\ast:=\phi + H_0\big(H_0h\cdot |\partial_x|\phi \big).
  \end{split}
\end{align}
Then we calculate
\begin{equation}\label{adi12.1}
\begin{split}
\partial_th^\ast&=\partial_th+|\partial_x|\big(H_0h \cdot H_0\partial_th\big),\\
\partial_t\phi^\ast&=\partial_t\phi + H_0\big(H_0\partial_th\cdot |\partial_x|\phi \big)+H_0\big(H_0h\cdot |\partial_x|\partial_t\phi \big).
\end{split}
\end{equation}

We can thus calculate
\begin{equation}\label{adi12.5}
(\partial_t+i|\partial_x|^{1/2})v=[\partial_th^\ast-|\partial_x|\phi^\ast]+i|\partial_x|^{1/2}[\partial_t\phi^\ast+h^\ast].
\end{equation}
Using the identities \eqref{adi12.2}, we calculate
\begin{equation*}
\begin{split}
\partial_th^\ast-|\partial_x|\phi^\ast&=-|\partial_x|\big[\phi+H_0\big(H_0h\cdot |\partial_x|\phi \big)\big]+|\partial_x|\phi+\sum_{n=2}^AG_n(h,\phi)+G_{>A}\\
&+|\partial_x|\Big[H_0h \cdot H_0\Big(|\partial_x|\phi+\sum_{n=2}^AG_n(h,\phi)+G_{>A}\Big)\Big]\\
&=\sum_{n=2}^A\widetilde{\mathcal{N}}^1_n(h,\phi)+\widetilde{\mathcal{N}}^1_{>A},
\end{split}
\end{equation*}
where, with $n\in\{3,\ldots,A\}$,
\begin{equation}\label{adi12.6}
\begin{split}
\widetilde{\mathcal{N}}^1_2(h,\phi)&:=-\partial_x\big(H_0h\cdot |\partial_x|\phi \big)+G_2(h,\phi)+|\partial_x|(H_0h\cdot\partial_x\phi),\\
\widetilde{\mathcal{N}}^1_n(h,\phi)&:=G_n(h,\phi)+|\partial_x|\big[H_0h \cdot H_0G_{n-1}(h,\phi)\big],\\
\widetilde{\mathcal{N}}^1_{>A}&:=G_{>A}+|\partial_x|\big[H_0h \cdot H_0\big(G_A(h,\phi)+G_{>A}\big)\big].
\end{split}
\end{equation}
Similarly, we also calculate
\begin{equation*}
\begin{split}
&\partial_t\phi^\ast+h^\ast=h+|\partial_x|\big(H_0h \cdot H_0h\big)/2-h+\sum_{n=2}^AH_n(h,\phi)+H_{>A}\\
&+ H_0\Big[H_0\Big(|\partial_x|\phi+\sum_{n=2}^AG_n(h,\phi)+G_{>A}\Big)\cdot |\partial_x|\phi \Big]+H_0\Big[H_0h\cdot |\partial_x|\Big(-h+\sum_{n=2}^AH_n(h,\phi)+H_{>A}\Big)\Big]\\
&=\sum_{n=2}^A\widetilde{\mathcal{N}}^2_n(h,\phi)+\widetilde{\mathcal{N}}^2_{>A},
\end{split}
\end{equation*}
where, with $n\in\{3,\ldots,A\}$,
\begin{equation}\label{adi12.7}
\begin{split}
\widetilde{\mathcal{N}}^2_2(h,\phi)&:=|\partial_x|\big(H_0h \cdot H_0h\big)/2+H_2(h,\phi)+H_0(\partial_x\phi\cdot|\partial_x|\phi)-H_0(H_0h\cdot |\partial_x|h),\\
\widetilde{\mathcal{N}}^2_n(h,\phi)&:=H_n(h,\phi)+H_0\big(H_0G_{n-1}(h,\phi)\cdot|\partial_x|\phi\big)+H_0\big(H_0h\cdot |\partial_x|H_{n-1}(h,\phi)\big),\\
\widetilde{\mathcal{N}}^2_{>A}&:=H_{>A}+H_0\big[H_0(G_{A}(h,\phi)+G_{>A})\cdot|\partial_x|\phi\big]+H_0\big[H_0h\cdot |\partial_x|(H_A(h,\phi)+H_{>A})\big].
\end{split}
\end{equation}

We show now that the quadratic terms $\widetilde{\mathcal{N}}^1_2(h,\phi)$ and $\widetilde{\mathcal{N}}^2_2(h,\phi)$ vanish. Indeed,
\begin{equation*}
\mathcal{F}(\widetilde{\mathcal{N}}^1_2(h,\phi))(\xi)=\frac{1}{2\pi R}\sum_{\eta,\rho\in\Zb/R,\,\eta+\rho=\xi}n_2^1(\eta,\rho)\widehat{h}(\eta)\widehat{\phi}(\rho)
\end{equation*}
where, recalling that $G_2(h,\phi)=-\partial_x(h\partial_x\phi)-|\partial_x|(h|\partial_x|\phi)$,
\begin{equation*}
\begin{split}
n_2^1(\eta,\rho)&:=(\eta+\rho)\,\mathrm{sgn}(\eta)|\rho|-|\eta+\rho|\,\mathrm{sgn}(\eta)\rho+(\eta+\rho)\rho-|\eta+\rho||\rho|\\
&=|\eta+\rho||\rho|\big[\mathrm{sgn}(\eta+\rho)\mathrm{sgn}(\eta)-\mathrm{sgn}(\rho)\mathrm{sgn}(\eta)+\mathrm{sgn}(\eta+\rho)\mathrm{sgn}(\rho)-1\big]=0.
\end{split}
\end{equation*}
The last identity follows by examining explicitly the cases $\rho>0$ and $\rho<0$. Moreover
\begin{equation*}
|\partial_x|\big(H_0h \cdot H_0h\big)/2-H_0(H_0h\cdot |\partial_x|h)=H_0\partial_x\big(H_0h \cdot H_0h\big)/2-H_0(H_0h\cdot |\partial_x|h)=0
\end{equation*}
and, recalling that $H_2(h,\phi)=-(\partial_x\phi)^2/2+(|\partial_x|\phi)^2/2$,
\begin{equation*}
\begin{split}
&\mathcal{F}[H_2(h,\phi)+H_0(\partial_x\phi\cdot|\partial_x|\phi)](\xi)=\frac{1}{2\pi R}\sum_{\eta,\rho\in\Zb/R,\,\eta+\rho=\xi}n_2^2(\eta,\rho)\widehat{\phi}(\eta)\widehat{\phi}(\rho),\\
&n_2^2(\eta,\rho)=\frac{1}{2}|\eta||\rho|+\frac{1}{2}\eta\rho-\frac{1}{2}\mathrm{sgn}(\eta+\rho)\eta|\rho|-\frac{1}{2}\mathrm{sgn}(\eta+\rho)|\eta|\rho\\
&\qquad\quad\,\,=(1/2)|\eta||\rho|\big[1+\mathrm{sgn}(\eta)\mathrm{sgn}(\rho)-\mathrm{sgn}(\eta+\rho)\mathrm{sgn}(\eta)-\mathrm{sgn}(\eta+\rho)\mathrm{sgn}(\rho)]=0.
\end{split}
\end{equation*}
Therefore $\widetilde{\mathcal{N}}^1_2(h,\phi)=0$ and $\widetilde{\mathcal{N}}^2_2(h,\phi)=0$. To summarize, we proved the following:

\begin{lem}\label{QuadNorm}
With $v$ defined as in \eqref{defv} we have
\begin{equation}\label{adi12.9}
(\partial_t+i|\partial_x|^{1/2})v=\sum_{n=3}^A\big[\widetilde{\mathcal{N}}^1_n(h,\phi)+i|\partial_x|^{1/2}\widetilde{\mathcal{N}}^2_n(h,\phi)\big]+\big[\widetilde{\mathcal{N}}^1_{>A}+i|\partial_x|^{1/2}\widetilde{\mathcal{N}}^2_{>A}\big],
\end{equation}
where the nonlinear terms $\widetilde{\mathcal{N}}^1_n(h,\phi), \widetilde{\mathcal{N}}^2_n(h,\phi), \widetilde{\mathcal{N}}^1_{>A},\widetilde{\mathcal{N}}^2_{>A}$ are defined in \eqref{adi12.6}--\eqref{adi12.7}.
\end{lem}

\subsection{The cubic nonlinearities} The cubic terms $\widetilde{\mathcal{N}}^1_3(h,\phi)$ and $|\partial_x|^{1/2}\widetilde{\mathcal{N}}^2_3(h,\phi)$ and be calculated explicitly in terms of the variable $u$. Using \eqref{adi12.6} and \eqref{adi6.10} we write
\begin{equation*}
\widetilde{\mathcal{N}}^1_3(h,\phi)=G_3(h,\phi)+|\partial_x|\big[H_0h \cdot H_0G_2(h,\phi)\big]=-\partial_x\widetilde{G}_3(h,\phi)-|\partial_x|\big[H_0h \cdot |\partial_x|R_0H_0\phi\big].
\end{equation*}
We examine the expansion \eqref{adi6} and identify the cubic term (using also \eqref{na23}),
\begin{equation*}
\widetilde{G}_3(h,\phi)=R_1\phi+R_0^2H_0\phi.
\end{equation*}
The identity (8.47) in \cite{DIP} shows that
\begin{equation}\label{adi13.1}
\begin{split}
\widehat{R_1f}(\xi)&=\frac{1}{(2\pi R)^2}\sum_{\eta_1,\eta_2,\sigma\in\Zb/R,\,\eta_1+\eta_2+\sigma=\xi}\frac{i}{2}\rho\big[|\sigma|+|\eta_1+\eta_2+\sigma|-|\eta_1+\sigma|-|\eta_2+\sigma|\big]\\
&\qquad\qquad\qquad\times\widehat{f}(\sigma)\widehat{h}(\eta_1)\widehat{h}(\eta_2).
\end{split}
\end{equation}
Using also \eqref{adi6.9} and symmetrizing in the variables $\eta_1,\eta_2$ (corresponding to the two copies of $h$) we have
\begin{equation}\label{adi13.2}
\begin{split}
&\mathcal{F}(G_3(h,\phi))(\xi)=\frac{1}{(2\pi R)^2}\sum_{\eta_1,\eta_2,\sigma\in\Zb/R,\,\eta_1+\eta_2+\sigma=\xi}q_3(\eta_1,\eta_2,\sigma)\widehat{\phi}(\sigma)\widehat{h}(\eta_1)\widehat{h}(\eta_2),\\
&q_3(\eta_1,\eta_2,\sigma):=\frac{|\sigma||\eta_1+\eta_2+\sigma|}{2}\big[|\eta_1+\sigma|+|\eta_2+\sigma|-|\eta_1+\eta_2+\sigma|-|\sigma|\big].
\end{split}
\end{equation}
Moreover, using \eqref{adi6.9} again,
\begin{equation*}
\begin{split}
\mathcal{F}&\big\{-|\partial_x|\big[H_0h \cdot |\partial_x|R_0H_0\phi\big]\big\}(\xi)=\frac{1}{(2\pi R)^2}\sum_{\eta_1,\eta_2,\sigma\in\Zb/R,\,\eta_1+\eta_2+\sigma=\xi}\widehat{\phi}(\sigma)\widehat{h}(\eta_1)\widehat{h}(\eta_2)\\
&\times\big\{|\eta_1+\eta_2+\sigma|\mathrm{sgn}(\eta_1)|\eta_2+\sigma||\sigma|\big[\mathrm{sgn}(\eta_2+\sigma)-\mathrm{sgn}(\sigma)\big]\big\}.
\end{split}
\end{equation*}
We symmetrize this formula in $\eta_1$ and $\eta_2$. Using also \eqref{adi13.2} we derive the identity
\begin{equation}\label{adi13.3}
\begin{split}
&\mathcal{F}[\widetilde{\mathcal{N}}^1_3(h,\phi)](\xi)=\frac{1}{(2\pi R)^2}\sum_{\eta,\rho,\sigma\in\Zb/R,\,\eta+\rho+\sigma=\xi}\widehat{h}(\eta)\widehat{h}(\rho)\widehat{|\partial_x|^{1/2}\phi}(\sigma)a_3^1(\eta,\rho,\sigma),\\
&a_3^1(\eta,\rho,\sigma):=\frac{|\xi||\sigma|^{1/2}}{2}\big\{|\sigma+\eta|+|\sigma+\rho|-|\xi|-|\sigma|-|\sigma+\eta|\mathrm{sgn}(\sigma)\mathrm{sgn}(\rho)\\
&\qquad\qquad\,\,-|\sigma+\rho|\mathrm{sgn}(\sigma)\mathrm{sgn}(\eta)+(\sigma+\eta)\mathrm{sgn}(\rho)+(\sigma+\rho)\mathrm{sgn}(\eta)\big\}.
\end{split}
\end{equation}

Similarly, using equation \eqref{adi12.7} we have
\begin{equation*}
\widetilde{\mathcal{N}}^2_3(h,\phi)=H_3(h,\phi)+H_0\big(H_0G_2(h,\phi)\cdot|\partial_x|\phi\big)+H_0\big(H_0h\cdot |\partial_x|H_2(h,\phi)\big).
\end{equation*}
Recalling the definitions \eqref{adi12.2} and the equations \eqref{ww0repeat}, we have
\begin{equation*}
H_3(h,\phi)=|\partial_x|\phi\cdot\big[G_2(h,\phi)+\partial_xh\cdot\partial_x\phi\big].
\end{equation*}
Using also the formulas for $G_2(h,\phi)$ and $H_2(h,\phi)$ in \eqref{adi7} and \eqref{adi12.3}, we have
\begin{equation*}
\begin{split}
\mathcal{F}[\widetilde{\mathcal{N}}^2_3&(h,\phi)](\xi)=\frac{1}{(2\pi R)^2}\sum_{\eta,\rho,\sigma\in\Zb/R,\,\eta+\rho+\sigma=\xi}\widehat{\phi}(\eta)\widehat{\phi}(\rho)\widehat{h}(\sigma)\big\{|\eta|\big[\rho(\sigma+\rho)-|\rho||\sigma+\rho|-\sigma\rho\big]\\
&-\mathrm{sgn}(\xi)|\eta|\big[\rho|\sigma+\rho|-|\rho|(\sigma+\rho)\big]-(1/2)\mathrm{sgn}(\xi)\mathrm{sgn}(\sigma)|\eta+\rho|(\eta\rho+|\eta||\rho|)\big\}.
\end{split}
\end{equation*}
We symmetrize in $\eta$ and $\rho$ to derive our second main formula
\begin{equation}\label{adi13.5}
\begin{split}
&\mathcal{F}[|\partial_x|^{1/2}\widetilde{\mathcal{N}}^2_3(h,\phi)](\xi)=\frac{1}{(2\pi R)^2}\sum_{\eta,\rho,\sigma\in\Zb/R,\,\eta+\rho+\sigma=\xi}\widehat{|\partial_x|^{1/2}\phi}(\eta)\widehat{||\partial_x|^{1/2}\phi}(\rho)\widehat{h}(\sigma)a_3^2(\eta,\rho,\sigma),\\
&a_3^2(\eta,\rho,\sigma):=\frac{|\xi|^{1/2}|\eta|^{1/2}|\rho|^{1/2}}{2}\big\{|\eta|+|\rho|+(\xi+\sigma)\mathrm{sgn}(\xi)-|\sigma+\rho|\big(1+\mathrm{sgn}(\xi)\mathrm{sgn}(\rho)\big)\\
&\qquad\qquad-|\sigma+\eta|\big(1+\mathrm{sgn}(\xi)\mathrm{sgn}(\eta)\big)-|\eta+\rho|\mathrm{sgn}(\xi)\mathrm{sgn}(\sigma)\big(1+\mathrm{sgn}(\eta)\mathrm{sgn}(\rho)\big)\big\}.
\end{split}
\end{equation}
Since
\begin{equation*}
h=\frac{u+\overline{u}}{2},\qquad |\partial_x|^{1/2}\phi=\frac{u-\overline{u}}{2i},
\end{equation*}
we can use the identities \eqref{adi13.3}--\eqref{adi13.5} to calculate
\begin{equation}\label{adi13.8}
\begin{split}
\widetilde{\mathcal{N}}^1_3(h,\phi)+i|\partial_x|^{1/2}\widetilde{\mathcal{N}}^2_3(h,\phi)=\mathcal{N}_{3,+++}+\mathcal{N}_{3,++-}+\mathcal{N}_{3,+--}+\mathcal{N}_{3,---},\\
\mathcal{F}(\mathcal{N}_{3,+++})(\xi):=\frac{1}{(2\pi R)^2}\sum_{\xi_1,\xi_2,\xi_3\in\Zb/R,\,\xi_1+\xi_2+\xi_3=\xi}p_{+++}(\xi_1,\xi_2,\xi_3)\widehat{u}(\xi_1)\widehat{u}(\xi_2)\widehat{u}(\xi_3),\\
\mathcal{F}(\mathcal{N}_{3,++-})(\xi):=\frac{1}{(2\pi R)^2}\sum_{\xi_1,\xi_2,\xi_3\in\Zb/R,\,\xi_1+\xi_2+\xi_3=\xi}p_{++-}(\xi_1,\xi_2,\xi_3)\widehat{u}(\xi_1)\widehat{u}(\xi_2)\widehat{\overline{u}}(\xi_3),\\
\mathcal{F}(\mathcal{N}_{3,+--})(\xi):=\frac{1}{(2\pi R)^2}\sum_{\xi_1,\xi_2,\xi_3\in\Zb/R,\,\xi_1+\xi_2+\xi_3=\xi}p_{+--}(\xi_1,\xi_2,\xi_3)\widehat{u}(\xi_1)\widehat{\overline{u}}(\xi_2)\widehat{\overline{u}}(\xi_3),\\
\mathcal{F}(\mathcal{N}_{3,---})(\xi):=\frac{1}{(2\pi R)^2}\sum_{\xi_1,\xi_2,\xi_3\in\Zb/R,\,\xi_1+\xi_2+\xi_3=\xi}p_{---}(\xi_1,\xi_2,\xi_3)\widehat{\overline{u}}(\xi_1)\widehat{\overline{u}}(\xi_2)\widehat{\overline{u}}(\xi_3),
\end{split}
\end{equation}
where the symbols $p_{\iota_1\iota_2\iota_3}$ are given by
\begin{equation}\label{adi13.9}
\begin{split}
p_{+++}(\xi_1,\xi_2,\xi_3)&:=\frac{1}{8i}\big[a_3^1(\xi_1,\xi_2,\xi_3)+a_3^2(\xi_1,\xi_2,\xi_3)\big],\\
p_{++-}(\xi_1,\xi_2,\xi_3)&:=\frac{1}{8i}\big[-a_3^1(\xi_1,\xi_2,\xi_3)+a_3^1(\xi_3,\xi_1,\xi_2)+a_3^1(\xi_3,\xi_2,\xi_1)\\
&\qquad\quad\,+a_3^2(\xi_1,\xi_2,\xi_3)-a_3^2(\xi_3,\xi_1,\xi_2)-a_3^2(\xi_3,\xi_2,\xi_1)\big],\\
p_{+--}(\xi_1,\xi_2,\xi_3)&:=\frac{1}{8i}\big[-a_3^1(\xi_1,\xi_2,\xi_3)-a_3^1(\xi_1,\xi_3,\xi_2)+a_3^1(\xi_2,\xi_3,\xi_1)\\
&\qquad\quad\,-a_3^2(\xi_1,\xi_2,\xi_3)-a_3^2(\xi_1,\xi_3,\xi_2)+a_3^2(\xi_2,\xi_3,\xi_1)\big],\\
p_{---}(\xi_1,\xi_2,\xi_3)&:=\frac{1}{8i}\big[-a_3^1(\xi_1,\xi_2,\xi_3)+a_3^2(\xi_1,\xi_2,\xi_3)\big].
\end{split}
\end{equation}
The identities \eqref{adi12.9} and \eqref{adi13.8} show that
\begin{equation}\label{adi32}
\begin{split}
&(\partial_t+i|\partial_x|^{1/2})v=\mathcal{IN}_3+\mathcal{IN}_{\geq 4},\\
&\mathcal{IN}_3:=\mathcal{N}_{3,+++}+\mathcal{N}_{3,++-}+\mathcal{N}_{3,+--}+\mathcal{N}_{3,---}\\
&\mathcal{IN}_{\geq 4}:=\sum_{n=4}^A\big[\widetilde{\mathcal{N}}^1_n(h,\phi)+i|\partial_x|^{1/2}\widetilde{\mathcal{N}}^2_n(h,\phi)\big]+\big[\widetilde{\mathcal{N}}^1_{>A}+i|\partial_x|^{1/2}\widetilde{\mathcal{N}}^2_{>A}\big].
\end{split}
\end{equation}

We can calculate explicitly the symbols $a_3^1$ and $a_3^2$ using the formulas \eqref{adi13.3}--\eqref{adi13.5}. Assuming that $\xi=\eta+\rho+\sigma\geq 0$, we calculate 
\begin{equation}\label{adi18.1}
a_3^1(\eta,\rho,\sigma)=0\,\,\text{ and }\,\,a_3^2(\eta,\rho,\sigma)=-|\xi|^{3/2}|\eta|^{1/2}|\rho|^{1/2}\qquad\text{ if }\eta,\rho,\sigma> 0,
\end{equation}
\begin{equation}\label{adi18.2}
\begin{split}
a_3^1(\eta,\rho,\sigma)=|\xi||\sigma|^{1/2}\big[|\sigma+\eta|+|\sigma+\rho|-|\sigma|\big]\,\,&\text{ and }\,\,a_3^2(\eta,\rho,\sigma)=3|\xi|^{3/2}|\eta|^{1/2}|\rho|^{1/2}\\
&\text{ if }\eta,\rho> 0\text{ and }\sigma <0,\\
\end{split}
\end{equation}
\begin{equation}\label{adi18.3}
a_3^1(\eta,\rho,\sigma)=a_3^2(\eta,\rho,\sigma)=0\qquad\text{ if }\eta<0\text{ and }\rho,\sigma> 0,
\end{equation}
\begin{equation}\label{adi18.4}
a_3^1(\eta,\rho,\sigma)=-|\xi|^2|\sigma|^{1/2}\,\,\text{ and }\,\,a_3^2(\eta,\rho,\sigma)=0\qquad\text{ if }\eta,\sigma<0\text{ and }\rho> 0,
\end{equation}
\begin{equation}\label{adi18.5}
a_3^1(\eta,\rho,\sigma)=a_3^2(\eta,\rho,\sigma)=0\qquad\text{ if }\eta,\sigma>0\text{ and }\rho<0,
\end{equation}
\begin{equation}\label{adi18.51}
a_3^1(\eta,\rho,\sigma)=-|\xi|^2|\sigma|^{1/2}\,\,\text{ and }\,\,a_3^2(\eta,\rho,\sigma)=0\qquad\text{ if }\eta>0\text{ and }\rho,\sigma <0,
\end{equation}
\begin{equation}\label{adi18.52}
a_3^1(\eta,\rho,\sigma)=0\,\,\text{ and }\,\,a_3^2(\eta,\rho,\sigma)=|\xi|^{3/2}|\eta|^{1/2}|\rho|^{1/2}\qquad\text{ if }\eta,\rho<0\text{ and }\sigma >0,
\end{equation}
\begin{equation}\label{adi18.8}
a_3^1(\eta,\rho,\sigma)=a_3^2(\eta,\rho,\sigma)=0\qquad\text{ if }\eta=0\text{ or }\rho=0\text{ or }\sigma=0.
\end{equation}

\subsection{The second normal form} The symbols $a_3^j$ in \eqref{adi18.1}--\eqref{adi18.8} are homogeneous of degree $5/2$ and satisfy good $S^\infty$ bounds of the form 
\begin{equation}\label{adi35.4}
\big\|\varphi_{k}(\xi_1+\xi_2+\xi_3)\varphi_{k_1}(\xi_1)\varphi_{k_2}(\xi_2)\varphi_{k_3}(\xi_3)a_3^j(\xi_1,\xi_2,\xi_3)\big\|_{S^\infty}\lesssim 2^{k}2^{\widetilde{k}_2/2}2^{\widetilde{k}_3},
\end{equation}
for any $k,k_1,k_2,k_3\in\overline{\Zb}$ and $j\in\{1,2\}$, where $\widetilde{k}_1\leq\widetilde{k}_2\leq\widetilde{k}_3$ denotes the nondecreasing rearangement of $k_1,k_2,k_3$. This can be easily verified using the explicit formulas. The main issue, however, is that these symbols do not satisfy the $C^{1/2}$ bounds \eqref{symbolcubic}, which are important in some of the combinatorial arguments involving cancellation. To correct this problem we use a second normal form to eliminate the cubic terms $\mathcal{N}_{3,+++}$, $\mathcal{N}_{3,+--}$, and $\mathcal{N}_{3,---}$, as well as part of the cubic term $\mathcal{N}_{3,++-}$. More precisely:

\begin{lem}\label{lemCubNor} (i) We define
\begin{equation}\label{adi30}
\begin{split}
&w:=v+A_{3,+++}(u,u,u)+A_{3,++-}(u,u,\overline{u})+A_{3,+--}(u,\overline{u},\overline{u})+A_{3,---}(\overline{u},\overline{u},\overline{u}),\\
&\mathcal{F}(A_{3,\iota_1\iota_2\iota_3}(f,g,h))(\xi):=\frac{1}{(2\pi R)^2}\sum_{\xi_1,\xi_2,\xi_3\in\Zb/R,\,\xi_1+\xi_2+\xi_3=\xi}q_{\iota_1\iota_2\iota_3}(\xi_1,\xi_2,\xi_3)\widehat{f}(\xi_1)\widehat{g}(\xi_2)\widehat{h}(\xi_3),
\end{split}
\end{equation}
where the symbols $q_{\iota_1\iota_2\iota_3}$ are defined for $(\iota_1\iota_2\iota_3)\in\{(+++), (++-), (+--), (---)\}$ by
\begin{equation}\label{adi31}
\begin{split}
&q_{\iota_1\iota_2\iota_3}(\xi_1,\xi_2,\xi_3):=\frac{ip_{3,\iota_1\iota_2\iota_3}(\xi_1,\xi_2,\xi_3)}{|\xi|^{1/2}-\iota_1|\xi_1|^{1/2}-\iota_2|\xi_2|^{1/2}-\iota_3|\xi_3|^{1/2}}\qquad\text{ if }\,\,(\iota_1\iota_2\iota_3)\neq (++-),\\
&q_{++-}(\xi_1,\xi_2,\xi_3):=\frac{ip_{3,++-}(\xi_1,\xi_2,\xi_3)\big[\mathbf{1}_{\mathcal{R}}(\xi_1,\xi_2,\xi_3)+\mathbf{1}_{{}^c\mathcal{R}}(\xi_1,\xi_2,\xi_3)\Upsilon(\xi_1,\xi_2,\xi_3)\big]}{|\xi|^{1/2}-|\xi_1|^{1/2}-|\xi_2|^{1/2}+|\xi_3|^{1/2}},\\
&\mathcal{R}:=\{(x_1,x_2,x_3)\in\mathbb{R}^3:\,x_1x_2x_3(x_1+x_2+x_3)\geq 0\},\\
&\Upsilon(\xi_1,\xi_2,\xi_3):=\varphi\Big(\frac{2^{10}(\xi_1^2+\xi_2^2+\xi_3^2+\xi^2)\min(\xi_1^2,\xi_2^2,\xi_3^2,\xi^2)}{\xi_1\xi_2\xi_3\xi}\Big),
\end{split}
\end{equation}
where $\xi=\xi_1+\xi_2+\xi_3$. Then
\begin{equation}\label{adi31.5}
\begin{split}
&(\partial_t+i|\partial_x|^{1/2})w=\mathcal{W}_3+\mathcal{W}_{\geq 4},\\
&\widehat{\mathcal{W}_3}(\xi):=\frac{1}{(2\pi R)^2}\sum_{\xi_1,\xi_2,\xi_3\in\Zb/R,\,\xi_1+\xi_2+\xi_3=\xi}[p_{3,++-}\mathbf{1}_{{}^c\mathcal{R}}(1-\Upsilon)](\xi_1,\xi_2,\xi_3)\widehat{u}(\xi_1)\widehat{u}(\xi_2)\widehat{\overline{u}}(\xi_3),\\
&\mathcal{W}_{\geq 4}:=\mathcal{IN}_{\geq 4}+\sum_{(\iota_1\iota_2\iota_3)\in\{(+++), (++-), (+--), (---)\}}\big\{A_{3,\iota_1\iota_2\iota_3}((\partial_tu+i|\partial_x|^{1/2}u)^{\iota_1},u^{\iota_2},u^{\iota_3})\\
&\qquad\,\,+A_{3,\iota_1\iota_2\iota_3}(u^{\iota_1},(\partial_tu+i|\partial_x|^{1/2}u)^{\iota_2},u^{\iota_3})+A_{3,\iota_1\iota_2\iota_3}(u^{\iota_1},u^{\iota_2},(\partial_tu+i|\partial_x|^{1/2}u)^{\iota_3})\big\},
\end{split}
\end{equation}
where $\mathcal{IN}_{\geq 4}$ is defined in \eqref{adi32} and $f^+=f$, $f^-=\overline{f}$. 

(ii) Moreover the symbols $q_{\iota_1\iota_2\iota_3}$ are homogeneous of degree $2$ and satisfy the $S^\infty$ bounds
\begin{equation}\label{adi31.6}
\begin{split}
&q_{\iota_1\iota_2\iota_3}(\lambda\xi_1,\lambda\xi_2,\lambda\xi_3)=\lambda^2q_{\iota_1\iota_2\iota_3}(\xi_1,\xi_2,\xi_3)\qquad\text{ for any }\lambda>0,\,\xi_1,\xi_2,\xi_3\in\mathbb{R},\\
&\big\|q_{\iota_1\iota_2\iota_3}(\xi_1,\xi_2,\xi_3)\varphi_k(\xi_1+\xi_2+\xi_3)\varphi_{k_1}(\xi_1)\varphi_{k_2}(\xi_2)\varphi_{k_3}(\xi_3)\big\|_{S^\infty}\lesssim 2^{k/2}2^{3\max(k_1,k_2,k_3)/2}
\end{split}
\end{equation}
for any $k,k_1,k_2,k_3\in\overline{\Zb}$.
\end{lem}

\begin{proof} (i) Using the definitions \eqref{adi30} we calculate
\begin{equation*}
\begin{split}
(\partial_t&+i|\partial_x|^{1/2})w=(\partial_t+i|\partial_x|^{1/2})v+\sum_{(\iota_1\iota_2\iota_3)}\big\{i|\partial_x|^{1/2}A_{3,\iota_1\iota_2\iota_3}(u^{\iota_1},u^{\iota_2},u^{\iota_3})\\
&+A_{3,\iota_1\iota_2\iota_3}(\partial_tu^{\iota_1},u^{\iota_2},u^{\iota_3})+A_{3,\iota_1\iota_2\iota_3}(u^{\iota_1},\partial_tu^{\iota_2},u^{\iota_3})+A_{3,\iota_1\iota_2\iota_3}(u^{\iota_1},u^{\iota_2},\partial_tu^{\iota_3})\big\},
\end{split}
\end{equation*}
where the sum is taken over $(\iota_1\iota_2\iota_3)\in\{(+++), (++-), (+--), (---)\}$. We use the first identity in \eqref{adi32} and notice that
\begin{equation*}
\begin{split}
\mathcal{N}_{3,\iota_1\iota_2\iota_3}&+i|\partial_x|^{1/2}A_{3,\iota_1\iota_2\iota_3}(u^{\iota_1},u^{\iota_2},u^{\iota_3})-i\iota_1A_{3,\iota_1\iota_2\iota_3}(\{|\partial_x|^{1/2}u^{\iota_1},u^{\iota_2},u^{\iota_3})\\
&-i\iota_2A_{3,\iota_1\iota_2\iota_3}(u^{\iota_1},|\partial_x|^{1/2}u^{\iota_2},u^{\iota_3})-i\iota_3A_{3,\iota_1\iota_2\iota_3}(u^{\iota_1},u^{\iota_2},|\partial_x|^{1/2}u^{\iota_3})=0
\end{split}
\end{equation*}
for any $(\iota_1\iota_2\iota_3)\in\{(+++), (+--), (---)\}$, due to the definitions \eqref{adi30}--\eqref{adi31}. The desired identity \eqref{adi31.5} follows.

(ii) The homogeneity property in \eqref{adi31.6} follows from the definitions \eqref{adi31}, since the symbols $p_{\iota_1\iota_2\iota_3}$ are homogeneous of degree $5/2$ (see the formulas \eqref{adi13.9} and \eqref{adi13.3}--\eqref{adi13.5}). 

To prove the $S^\infty$ bounds in the second line we notice the elementary inequality
\begin{equation}\label{adi35}
|x_1|^{1/2}+|x_2|^{1/2}+|x_3|^{1/2}-|x_1+x_2+x_3|^{1/2}\geq |x_2|^{1/2}/2
\end{equation}
for any $x_1,x_2,x_3\in\mathbb{R}$ with $|x_1|\leq |x_2|\leq |x_3|$. In fact, using  Lemma \ref{touse} (iii) it is easy to see that
\begin{equation}\label{adi35.3}
\Big\|\frac{\varphi_{k_0}(\xi_1+\xi_2+\xi_3)\varphi_{k_1}(\xi_1)\varphi_{k_2}(\xi_2)\varphi_{k_3}(\xi_3)}{|\xi_1+\xi_2+\xi_3|^{1/2}-\iota_1|\xi_1|^{1/2}-\iota_2|\xi_2|^{1/2}-\iota_3|\xi_3|^{1/2}}\Big\|_{S^\infty}\lesssim 2^{-\widetilde{k}_1/2},
\end{equation}
for any $k_0,k_1,k_2,k_3\in\overline{\Zb}$ and triplets $(\iota_1\iota_2\iota_3)\in\{(+++),\,(+--),\,(-+-),\,(--+), (---)\}$, where $\widetilde{k}_0\leq\widetilde{k}_1\leq\widetilde{k}_2\leq\widetilde{k}_3$ is the nondecreasing rearrangement of $k_0,k_1,k_2,k_3$. The desired bounds in \eqref{adi31.6} follow using also \eqref{adi35.4} and \eqref{al8}.

In the remaining case $(\iota_1\iota_2\iota_3)=(++-)$ we need to be more careful because we have to exploit certain cancellations in order to estimate the contribution of the Benjamin-Feir resonances. The formulas \eqref{adi18.1}--\eqref{adi18.8} show that if $\xi=\eta+\rho+\sigma\geq 0$ then
\begin{equation}\label{adi18.55}
\begin{split}
&(a_3^1-a_3^2)(\eta,\rho,\sigma)=|\xi|^{3/2}|\eta|^{1/2}|\rho|^{1/2}\quad\text{ if }\eta,\rho,\sigma> 0,\\
&(a_3^1-a_3^2)(\eta,\rho,\sigma)=|\xi||\sigma|^{1/2}\big[|\sigma+\eta|+|\sigma+\rho|-|\sigma|\big]-3|\xi|^{3/2}|\eta|^{1/2}|\rho|^{1/2}\quad\text{ if }\eta,\rho> 0\text{ and }\sigma <0,\\
&(a_3^1-a_3^2)(\eta,\rho,\sigma)=0\quad\text{ if }\eta\cdot\rho<0\text{ and }\sigma>0,\\
&(a_3^1-a_3^2)(\eta,\rho,\sigma)=-|\xi|^2|\sigma|^{1/2}\quad\text{ if }\eta\cdot\rho<0\text{ and }\sigma<0,\\
&(a_3^1-a_3^2)(\eta,\rho,\sigma)=-|\xi|^{3/2}|\eta|^{1/2}|\rho|^{1/2}\quad\text{ if }\eta,\rho<0\text{ and }\sigma >0.
\end{split}
\end{equation}
Using now the formulas \eqref{adi13.9}, if $\xi_1+\xi_2+\xi_3=\xi\geq 0$ we calculate
\begin{equation*}
\begin{split}
p_{++-}(\xi_1,\xi_2,\xi_3)&=\frac{|\xi|^{3/2}\big(|\xi_1|^{1/2}|\xi_3|^{1/2}+|\xi_2|^{1/2}|\xi_3|^{1/2}-|\xi_1|^{1/2}|\xi_2|^{1/2}\big)}{8i}\qquad \text{ if }\xi_1,\xi_2,\xi_3>0,\\
p_{++-}(\xi_1,\xi_2,\xi_3)&=\frac{|\xi|^{3/2}\big(|\xi|^{1/2}|\xi_3|^{1/2}-|\xi_2|^{1/2}|\xi|^{1/2}-|\xi_2|^{1/2}|\xi_3|^{1/2}\big)}{8i}\qquad \text{ if }\xi_1>0,\,\xi_2,\xi_3<0,\\
p_{++-}(\xi_1,\xi_2,\xi_3)&=\frac{|\xi|^{3/2}\big(|\xi|^{1/2}|\xi_3|^{1/2}-|\xi_1|^{1/2}|\xi|^{1/2}-|\xi_1|^{1/2}|\xi_3|^{1/2}\big)}{8i}\qquad \text{ if }\xi_2>0,\,\xi_1,\xi_3<0,\\
p_{++-}(\xi_1,\xi_2,\xi_3)&=\frac{|\xi|^{3/2}\big(|\xi_1|^{1/2}|\xi_2|^{1/2}-|\xi_1|^{1/2}|\xi|^{1/2}-|\xi_2|^{1/2}|\xi|^{1/2}\big)}{8i}\qquad \text{ if }\xi_3>0,\,\xi_1,\xi_2<0,
\end{split}
\end{equation*}
and
\begin{equation}\label{adi18.57}
\begin{split}
p_{++-}(\xi_1,\xi_2,\xi_3)&=\frac{(a^1_3-a^2_3)(\xi_1,\xi_2,\xi_3)}{-8i}\qquad \text{ if }\xi_1,\xi_2>0,\,\xi_3<0,\\
p_{++-}(\xi_1,\xi_2,\xi_3)&=\frac{(a^1_3-a^2_3)(\xi_3,\xi_1,\xi_2)}{8i}\qquad \text{ if }\xi_1,\xi_3>0,\,\xi_2<0,\\
p_{++-}(\xi_1,\xi_2,\xi_3)&=\frac{(a^1_3-a^2_3)(\xi_3,\xi_2,\xi_1)}{8i}\qquad \text{ if }\xi_2,\xi_3>0,\,\xi_1<0.
\end{split}
\end{equation}
Therefore, after algebraic simplifications,
\begin{equation}\label{adi18.6}
\begin{split}
\frac{ip_{++-}(\xi_1,\xi_2,\xi_3)}{|\xi|^{1/2}-|\xi_1|^{1/2}-|\xi_2|^{1/2}+|\xi_3|^{1/2}}&=\frac{|\xi|^{3/2}\big(|\xi|^{1/2}+|\xi_1|^{1/2}+|\xi_2|^{1/2}-|\xi_3|^{1/2}\big)}{16}\,\,\,\text{ if }\xi_1,\xi_2,\xi_3>0,\\
\frac{ip_{++-}(\xi_1,\xi_2,\xi_3)}{|\xi|^{1/2}-|\xi_1|^{1/2}-|\xi_2|^{1/2}+|\xi_3|^{1/2}}&=\frac{|\xi|^{3/2}\big(|\xi|^{1/2}+|\xi_1|^{1/2}-|\xi_2|^{1/2}+|\xi_3|^{1/2}\big)}{16}\,\,\,\text{ if }\xi_1>0,\,\xi_2,\xi_3<0,\\
\frac{ip_{++-}(\xi_1,\xi_2,\xi_3)}{|\xi|^{1/2}-|\xi_1|^{1/2}-|\xi_2|^{1/2}+|\xi_3|^{1/2}}&=\frac{|\xi|^{3/2}(|\xi|^{1/2}-|\xi_1|^{1/2}+|\xi_2|^{1/2}+|\xi_3|^{1/2})}{16}\,\,\, \text{ if }\xi_2>0,\,\xi_1,\xi_3<0,\\
\frac{ip_{++-}(\xi_1,\xi_2,\xi_3)}{|\xi|^{1/2}-|\xi_1|^{1/2}-|\xi_2|^{1/2}+|\xi_3|^{1/2}}&=\frac{|\xi|^{3/2}\big(|\xi|^{1/2}-|\xi_1|^{1/2}-|\xi_2|^{1/2}-|\xi_3|^{1/2}\big)}{16}\,\,\,\text{ if }\xi_3>0,\,\xi_1,\xi_2<0.
\end{split}
\end{equation}
The multiplier $q^1_{++-}:=\mathbf{1}_{\mathcal{R}}\cdot q_{++-}$ is therefore smooth in the region $\mathcal{R}$, and the desired bounds \eqref{adi31.6} follow for the symbol $q^1_{++-}$.

It remains to prove similar bounds for the multiplier
\begin{equation}\label{adi18.7}
q^2_{++-}(\xi_1,\xi_2,\xi_3):=(\mathbf{1}_{{}^c\mathcal{R}}\cdot q_{++-})(\xi_1,\xi_2,\xi_3)=\frac{ip_{3,++-}(\xi_1,\xi_2,\xi_3)\mathbf{1}_{{}^c\mathcal{R}}(\xi_1,\xi_2,\xi_3)\Upsilon(\xi_1,\xi_2,\xi_3)}{|\xi|^{1/2}-|\xi_1|^{1/2}-|\xi_2|^{1/2}+|\xi_3|^{1/2}}.
\end{equation}
The point is that the denominator does not vanish in the support of the function $\Upsilon$. Indeed, we notice that if $a,b,c,d\in\Rb$, $a+b+c+d=0$, and $|d|\leq \min(|a|, |b|, |c|)/20$ then
\begin{equation}\label{adi19}
\begin{split}
|a|^{1/2}+|b|^{1/2}-|c|^{1/2}-|d|^{1/2}&=\frac{|a|+|b|+2\sqrt{|a||b|}-|c|-|d|-2\sqrt{|c||d|}}{|a|^{1/2}+|b|^{1/2}+|c|^{1/2}+|d|^{1/2}}\\
&\gtrsim
\begin{cases}
\min(|a|^{1/2},|b|^{1/2})\quad&\text{ in all cases},\\
\max(|a|^{1/2},|b|^{1/2})\quad&\text{ if }|c|\lesssim \min(|a|,|b|).
\end{cases}
\end{split}
\end{equation}
Since $\varphi$ is supported in $[-2,2]$, the definition of the function $\Upsilon$ in \eqref{adi31} shows that $$q^2_{++-}(\xi_1,\xi_2,\xi_3)\varphi_{k_1}(\xi_1)\varphi_{k_2}(\xi_2)\varphi_{k_3}(\xi_3)\varphi_{k_0}(\xi_1+\xi_2+\xi_3)\equiv 0$$ if $\widetilde{k}_0\geq\widetilde{k}_1-10$, where $\widetilde{k}_0\leq\widetilde{k}_1\leq\widetilde{k}_2\leq\widetilde{k}_3$ denotes the nondecreasing rearangement of $k_0,k_1,k_2,k_3$. On the other hand, if $\widetilde{k}_0\leq\widetilde{k}_1-10$ then it is easy to see that
\begin{equation*}
\Big\|\frac{\varphi_{k_0}(\xi_1+\xi_2+\xi_3)\varphi_{k_1}(\xi_1)\varphi_{k_2}(\xi_2)\varphi_{k_3}(\xi_3)}{|\xi_1+\xi_2+\xi_3|^{1/2}-|\xi_1|^{1/2}-|\xi_2|^{1/2}+|\xi_3|^{1/2}}\Big\|_{S^\infty}\lesssim 2^{-\widetilde{k}_1/2},
\end{equation*}
due to \eqref{adi19} and \eqref{al8.6}. Moreover, using \eqref{adi35.4},
\begin{equation*}
\big\|\varphi_{k_0}(\xi_1+\xi_2+\xi_3)\varphi_{k_1}(\xi_1)\varphi_{k_2}(\xi_2)\varphi_{k_3}(\xi_3)[p_{3,++-}\mathbf{1}_{{}^c\mathcal{R}}\Upsilon](\xi_1,\xi_2,\xi_3)\big\|_{S^\infty}\lesssim 2^{k_0/2}2^{\widetilde{k}_1/2}2^{3\widetilde{k}_3/2}.
\end{equation*}
Therefore, using \eqref{al8},
\begin{equation*}
\big\|\varphi_{k_0}(\xi_1+\xi_2+\xi_3)\varphi_{k_1}(\xi_1)\varphi_{k_2}(\xi_2)\varphi_{k_3}(\xi_3)q^2_{++-}(\xi_1,\xi_2,\xi_3)\big\|_{S^\infty}\lesssim 2^{k_0/2}2^{3\widetilde{k}_3/2},
\end{equation*}
which completes the proof of the bounds \eqref{adi31.6} in the case $(\iota_1\iota_2\iota_3)=(++-)$. 
\end{proof}

\subsection{Proof of Proposition \ref{normalprop}} We start by proving canonical expansions for the nonlinearities $\widetilde{\mathcal{N}}^1_n(h,\phi)$, $\widetilde{\mathcal{N}}^2_n(h,\phi)$, $\widetilde{\mathcal{N}}^1_{>A}$, $\widetilde{\mathcal{N}}^2_{>A}(h,\phi)$ defined in \eqref{adi12.6}--\eqref{adi12.7}, in terms of the original variable $u$.

\begin{lem}\label{level1T}
With $u=h+i|\partial_x|^{1/2}\phi$ as before, for any $n\in[3,A]$ we can write
\begin{equation}\label{adi20.1}
\begin{split}
\mathcal{F}&\big\{\widetilde{\mathcal{N}}^1_n(h,\phi)+i|\partial_x|^{1/2}\widetilde{\mathcal{N}}^2_n(h,\phi)\big\}(\xi)\\
&=\frac{1}{(2\pi R)^{n-1}}\sum_{\iota_j\in\{\pm\}}\sum_{\eta_1,\ldots,\eta_n\in\Zb/R,\,\eta_1+\ldots+\eta_n=\xi}b_{\iota_1\ldots\iota_n}(\eta_1,\ldots,\eta_n)\prod_{j=1}^n\widehat{u^{\iota_j}}(\eta_j),
\end{split}
\end{equation}
where $u^+=u$, $u^-=\overline{u}$, and the symbols $b_{\iota_1\ldots\iota_n}$ are homogeneous of degree $n-1/2$ and satisfy the $S^\infty$ bounds
\begin{equation}\label{adi20.2}
\big\|b_{\iota_1\ldots\iota_n}(\eta_1,\ldots,\eta_n)\varphi_k(\eta_1+\ldots+\eta_n)\varphi_{k_1}(\eta_1)\cdot\ldots\cdot\varphi_{k_n}(\eta_n)\big\|_{S^\infty}\lesssim 2^{k/2}2^{(n-1)\max(k_1,\ldots,k_n)},
\end{equation}
for any $k,k_1,\ldots k_n\in\overline{\Zb}$. Moreover
\begin{equation}\label{adi20.25}
\big\|\widetilde{\mathcal{N}}^1_{>A}(h,\phi)+i|\partial_x|^{1/2}\widetilde{\mathcal{N}}^2_{>A}(h,\phi)\big\|_{H^{N_2-A-1}}\lesssim \eps_\infty^AB_2.
\end{equation}
\end{lem}

\begin{proof} The bounds follow easily from the formulas \eqref{adi12.6}--\eqref{adi12.7} and the estimates on the functions $G_n,G_{>A},H_n,H_{>A}$ in Lemma \ref{DNexpansion}.  At this stage one can actually prove stronger bounds in terms of derivative loss (for example one could replace the Sobolev space $H^{N_2-A-1}$ with $H^{N_2-3}$ in \eqref{adi20.25}), but such an improvement would not change the following steps.
\end{proof} 

To prove the main bounds in Proposition \ref{normalprop} we need now to understand compositions of operators defined by suitable $S^\infty$ symbols. We start by defining a suitable class of operators:

\begin{df}\label{DefO} For any $n\in\{1,2,\ldots\}$ and $n'\in[0,\infty)$ we define $\mathcal{O}_{n,n'}$ as the normed space of $n$-linear operators $T:(H^{n'+2}({\Tb_R}))^n\to L^2(\Tb_R)$ of the form
\begin{equation}\label{adi20.3}
\mathcal{F}\{T(f_1,\ldots,f_n)\}(\xi)=\frac{1}{(2\pi R)^{n-1}}\sum_{\eta_1,\ldots,\eta_n\in\Zb/R,\,\eta_1+\ldots+\eta_n=\xi}a(\eta_1,\ldots,\eta_n)\widehat{f_1}(\eta_1)\cdot\ldots\cdot \widehat{f_n}(\eta_n),
\end{equation}
where the symbol $a$ satisfies the $S^\infty$ bounds
\begin{equation}\label{adi20.4}
\big\|a(\eta_1,\ldots,\eta_n)\varphi_k(\eta_1+\ldots+\eta_n)\varphi_{k_1}(\eta_1)\cdot\ldots\cdot\varphi_{k_n}(\eta_n)\big\|_{S^\infty}\leq B2^{k/2}2^{(n'-1/2)\max(k_1,\ldots,k_n)},
\end{equation}
for any $k,k_1,\ldots,k_n\in\overline{\Zb}$. Here $\|T\|_{\mathcal{O}_{n,n'}}:=B$ is the smallest constant such that the inequality \eqref{adi20.4} holds for any $k,k_1,\ldots,k_n\in\overline{\Zb}$.
\end{df}

The point of this definition is that we can prove the following lemma concerning compositions and mapping properties of $\mathcal{O}_{n,n'}$ operators:

\begin{lem}\label{level2T} (i) Assume $T\in\mathcal{O}_{n,n'}$ and $S_j\in \mathcal{O}_{m_j,m'_j}$, $j\in\{1,\ldots,n\}$, are multi-linear operators. We define the multi-linear operator $X$ (of order $d:=m_1+\ldots+m_n$) by
\begin{equation}\label{adi20.41}
X(f_{11},\ldots f_{1m_1},\ldots f_{n1},\ldots,f_{n\,m_n}):=T(S_1(f_{11},\ldots f_{1m_1}),\ldots, S_n(f_{n1},\ldots,f_{n\,m_n})).
\end{equation}
Then $X\in\mathcal{O}_{d,d'}$, where $d':=n'+m'_1+\ldots+m'_n$, and
\begin{equation}\label{adi20.42}
\|X\|_{\mathcal{O}_{d,d'}}\lesssim \|T\|_{\mathcal{O}_{n,n'}}\|S_1\|_{\mathcal{O}_{m_1,m'_1}}\cdot\ldots\cdot\|S_n\|_{\mathcal{O}_{m_n,m'_n}}.
\end{equation}

(ii) Assume $T\in\mathcal{O}_{n,n'}$ and $f_1,\ldots,f_n\in L^2(\T_R)$ satisfy the bounds 
\begin{equation}\label{adi20.43}
\|f_i\|_{H^{s_2}}\leq A_2,\qquad \|f_i\|_{\dot{W}^{s_\infty,0}}\leq A_\infty,
\end{equation}
for any $i\in\{1,\ldots,n\}$, where $s_2\geq n'$ and $s_\infty\geq 0$. Then
\begin{equation}\label{adi20.44}
\|T(f_1,\ldots,f_n)\|_{H^{s_2-n'}}\lesssim \|T\|_{\mathcal{O}_{n,n'}}A_2A_\infty^{n-1}.
\end{equation}
Moreover, if $s_\infty\geq n'$ then we also have
\begin{equation}\label{adi20.45}
\|T(f_1,\ldots,f_n)\|_{\dot{W}^{s_\infty-n',0}}\lesssim \|T\|_{\mathcal{O}_{n,n'}}A_\infty^{n}.
\end{equation}
\end{lem}

\begin{proof} (i) Let $a,b_1, \ldots,b_n$ denote the symbols associated to the operators $T,S_1,\ldots,S_n$ respectively. Then $X$ can be written in the form
\begin{equation}\label{adi20.5}
\begin{split}
&\mathcal{F}\{X(f_{11},\ldots,f_{n\,m_n})\}(\xi)=\frac{1}{(2\pi R)^{d-1}}\sum_{\eta_{11},\ldots,\eta_{n\,m_n}\in\Zb/R,\,\eta_{11}+\ldots+\eta_{n\,m_n}=\xi}p(\eta_{11},\ldots,\eta_{n\,m_n})\\
&\qquad\qquad\times\widehat{f_{11}}(\eta_{11})\cdot\ldots\cdot \widehat{f_{n\,m_n}}(\eta_{n\,m_n}),\\
&p(\eta_{11},\ldots,\eta_{n\,m_n}):=a(\eta_{11}+\ldots+\eta_{1\,m_1},\ldots,\eta_{n1}+\ldots+\eta_{n\,m_n})\prod_{i=1}^nb_i(\eta_{i1},\ldots,\eta_{i\,m_i}).
\end{split}
\end{equation}
By multi-linearity we may assume that  $\|T\|_{\mathcal{O}_{d,d'}}=\|S_1\|_{\mathcal{O}_{m_1,m'_1}}=\ldots=\|S_n\|_{\mathcal{O}_{m_n,m'_n}}=1$. We have to prove that for any $k, l_{11},\ldots,l_{n\,m_n}\in\overline{\Zb}$ we have
\begin{equation}\label{adi20.6}
\big\|p(\eta_{11},\ldots,\eta_{n\,m_n})\varphi_k(\eta_{11}+\ldots+\eta_{n\,m_n})\varphi_{l_{11}}(\eta_{11})\cdot\ldots\cdot\varphi_{l_{n\,m_n}}(\eta_{n\,m_n})\big\|_{S^\infty}\lesssim 2^{k/2}2^{(m'-1/2)L},
\end{equation}
where $L:=\max(l_{11},\ldots,l_{n\,m_n})$. 

For any $k,k_1,\ldots,k_n\in\overline{\Zb}$ let $K_{k;k_1\ldots k_n}$ denote the inverse Fourier transform of the symbol $(\xi_1,\ldots,\xi_n)\to a(\xi_1,\ldots,\xi_n)\varphi_{k_1}(\xi_1)\cdot\ldots\cdot\varphi_{k_n}(\xi_n)\varphi_k(\xi_1+\ldots+\xi_n)$, thus
\begin{equation*}
\begin{split}
&a(\xi_1,\ldots,\xi_n)\varphi_{k_1}(\xi_1)\cdot\ldots\cdot\varphi_{k_n}(\xi_n)\varphi_k(\xi_1+\ldots+\xi_n)\\
&\qquad\qquad\qquad=\int_{\Tb_R^n}K_{k;k_1\ldots k_n}(x_1,\ldots,x_n)e^{-i(x_1\xi_1+\ldots+x_n\xi_n)}\,dx_1\ldots dx_n.
\end{split}
\end{equation*}
Therefore, using \eqref{adi20.5}, the inverse Fourier transform of the multiplier in the left-hand side of \eqref{adi20.6} is equal to
\begin{equation*}
\begin{split}
&G_{k;l_{11}\ldots l_{nm_n}}(y_{11},\ldots,y_{n\,m_n})=\frac{1}{(2\pi R)^d}\sum_{k_1,\ldots,k_n\leq L+2d}\sum_{\eta_{11},\ldots,\eta_{n\,m_n}\in\Zb/R}e^{i(y_{11}\eta_{11}+\ldots+y_{n\,m_n}\eta_{n\,m_n})}\\
&\times\prod_{i=1}^n\big[b_i(\eta_{i1},\ldots,\eta_{i\,m_i})\varphi_{l_{i1}}(\eta_{i1})\cdot\ldots\cdot\varphi_{l_{i\,m_i}}(\eta_{i\,m_i})\varphi_{[k_i-2,k_i+2]}(\eta_{i1}+\ldots+\eta_{i\,m_i})\big]\\
&\times\int_{\Tb_R^n}K_{k;k_1\ldots k_n}(x_1,\ldots,x_n)e^{-i\sum_{i=1}^n x_i(\eta_{i1}+\ldots+\eta_{i\,m_i})}\,dx_1\ldots dx_n.
\end{split}
\end{equation*}
Letting $L^i_{k_i;l_{i1}\ldots l_{i\,m_i}}$ denote the inverse Fourier transform of the multiplier $$(\eta_1,\ldots,\eta_{m_i})\to b_i(\eta_{1},\ldots,\eta_{m_i})\varphi_{l_{i1}}(\eta_{1})\cdot\ldots\cdot\varphi_{l_{i\,m_i}}(\eta_{m_i})\varphi_{[k_i-2,k_i+2]}(\eta_{1}+\ldots+\eta_{m_i}),$$ we therefore have
\begin{equation*}
\begin{split}
&G_{k;l_{11}\ldots l_{nm_n}}(y_{11},\ldots,y_{n\,m_n})=\sum_{k_1,\ldots,k_n\leq L+2d}\int_{\Tb_R^n}K_{k;k_1\ldots k_n}(x_1,\ldots,x_n)\\
&\qquad\qquad\times\prod_{i=1}^nL^i_{k_i;l_{i1}\ldots l_{i\,m_i}}(y_{i1}-x_i,\ldots,y_{i\,m_i}-x_i)\,dx_1\ldots dx_n.
\end{split}
\end{equation*}
In particular
\begin{equation}\label{adi20.8}
\|G_{k;l_{11}\ldots l_{nm_n}}\|_{L^1(\Tb_R^d)}\lesssim \sum_{k_1,\ldots,k_n\leq L+2d}\|K_{k;k_1\ldots k_n}\|_{L^1(\Tb_R^n)}\prod_{i=1}^n\big\|L^i_{k_i;l_{i1}\ldots l_{i\,m_i}}\big\|_{L^1(\Tb_R^{m_i})}.
\end{equation}
The desired bounds \eqref{adi20.6} follow since
\begin{equation*}
\|K_{k;k_1\ldots k_n}\|_{L^1(\Tb_R^n)}\lesssim 2^{k/2}2^{(n'-1/2)\max(k_1,\ldots,k_n)},\qquad \big\|L^i_{k_i;l_{i1}\ldots l_{i\,m_i}}\big\|_{L^1}\lesssim 2^{k_i/2}2^{(m'_i-1/2)L},
\end{equation*}
due to the assumption $\|T\|_{\mathcal{O}_{n,n'}}=\|S_1\|_{\mathcal{O}_{m_1,m'_1}}=\ldots=\|S_n\|_{\mathcal{O}_{m_n,m'_n}}=1$.

(ii) To prove \eqref{adi20.44} we may assume $\|T\|_{\mathcal{O}_{n,n'}}=1$ and define the frequency envelope
\begin{equation}\label{adi20.85}
\rho_k:=\sup_{k'\in\overline{\Zb},\,i\in\{1,\ldots,n\}}2^{-|k-k'|/10}2^{s_2k'_+}\|P_{k'}f_i\|_{L^2},\qquad \text{ for any }k\in\overline{\Zb},
\end{equation}
where $b_+:=\max(b,0)$ for any $b\in\overline{\Zb}$. Then, using \eqref{adi20.43}
\begin{equation}\label{adi20.9}
\sum_{k\in\overline{\Zb}}\rho_k^2\lesssim A_2^2,\qquad \|P_{k'}f_i\|_{L^2}\lesssim\rho_k2^{|k-k'|/10}2^{-s_2k'_+}\text{ for any }k,k'\in\overline{\Zb},\,i\in\{1,\ldots,n\}.
\end{equation}
Using now \eqref{mk6} and the assumption $\|T\|_{\mathcal{O}_{n,n'}}=1$, for any $k\in\overline{\Zb}$ we estimate
\begin{equation*}
\begin{split}
\|P_k(T(f_1,\ldots,f_n))\|_{H^{s_2-n'}}&\lesssim 2^{(s_2-n')k_+}\sum_{k_1,\ldots,k_n\in\Zb}\|P_k(T(P_{k_1}f_1,\ldots,P_{k_n}f_n))\|_{L^2}\\
&\lesssim A_\infty^{n-1}2^{(s_2-n')k_+}2^{k/2}\sum_{k'\geq k-4n,\,i\in\{1,\ldots,n\}}2^{(n'-1/2)k'}\|P_{k'}f_i\|_{L^2},
\end{split}
\end{equation*}
by always estimating the highest frequency in $L^2$ and all the lower frequencies in $L^\infty$. Thus
\begin{equation*}
\begin{split}
\|P_k(T(f_1,\ldots,f_n))\|_{H^{s_2-n'}}&\lesssim A_\infty^{n-1}2^{(s_2-n')k_+}2^{k/2}\sum_{k'\geq k-4n}2^{(n'-1/2)k'}\rho_k2^{(k'-k)/10}2^{-s_2k'_+}\\
&\lesssim A_\infty^{n-1}\rho_k
\end{split}
\end{equation*}
The bounds \eqref{adi20.44} follow due to the first estimate in \eqref{adi20.9}. 

The bounds \eqref{adi20.45} can be proved in a similar way, by defining a suitable frequency envelope analogous to \eqref{adi20.85}, based on frequency-localized $L^\infty$ norms.
\end{proof}

We can now complete the proof of all the claims in Proposition \ref{normalprop}. 

{\bf{Step 1.}} The function $w$ is defined as in \eqref{defv} and \eqref{adi30},
\begin{equation}\label{adi40}
\begin{split}
&w:=u+\mathcal{A}_2(u)+\mathcal{A}_3(u),\\
&\mathcal{A}_2(u):=A_{2,++}(u,u)+A_{2,+-}(u,\overline{u})+A_{2,--}(\overline{u},\overline{u}),\\
&\mathcal{A}_3(u):=A_{3,+++}(u,u,u)+A_{3,++-}(u,u,\overline{u})+A_{3,+--}(u,\overline{u},\overline{u})+A_{3,---}(\overline{u},\overline{u},\overline{u}),
\end{split}
\end{equation}
where the operators $A_{3,\iota_1\iota_2\iota}$ are defined as in Lemma \ref{lemCubNor}, and
 \begin{equation}\label{adi41}
\begin{split}
&A_{2,++}(f,g):=(1/8)|\partial_x|(H_0f\cdot H_0g)+(1/8)|\partial_x|^{1/2}H_0[H_0f\cdot |\partial_x|^{1/2}g+|\partial_x|^{1/2}f\cdot H_0g\big],\\
&A_{2,+-}(f,g):=(1/4)|\partial_x|(H_0f\cdot H_0g)-(1/4)|\partial_x|^{1/2}H_0\big[H_0f\cdot |\partial_x|^{1/2}g-|\partial_x|^{1/2}f\cdot H_0g\big],\\
&A_{2,--}(f,g):=(1/8)|\partial_x|(H_0f\cdot H_0g)-(1/8)|\partial_x|^{1/2}H_0[H_0f\cdot |\partial_x|^{1/2}g+|\partial_x|^{1/2}f\cdot H_0g\big].
\end{split}
\end{equation}
These definitions show easily that $A_{2,\iota_1\iota_2}\in\mathcal{O}_{2,1}$, while the bounds \eqref{adi31.6} show that $A_{3,\iota_1\iota_2\iota_3}\in\mathcal{O}_{3,2}$. This completes the proof of the identities and the bounds \eqref{normal}, \eqref{symbolboundN}, and the identities in the first line of \eqref{nonlin1}. The bounds $\|w\|_{H^{N_2-2}}\lesssim B_2$ and $\|w\|_{\dot{W}^{N_\infty-2,0}}\lesssim \varepsilon_\infty$ in \eqref{nonlin2} follow from \eqref{apriorepeat} and Lemma \ref{level2T} (ii).

{\bf{Step 2.}} We show now that the nonlinearity $\mathcal{W}_{\geq 4}$ defined in \eqref{adi31.5} can be written as
\begin{equation}\label{adi45}
\begin{split}
&\mathcal{W}_{\geq 4}=\sum_{n=4}^A\sum_{\iota_1,\ldots,\iota_n\in\{+,-\}}W_{n,\iota_1\ldots\iota_n}(u^{\iota_1},\ldots, u^{\iota_n})+W_{>A},\\
&W_{n,\iota_1\ldots\iota_n}\in\mathcal{O}_{n,n-1/2},\qquad\|W_{>A}\|_{H^{N_2-A-2}}\lesssim \eps_\infty^AB_2.
\end{split}
\end{equation}

Indeed, the terms in $\mathcal{IN}_{\geq 4}$ are already expressed in the desired form in Lemma \ref{level1T}. To express the remaining terms, we notice that, as a consequence of Lemma \ref{DNexpansion}
\begin{equation}\label{adi47}
\begin{split}
\partial_tu+i|\partial_x|^{1/2}u&=\sum_{n=2}^A[G_n(h,\phi)+i|\partial_x|^{1/2}H_n(h,\phi)]+(G_{>A}+i|\partial_x|^{1/2}H_{>A})\\
&=\sum_{n=2}^A\sum_{\iota_1,\ldots,\iota_n\in\{+,-\}}L_{n,\iota_1\ldots\iota_n}(u^{\iota_1},\ldots,u^{\iota_n})+L_{>A},
\end{split}
\end{equation}
for some operators $L_{n,\iota_1\ldots\iota_n}\in\mathcal{O}_{n,n-1/2}$ defined by homogeneous symbols of degree $n-1/2$ and an error term $L_{>A}$ satisfying the bounds $\|L_{>A}\|_{\dot{H}^{N_2-2,-1/2}}\lesssim \epsilon_\infty^AB_2$ and $\|L_{>A}\|_{\dot{W}^{N_\infty-2,-1/4}}\lesssim \epsilon_\infty^{A+1}$. The desired representation \eqref{adi45} follows using Lemma \ref{level2T}, since $A_{3,\iota_1\iota_2\iota_3}\in\mathcal{O}_{3,2}$ (see \eqref{adi31.6}).

{\bf{Step 3.}} Finally, we prove the conclusions on the nonlinear terms $\mathcal{N}_n(w)$ stated in Proposition \eqref{normalprop}. Notice that, as a consequence of the identities \eqref{adi31.5} and \eqref{adi45},
\begin{equation}\label{adi48}
\partial_tw+i|\partial_x|^{1/2}w=\sum_{n=3}^A\sum_{\iota_1,\ldots,\iota_n\in\{+,-\}}W_{n,\iota_1\ldots\iota_n}(u^{\iota_1},\ldots, u^{\iota_n})+W_{>A},
\end{equation}
for some operators $W_{n,\iota_1\ldots\iota_n}\in\mathcal{O}_{n,n-1/2}$ defined by homogeneous symbols of degree $n-1/2$ and an error term $W_{>A}$ satisfying the bounds $\|W_{>A}\|_{\dot{H}^{N_2-A-2}}\lesssim \epsilon_\infty^AB_2$. We have to show that we can express the nonlinear terms in the right-hand side of \eqref{adi48} in terms of the variable $w$.

Assume $n\geq 3$, $S\in\mathcal{O}_{n,n-1/2}$ is a multi-linear operator, and $\iota_1,\ldots\iota_n\in\{+,-\}$. We use the identity $u=w-\mathcal{A}_{\geq 2}(u)$ in \eqref{adi40}, where $\mathcal{A}_{\geq 2}(u):=\mathcal{A}_{2}(u)+\mathcal{A}_3(u)$, to write
\begin{equation}\label{adi49}
\begin{split}
&S(u^{\iota_1},\ldots,u^{\iota_n})=S(w^{\iota_1},u^{\iota_2},\ldots,u^{\iota_n})-S([\mathcal{A}_{\geq 2}(u)]^{\iota_1},u^{\iota_2},\ldots,u^{\iota_n})\\
&=S(w^{\iota_1},w^{\iota_2},\ldots,u^{\iota_n})-S(w^{\iota_1},[\mathcal{A}_{\geq 2}(u)]^{\iota_2},\ldots,u^{\iota_n})-S([\mathcal{A}_{\geq 2}(u)]^{\iota_1},u^{\iota_2},\ldots,u^{\iota_n})\\
&=\ldots=S(w^{\iota_1},\ldots,w^{\iota_n})-\sum_{a=0}^{n-1}S(w^{\iota_1},\ldots,w^{\iota_a},[\mathcal{A}_{\geq 2}(u)]^{\iota_{a+1}},\ldots,u^{\iota_n})\\
&=S(w^{\iota_1},\ldots,w^{\iota_n})-\sum_{a=0}^{n-1}S(u^{\iota_1}+[\mathcal{A}_{\geq 2}(u)]^{\iota_1},\ldots,u^{\iota_a}+[\mathcal{A}_{\geq 2}(u)]^{\iota_a},[\mathcal{A}_{\geq 2}(u)]^{\iota_{a+1}},\ldots,u^{\iota_n}).
\end{split}
\end{equation}
We can therefore replace $S(u^{\iota_1},\ldots,u^{\iota_n})$ with $S(w^{\iota_1},\ldots,w^{\iota_n})$ at the expense of errors of higher homogeneity; more precisely we can write
\begin{equation}\label{adi49.2}
S(u^{\iota_1},\ldots,u^{\iota_n})=S(w^{\iota_1},\ldots,w^{\iota_n})+\sum_{a=n+1}^{\min(3n,A)}\sum_{\zeta_1,\ldots,\zeta_a\in\{+,-\}}S'_{a,\zeta_1\ldots\zeta_a}(u^{\zeta_1},\ldots,u^{\zeta_a})+S'_{>A}
\end{equation}
for some multi-linear operators $S'_{a,\zeta_1\ldots\zeta_a}\in\mathcal{O}_{a,a-1/2}$ and an error term $S'_{>A}$ satisfying the Sobolev bounds $\|S'_{>A}\|_{H^{N_2-A-2}}\lesssim \epsilon_\infty^A B_2$. The error bounds are obtained by applying Lemma \ref{level2T} (ii) to the terms of homogeneity  $\geq A+1$ in the right-hand side of \eqref{adi49}, and using the bounds $\|\mathcal{A}_l(u)\|_{H^{N_2-l}}\lesssim \epsilon_\infty^lB_2$ and $\|\mathcal{A}_l(u)\|_{\dot{W}^{N_\infty-l,0}}\lesssim \epsilon_\infty^{l+1}$, $l\in\{2,3\}$. 

We can now apply the formula \eqref{adi49.2} recursively to the terms in the right-hand side of \eqref{adi48} to conclude that
\begin{equation}\label{adi49.3}
\partial_tw+i|\partial_x|^{1/2}w=\sum_{n=3}^A\sum_{\iota_1,\ldots,\iota_n\in\{+,-\}}N_{n,\iota_1\ldots\iota_n}(u^{\iota_1},\ldots, u^{\iota_n})+N_{>A},
\end{equation}
for some operators $N_{n,\iota_1\ldots\iota_n}\in\mathcal{O}_{n,n-1/2}$ defined by homogeneous symbols of degree $n-1/2$ and an error term $N_{>A}$ satisfying the bounds $\|W_{>A}\|_{\dot{H}^{N_2-A-2}}\lesssim \epsilon_\infty^AB_2$. This suffices to prove all the claims concerning the nonlinearities $\mathcal{N}_n(w)$ and $\mathcal{N}_{>A}$ in Proposition \ref{normalprop}, with the exception of the special claims concerning the cubic nonlinearity $\mathcal{N}_3$.

{\bf{Step 4.}} We start from the identity \eqref{adi31.5}. The argument above shows that the cubic nonli\-ne\-arity $\mathcal{N}_3(w)$ is defined by
\begin{equation}\label{adi50}
\begin{split}
&\widehat{\mathcal{N}_3(w)}(\xi)=\frac{1}{(2\pi R)^2}\sum_{\xi_1,\xi_2,\xi_3\in\Zb/R,\,\xi_1+\xi_2+\xi_3=\xi}n_{3,++-}(\xi_1,\xi_2,\xi_3)\widehat{w}(\xi_1)\widehat{w}(\xi_2)\widehat{\overline{w}}(\xi_3),\\
&n_{3,++-}:=p_{3,++-}\mathbf{1}_{{}^c\mathcal{R}}(1-\Upsilon).
\end{split}
\end{equation}
Clearly the cubic symbols $n_{3,\iota_1\iota_2\iota_3}$ vanish unless $(\iota_1\iota_2\iota_3)=(++-)$, and $n_{3,++-}$ is purely imaginary due to the formulas \eqref{adi13.9} and \eqref{adi18.1}--\eqref{adi18.8}.

To prove the $C^{1/2}$ bounds \eqref{symbolcubic} we examine the formulas \eqref{adi18.57} and the formulas in the second line of \eqref{adi18.55}. Recalling also the definitions of the domain $\mathcal{R}$ and the function $\Upsilon$ in \eqref{adi31}, it suffices to prove that if $b:\Rb^3\to\Rb$ is the function defined by
\begin{equation}\label{adi51}
\begin{split}
b(\eta,\rho,&\sigma):=\mathbf{1}_+(\eta)\mathbf{1}_+(\rho)\mathbf{1}_{-}(\sigma)\mathbf{1}_+(\eta+\rho+\sigma)(1-\Upsilon)(\eta,\rho,\sigma)\\
&\times\big\{|\eta+\rho+\sigma||\sigma|^{1/2}\big[|\sigma+\eta|+|\sigma+\rho|-|\sigma|\big]-3|\eta+\rho+\sigma|^{3/2}|\eta|^{1/2}|\rho|^{1/2}\big\},
\end{split}
\end{equation}
where $\mathbf{1}_+$ and $\mathbf{1}_-$ denote the characteristic functions of the intervals $[0,\infty)$ and $(-\infty,0]$, then
\begin{equation}\label{adi52}
|b(x)-b(y)|\lesssim |x-y|^{1/2}(|x|^2+|y|^2)\qquad\text{ for any }x,y\in\Rb^3.
\end{equation}

The definition of the function $\Upsilon$ shows that $1-\Upsilon$ vanishes if the smallest of the variables $|\eta|,|\rho|,|\sigma|,|\eta+\rho+\sigma|$ is sufficiently smaller than the second smallest of these variables. In other words, the smallest and the second smallest of these variables are proportional in the support of $1-\Upsilon$. The definition \eqref{adi51} then shows that
\begin{equation}\label{adi55}
|b(\eta,\rho,\sigma)|\lesssim \mathbf{1}_+(\eta)\mathbf{1}_+(\rho)\mathbf{1}_{-}(\sigma)\mathbf{1}_+(\eta+\rho+\sigma)(|\eta|^2+|\rho|^2+|\sigma|^2)\min(|\eta|,|\rho|, |\sigma|, |\eta+\rho+\sigma|)^{1/2}
\end{equation}
for any $\eta,\rho,\sigma\in\Rb$. In particular, the function $b$ is continuous on $\Rb^3$. 

To prove the $C^{1/2}$ bounds \eqref{adi52} we may assume by homogeneity that $|x|=1$ and $|y|\leq 2$. Letting $y=x+\delta e$ it suffices to prove that 
\begin{equation}\label{adi56}
|b(x+\delta e)-b(x)|\lesssim \delta^{1/2}\qquad\text{ for any }x,e\in\mathbb{S}^2\text{ and }\delta\in [0,4].
\end{equation}
In proving \eqref{adi56} we may assume that $\delta$ is small, $\delta\leq 1/10$. The bounds follow from \eqref{adi55} if $x=(\eta,\rho,\sigma)\in\mathbb{S}^2$ has the property that $\min(|\eta|, |\rho|, |\sigma|, |\eta+\rho+\sigma|)\leq 2\delta$. On the other hand, if $\eta,\rho,\eta+\rho+\sigma\geq 2\delta$, $\eta\leq-2\delta$, and $|x|=|\eta,\rho,\sigma|=1$ then the definition \eqref{adi51} shows that the function $b$ satisfies the stronger Lipschitz bounds
\begin{equation*}
|b(x+\delta e)-b(x)|\lesssim \delta\min(|\eta|,|\rho|, |\sigma|, |\eta+\rho+\sigma|)^{-1/2}.
\end{equation*}
This suffices to prove \eqref{adi56} in all cases, which completes the proof of Proposition \ref{normalprop}.

\section{A set of divergent diagrams}\label{SecDiv}
In this section we present a set of couples that produce divergent contributions,
\emph{in the case when $T_1\sim \epsilon^{-(8/3)-}$.} This demonstrates 
the optimality of our choice $T_1\sim \epsilon^{-(8/3)+}$ in Theorem \ref{main}, 
(at least for our current arguments).
In this section, we assume $T_1=\epsilon^{-\alpha'}$ with some fixed $\alpha'>(8/3)$.

\medskip
\textbf{Scenario}. Consider the following scenario in a tree where 
each branching node is an I-branching node and has exactly 3 children nodes. Assume $\nf_{j+1}$ is a child node of $\nf_j$ (where $0\leq j\leq 4q$), and the other children nodes $(\mf_j,\pf_j)$ of $\nf_j$ are all leaves. Assume each $\nf_j$ has sign $+$ and $(\mf_j,\pf_j)$ has sign $(-,+)$. Moreover, assuming they have the following pairing structure:
\begin{itemize}
\item $\mf_{2j}$ is paired with $\pf_{2j+2}$ for each $0\leq j\leq 2q-1$;
\item $(\mf_{2j-1},\pf_{2j-1})$ is paired with $(\pf_{2\pi(j)-1},\mf_{2\pi(j)-1})$ 
for an involution $\pi$ from $\{1,\cdots,2q\}$ to itself.
\end{itemize} An example with $q=2$ is shown in Figure \ref{fig:counter}, where $\pi =(3,4,1,2)$.
  \begin{figure}[h!]
  \includegraphics[scale=0.20]{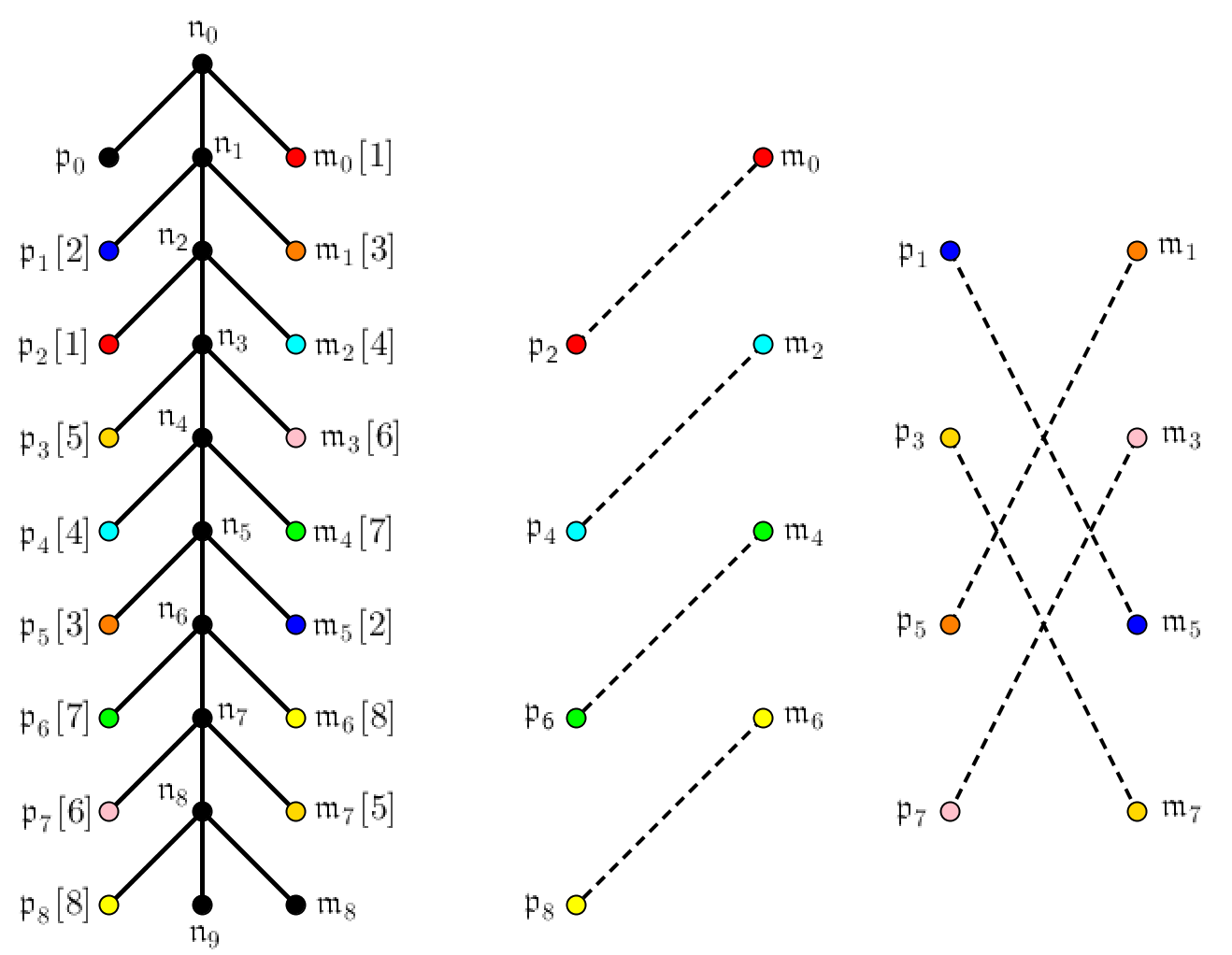}
  \caption{An example of a problematic scenario with $q=2$ and $\pi=(3,4,1,2)$. Here the leaves of the same color are paired (with the bracket after the leaf labels representing the number of its pair); for illustration, we also draw separate pictures that include the $\mf_j$ and $\pf_j$ for even and odd $j$ separately.}
  \label{fig:counter}
\end{figure}

Regarding the above example, we (heuristically) demonstrate the following points:
\smallskip

\begin{enumerate}
\item 
For an individual structure $\Rc$ as above that is contained in a couple $\Qc$, 
if its size $q$ is large enough, then the corresponding sub-expression in $\Kc_\Qc$
involving this structure $\Rc$ diverges in the limit $\epsilon\to 0$;

\smallskip
\item This divergence is a \emph{pure 1D phenomenon}; 

\smallskip
\item There does not seem to be the same cancellation between different choices 
of $\Rc$, similar to the one involving the irregular chains in Section \ref{sec.irre}.

\end{enumerate}

\medskip
{\it About (1)}:
Consider a decoration of $\Rc$ where $k_{\nf_j}=k_j$ and $(k_{\mf_j},k_{\pf_j})=(\ell_j,p_j)$; 
for simplicity we will ignore the $\varphi_{\leq K_{\mathrm{tr}}}$ cutoffs, the phase corrections $\Gamma_0$ and $\Gamma_1$, 
and the symbol $n_3$ from the equation.
Now, from the definition (\ref{defofkq}) for $\Kc_\Qc$, 
by arguments similar to the irregular chain in Section \ref{sec.irre}, we can extract its part that involves $\Rc$ as 
\begin{equation}\label{eq.exp_counter}
\widetilde{\Gc}:=\sum_{(k_j,\ell_j,p_j)}^{(*)}\prod_{j=0}^{4q-1}
	\psi(\ell_j) \psi(p_{j+1})\cdot\int_{t'>t_1>\cdots >t_{4q}>t''}\prod_{j=0}^{4q}
	e^{iT_1\Omega_jt_j}\,\mathrm{d}t_1\cdots\mathrm{d}t_{4q}.
\end{equation} Here in (\ref{eq.exp_counter}):
\begin{itemize}
\item The summation is taken over all $(k_j:1\leq j\leq 4q)$, $(m_j:0\leq j\leq 4q-1)$ and $(p_j:1\leq j\leq 4q)$, with the values of $(k_0,p_0,k_{4q+1},\ell_{4q})=(k_{\nf_0},k_{\pf_0},k_{\nf_{4q+1}},k_{\mf_{4q}})$ being fixed, and under the pairing assumptions $\ell_{2j}=p_{2j+2}\,(0\leq j\leq 2q-1)$ and $\ell_{2j-1}=p_{2\pi(j)-1}\,(1\leq j\leq 2q)$.
\item In the integral in (\ref{eq.exp_counter}), we have $t_j=t_{\nf_j}\,(1\leq j\leq 4q)$ and $(t',t'')=(t_{\nf_0},t_{\nf_{4q+1}})$, and 
\begin{equation}\label{eq.defomegacounter}\Omega_j:=|k_j|^{1/2}-|k_{j+1}|^{1/2}+|\ell_j|^{1/2}-|p_j|^{1/2}.\end{equation}
\end{itemize}

To estimate $\widetilde{\Gc}$, we restrict the decoration to
\begin{align*}
|\ell_j-p_j|=|k_j-k_{j+1}|\leq T_1^{-1} \qquad \mbox{for each \emph{odd} $j$,}
\end{align*}
and $|k_0-p_0|\leq T_1^{-1}$. Note that by pairing,
\begin{align*}
|k_2-p_2| = |k_2-m_0|\leq|k_2-k_1|+|k_1-m_0|=|\ell_1-p_1|+|k_0-p_0|\lesssim T_1^{-1},
\end{align*}
and, similarly,
\begin{align*}
|k_j-p_j|=|\ell_j-k_{j+1}|\lesssim T_1^{-1} \qquad \mbox{for each \emph{even} $j$}
\end{align*}
by induction. 
By \eqref{eq.defomegacounter}, for \emph{each} $j$ (even or odd), 
we always have $|\Omega_j|\lesssim T_1^{-1}$, so there is no oscillation in the time integral in \eqref{eq.exp_counter}. 
Therefore, ignoring the constants that depend only on $q$, 
we see the (absolute value of the) time integral in \eqref{eq.exp_counter} is bounded below by $O(1)$.

It then suffices to estimate the number of choices for the decoration $(k_j,\ell_j,p_j)$. For this, we first fix the values of $(\ell_j,p_j)$ for each \emph{odd} $j$; by pairing and the assumption $|\ell_j-p_j|\leq T_1^{-1}$, we get $(R^2T_1^{-1})^q$ choices. Once they are fixed, we then only need to fix the values of $k_j$ for even $j$ (as they uniquely determine the values of $k_j$ for odd $j$ and thus all the remaining decorations), which has $R^{2q}$ choices. This leads to the estimate
\begin{equation}\label{eq.lower_bd}
|\widetilde{\Gc}|\gtrsim (R^4T_1^{-1})^q.
\end{equation}

Finally, if we take into account the pre-factors in (\ref{defofkq}) (that involve the structure $\Rc$), which is
\begin{equation}\label{eq.power_counter}
(\epsilon R^{-1/2})^{8q}T_1^{4q}
\end{equation} 
(because the number $r(\Rc)$ of leaves in $\Rc$ equals $8q$, 
while the number $n_I(\Rc)$ of I-branching nodes in $\Rc$ equals $4q$), 
we get
\begin{equation}\label{eq.total_counter}
(R^4T_1^{-1})^q\cdot (\epsilon R^{-1/2})^{8q}T_1^{4q}= (T_1^3\epsilon^8)^q=\epsilon^{(8-3\alpha')q}
\end{equation} 
which diverges as $\epsilon\to 0$, assuming $\alpha'>8/3$.

\medskip
About (2):
Suppose we work in dimension $d\geq 2$ instead of $d=1$, then the bound \eqref{eq.lower_bd} should be replaced by
\begin{equation}\label{eq.lower_bd_1}
|\widetilde{\Gc}|\gtrsim (R^{4d}T_1^{-d})^q,
\end{equation} and the pre-factor in (\ref{eq.power_counter}) should be replaced by
\begin{equation}\label{eq.power_counter_2}
(\epsilon R^{-d/2})^{8q}T_1^{4q},
\end{equation} hence instead of (\ref{eq.total_counter}) we have
\begin{equation}\label{eq.total_counter_1}
(R^{4d}T_1^{-d})^q\cdot (\epsilon R^{-d/2})^{8q}T_1^{4q}= (T_1^{4-d}\epsilon^8)^q,
\end{equation} 
which is bounded as long as $d\geq 2$ and $T_1\leq T_{\mathrm{kin}}:=\epsilon^{-4}$
(which allows us to cover the conjectured kinetic time).

\medskip
About (3):
A key feature of the cancellation in Section \ref{sec.irre} is the following:
suppose two irregular chains $\Hc$ and $\Hc'$ are twists of each other, 
then there exists a bijection between decorations of $\Hc$ and decorations of $\Hc'$ 
that preserves the \emph{exact expressions of $\Omega_\nf$ for each branching node $\nf$}. 
This relies on the fact that the molecules for $\Hc$ and $\Hc'$ are 
\emph{exactly the same up to directions of all the bonds}. 
However, by examining the structure of the molecules corresponding to $\Rc$, 
one can see that there does not exist such a correspondence between different choices of $\Rc$. 
Therefore it is unlikely the terms for different $\Rc$ can exhibit cancellations 
similar to those already observed in the literature.

For example, consider the (partial) molecule $\Mb$ corresponding to the structure in
Figure \ref{fig:counter}, which has the shape as in Figure \ref{fig:count_mol}.
  \begin{figure}[h!]
  \includegraphics[scale=0.25]{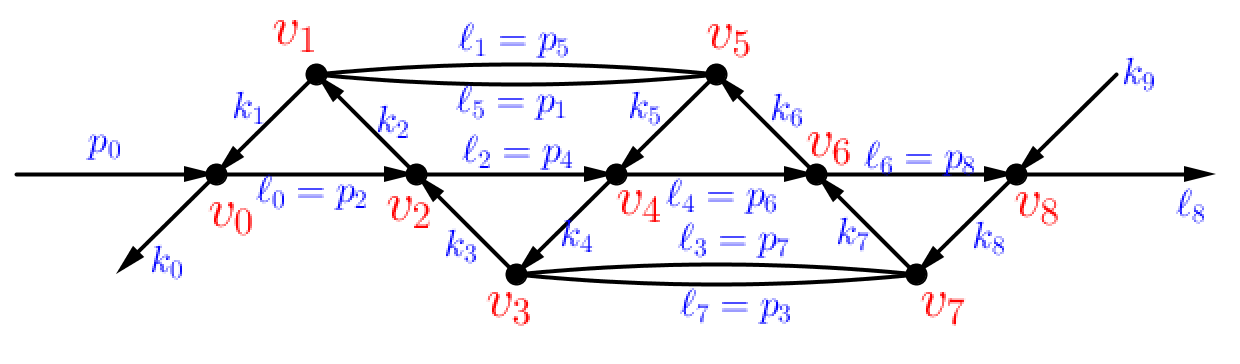}
  \caption{The molecule $\Mb$ corresponding to the structure $\Rc$ in Figure \ref{fig:counter}. 
  Here each atom $v_j$ corresponds to the I-branching node $\nf_j$, 
  with each edge marked by the corresponding vector in the decoration. 
  Each of the two double bonds is understood as consisting of two bonds of opposite direction.}
  \label{fig:count_mol}
\end{figure}

Now, suppose another structure $\Rc'$ corresponds to the same molecule $\Mb$ in Figure \ref{fig:count_mol}, 
with the same directions of bonds etc. 
Assume $\Rc'$ has the same structure as $\Rc$ in Figure \ref{fig:counter}, 
where the atoms $v_j$ in $\Mb$ correspond to I-branching nodes $\nf_j$ in $\Rc$ and $\nf_j'$ in $\Rc'$,
where (say) $\nf_0$ and $\nf_0'$ both have sign $+$. 
Then, since the bond between $v_0$ and $v_1$ goes from $v_1$ to $v_0$, 
by Definition \ref{defmol}, we know that both $\nf_1$ and $\nf_1'$ must have sign $+$. 
In the same way, all nodes $\nf_j$ and $\nf_j'$ 
must have sign $+$\footnote{In contrast, for two congruence irregular chains, 
while $\nf_0$ must have the same sign, the subsequent $\nf_j$ may not have the same sign.}, 
so by symmetry we may assume $\nf_j$ is the second child node of $\nf_{j-1}$, 
and $\nf_j'$ is the second child of $\nf_{j-1}'$. 
By considering the other bonds in $\Mb$ and proceeding in this way, 
we see that $\Rc$ must be equal to $\Rc'$ (up to symmetry) if they correspond to the same molecule,
thus they give rise to the same terms that do not cancel.


\end{document}